\documentclass{amsart}
\usepackage{verbatim}
\usepackage{epsfig}
\usepackage{graphicx}
\usepackage{mathrsfs}
\usepackage{amsmath, amssymb, amsfonts, euscript, enumerate}
\usepackage{txfonts}
\usepackage{color}

\newcommand{\R}{{\mathbb R}}
\newcommand{\N}{{\mathbb N}}

\newcommand{\cA}{{\mathcal A}}
\newcommand{\cH}{{\mathcal H}}
\newcommand{\cD}{{\mathcal D}}
\newcommand{\cP}{{\mathcal P}}
\newcommand{\cM}{{\mathcal M}}
\newcommand{\cS}{{\mathcal S}}
\newcommand{\cJ}{{\mathcal J}}
\newcommand{\sL}{{\mathscr L}}
\newcommand{\sH}{{\mathscr H}}
\newcommand{\sS}{{\mathscr S}}

\newcommand{\e}{\epsilon}
\newcommand{\ve}{\varepsilon}
\newcommand{\al}{\alpha}
\newcommand{\be}{\beta}
\newcommand{\p}{\partial}
\newcommand{\vp}{\varphi}

\newcommand{\supp}{\operatorname{supp}}
\newcommand{\Cut}{\operatorname{Cut}}
\newcommand{\Jac}{\operatorname{Jac}}
\newcommand{\tr}{\operatorname{tr}}
\newcommand{\Ric}{\operatorname{Ric}}
\newcommand{\Sec}{\operatorname{Sec}}
\newcommand{\Sym}{\operatorname{Sym}}

\newcommand{\D}{\nabla}
\newcommand{\ra}{\rightarrow}
\newcommand{\La}{\Delta}
\newcommand{\bs}{\backslash}

\newcommand{\vol}{\operatorname{Vol}}
\newcommand{\diam}{\operatorname{diam}}

\newtheorem{thm}{Theorem}[section]
\newtheorem{lemma}[thm]{Lemma}

\newtheorem{remark}[thm]{Remark}
\newtheorem{prop}[thm]{Proposition}
\newtheorem{definition}[thm]{Definition}

\newtheorem*{notation}{Notation}

\theoremstyle{definition}

\begin{document}
\title[Harnack inequality on Riemannian manifolds]{Parabolic Harnack inequality  of viscosity solutions  on Riemannian manifolds}

\author[Soojung Kim]{Soojung Kim}
\address{Soojung Kim :
National Institue for Mathematical Sciences,
70 Yuseong-daero, 1689 beon-gil, Yuseong-gu, Daejeon, 306-390, Republic of Korea}
\email{soojung26@nims.re.kr,soojung26@gmail.com }

\author[Ki-Ahm Lee]{Ki-Ahm Lee}
\address{Ki-Ahm Lee:
 School of Mathematical Sciences, Seoul National University, 1 Gwanak-ro, Gwanak-gu, Seoul 151-747, Republic of Korea \& Center for Mathematical Challenges, Korea Institute for Advanced Study, Seoul,130-722, Korea}
\email{kiahm@math.snu.ac.kr}

%\keywords{}
%\subjclass{Primary 35K55, 35K65}

\begin{abstract}
We consider viscosity solutions to   nonlinear uniformly parabolic equations  in nondivergence form  on a  Riemannian manifold $M$,  with the sectional curvature bounded from below by $-\kappa$  
for $\kappa\geq 0$.
In the elliptic case, Wang and Zhang \cite{WZ} recently   extended the results of \cite{Ca}   to nonlinear elliptic equations  in nondivergence form on such $M$, where they obtained the   Harnack inequality for  classical solutions.   We  establish the  Harnack inequality for nonnegative {\it viscosity solutions} to nonlinear uniformly {\it  parabolic equations} in nondivergence form on $M$.
 The Harnack inequality of nonnegative viscosity solutions to the elliptic equations is also proved. %in a manifold with  Ricci curvature bounded from below.
\end{abstract}
%{\bf Contents:}\\
 
\maketitle

\tableofcontents 
\newpage

%%%%%%%%%%%%%%%%%%%%%%%%%%%%%%%%%%%%%%%%%%%%%%%%%%%%%%%%%%%%%%%%%%%%
\section{Introduction and main results}%\label{sec-intro}
In this paper, we study the  Harnack inequality  of viscosity solutions to  nonlinear uniformly parabolic equations in nondivergence form on Riemannian manifolds.
Let $(M,g)$ be a smooth, complete Riemannian manifold of dimension $n$.  Consider  a nonlinear uniformly parabolic equation
\begin{equation}\label{eq-main}
F(D^2u) - \p_t u=f \quad\mbox{in}\,\,\,M\times \R,
\end{equation}
where $D^2 u$ denotes the Hessian of the function $u$ defined by
\[
D^2u  \,  (X, Y)= g \left(\nabla_X \nabla u, Y\right),
\]
 for any vector fields $X, Y$ on $M,$ and  $\nabla u$   is the gradient of $u.$  
We notice that in the case, when $F$ is the trace operator,    \eqref{eq-main} is the well-known heat equation with a source term.

In the  setting of elliptic equations on   $M,$    Cabr\'e \cite{Ca}  established  the   Krylov-Safonov type  Harnack inequality of classical solutions to linear, uniformly elliptic equations in nondivergence form, when  $M$ has nonnegative sectional curvature. 
The Krylov-Safonov Harnack inequality is  based on the Aleksandrov-Bakelman-Pucci (ABP) estimate,   which is proved using affine functions in the Euclidean case.  Since  affine functions  can not be generalized into an intrinsic notion on Riemannian  manifolds,   Cabr\'e   considered    the functions of   the  squared distance  instead of the affine functions     to overcome the difficulty.
% Cabr\'e \cite{Ca}, proved that if the underlying manifold $M$ has nonnegative sectional curvature, then Krylov-Safonov type  Harnack inequality holds for $C^2$-solutions to linear, uniformly elliptic equations in non-divergence form. One crucial ingredient in proving Krylov-Safonov Harnack inequality  is  the  {\color{red}  ABP} estimate   which is proved using affine functions, which have no intrinsic interpretation in  general Riemannian manifolds.  Cabr\'e ingeniously overcame this difficulty by replacing the affine functions by   the square of distance functions.
 Later, Kim \cite{K} improved Cabr\'e's result  removing the sectional curvature assumption and     imposing the certain condition on the distance function  (see \cite[p. 283]{K}). 
 \begin{comment}
 A  linear uniformly elliptic  operator  considered   in \cite{Ca,K} 
is defined as  follows:% for $S\in\Sym TM,$ 
$$\sL u:=\,\mathrm{trace}\,(A_x\circ D^2u)=\,\mathrm{trace}\,(  D^2u\circ A_x)%\,\mathrm{trace}\,\left\{ X\mapsto A_x  \D_X\D u \right\}
,\quad \forall x\in M,$$ where 
  $A_x $  is  a positive definite  self-adjoint operator  on $T_xM$  satisfying
 \begin{equation*}%\label{cond-lin-ellip} 
\lambda \abs{X}^2 \le g\left({A_{x}\, X,X}\right) \le \Lambda |X|^2,\quad\forall x\in M, \,\,\,\forall X\in T_x M,
\end{equation*}    and 
  $D^2u$ is considered as a self-adjoint operator  on $T_x M$ by a canonical identification of   
  the space of self-adjoint  bilinear forms on $T_x M$with the space of self-adjoint linear operator from $T_xM$ into itself. 
  \end{comment}
Recently, Wang and Zhang \cite{WZ} obtained   a version of the  ABP  estimate on $M$ with Ricci curvature bounded from below,  and the 
  Harnack inequality  of   classical solutions for  nonlinear uniformly elliptic operators provided that  $M$  has    a lower bound of the sectional curvature. %,  where $F$ satisfies \eqref{Hypo1} below.

In the parabolic case, 
 the Krylov-Safonov Harnack inequality 
 was proved in \cite{KKL} 
 for  classical solutions to    linear, uniformly parabolic equations in nondivergence form, assuming essentially the same condition introduced by Kim \cite{K}.  The result in \cite{KKL}, in particular, gives a non-divergent proof of Li-Yau's  Harnack inequality for   the heat equation in a manifold with nonnegative Ricci curvature \cite{LY}.
 \begin{comment}
 , which reads  as follows: Let $d_y(x)=d(x,y)$ be the geodesic distance  between $x$ and $y$, and let $\Cut(y)$ denote the cut locus of $y$.
For all $y\in M$,  it is assumed that 
\begin{equation*}\label{cond-M-1}
\La d_y(x) \leq \frac{n-1}{d_y(x)}\quad\forall x\notin\Cut(y)\cup\{y\}
\end{equation*}
and there exists a constant $a_L$ such that  
\begin{equation*}
\label{cond-M-2}
\sL d_y(x) \leq \frac{a_{L}}{d_y(x)}\quad\forall x\notin\Cut(y)\cup\{y\}.%\;\mbox{ and }\; t\in \bR,
\end{equation*}
where $\sL$ is a     linear uniformly  elliptic operator in nondivergence form (see \cite{K,Ca}).  
\end{comment}
The ABP-Krylov-Tso  estimate  discovered by Krylov  \cite{Kr} in  the Euclidean case (see also \cite{T,W})  is a parabolic analogue of  the ABP estimate, and  a key ingredient in   proving the   parabolic Harnack inequality.  
In order to prove the ABP-Krylov-Tso type estimate  on Riemannian manifolds
, an intrinsically geometric version of the Krylov-Tso normal map, namely,
\[
\Phi(x,t):=\left(\exp_x \nabla_x u(x,t), -\frac{1}{2}d^2\left(x,\exp_x \nabla u(x,t)\right) - u(x,t)\right)
\]
was introduced. % making use of   the  squared distance function.
The map $\Phi$ is called  the parabolic normal  map related to $u(x,t)$ and the Jacobian determinant of $\Phi$ was explicitly computed in \cite[Lemma 3.1]{KKL}.

In this paper, we shall  prove  the Krylov-Safonov Harnack inequality for a class of  viscosity solutions to uniformly parabolic equations %  \eqref{eq-main}
% nonlinear uniformly parabolic equations in nondivergence form:
%\begin{equation}\label{eq-main}
%F(D^2u)-\p_tu=f
%\end{equation} 
 on   $M$ with  the sectional  curvature bounded from below. 
  Let  $\Sym TM$ be the bundle of symmetric 2-tensors over $M.$ A nonlinear operator  $F : \Sym TM\rightarrow \R$ will be always    assumed    in this article to satisfy 
the following basic  hypothesis: 
 \begin{enumerate}
\item[]
   $F$ is uniformly elliptic with the so-called ellipticity constants $0<\lambda\leq\Lambda,$ i.e., for  any $S\in \Sym TM,$ and for any positive semidefinite $P\in\Sym TM,$
 \begin{equation}\label{Hypo1}\tag{H1}
   \lambda\,\mathrm{trace}(P_x)\leq F(S_x+P_x)-F(S_x)\leq \Lambda\, \mathrm{trace}(P_x),\quad\forall x\in M.
 \end{equation}
% We also assume that $F(0)=0.$
 \end{enumerate} 
 We may assume that $0<\lambda\leq 1\leq\Lambda.$  In order to   establish  the  uniform Harnack inequality for a    
class of  uniformly parabolic   equations   including \eqref{eq-main},  we  introduce
% We end this subsection by introducing 
Pucci's extremal operators as  in \cite{CC}: %  follows (see \cite{CC}):
%\begin{definition}[Pucci's extremal operator]\label{def-pucci}
 % For $0<\lambda\leq \Lambda$ (called ellipticity constants), the Pucci's extremal operators are defined as follows: 
 for any $x\in M,$ and  $S_x\in\Sym TM_x,$
\begin{align*}
\cM^+_{\lambda,\Lambda}(S_x):=\cM^+(S_x)=\lambda\sum_{e_i<0}e_i+\Lambda\sum_{e_i>0}e_i,\\
\cM^-_{\lambda,\Lambda}(S_x):=\cM^-(S_x)=\Lambda\sum_{e_i<0}e_i+\lambda\sum_{e_i>0}e_i,
\end{align*}
where $e_i=e_i(S_x)$ are the eigenvalues of $S_x.$
%\end{definition}
%In the special case when $\lambda=\Lambda=1,$  the Pucci's extremal operators $\cM^\pm$ simply coincide with the trace operator.
 In terms of the Pucci operators,   the hypothesis  \eqref{Hypo1} of uniform ellipticity is equivalent to the following: 
 for any $S, P\in\Sym TM,$
  \begin{equation}
  \label{Hypo1'}\tag{H1'}
\cM^-(P_x) \leq F(S_x+P_x)-F(S_x) \leq\cM^+(P_x),\quad\forall x\in M.
 \end{equation}

 Now we recall  viscosity solutions,   which are   proper   weak solutions for  nonlinear  equations in nondivergence form.  In the Euclidean space, the existence, uniqueness    and    regularity theory for the viscosity solutions have been developed by many authors (see for instance,  \cite{CIL, CC, W}).
In  \cite{AFS,Z}, the concept of viscosity solutions  has been naturally extended  on Riemannian manifolds, which can be found in Definitions \ref{def-visc-sol-p} and  \ref{def-visc-sol}. The authors in   \cite{AFS,PZ,Z}  have shown  comparison, uniqueness and existence results for the viscosity solutions on Riemannian manifolds. % Using   the Pucci extremal operators, 
Using  Pucci's extremal operators, 
  we introduce  a      class  of viscosity solutions to the uniformly parabolic equations;  see \cite{CC}. % including \eqref{eq-main}.    
 \begin{definition}%\label{def-class-solutions-para}
%Let $f$ be  a continuous function in an open set 
Let $\Omega\subset M$  be open,   $T>0,$ and  
  $0<\lambda\leq \Lambda. $  We denote by    $
\overline \cS_P\left(\lambda,\Lambda,f\right)$ a class of a viscosity  supersolution    $u\in C(\Omega\times(0,T])$  satisfying
$$\cM^-(D^2u)-\p_tu\leq f\quad\mbox{in $\Omega\times(0,T]$} $$  in the viscosity sense. 
Similarly,  a class of viscosity subsolutions $\underline\cS_P\left(\lambda,\Lambda,f\right)$ is defined as the set of     $u\in C(\Omega\times(0,T])$ such that 
$$\cM^+(D^2u)-\p_t u\geq f\quad\mbox{in $\Omega\times(0,T]$} $$ in the viscosity sense.  
We   denote $$\cS_P^*\left(\lambda,\Lambda,f\right):=\overline \cS_P\left(\lambda,\Lambda,|f|\right)\cap \underline \cS_P\left(\lambda,\Lambda,-|f|\right).$$ 
Simply, we write $\overline \cS_P(f), $ $\underline \cS_P(f), $ and $  \cS_P^*(f)$ for  $\overline \cS_P\left(\lambda,\Lambda,f\right),$ $\underline \cS_P\left(\lambda,\Lambda,f\right),  $ and $\cS_P^*\left(\lambda,\Lambda,f\right),$   
 respectively. 
\end{definition}  
Note that  the  viscosity  solution to the fully nonlinear uniformly  parabolic equation \eqref{eq-main}   belongs to the class $\cS_P^*\left(\lambda,\Lambda, f-F(0,\cdot)\right)$ owing to the equivalence between \eqref{Hypo1} and   \eqref{Hypo1'}. %  below.

To obtain the  Harnack inequality of viscosity solutions contained  in $\cS_P^*$ from a priori Harnack estimates (Subsection \ref{subsec-PHI-C2}, \cite{KKL}, and \cite{Ca,K,WZ}), we use  regularization by sup and inf-convolutions, introduced by Jensen \cite{J}. 
The classical ABP estimate  for viscosity solutions   was proved by making use of affine functions, especially 
 the convex envelope of the  viscosity solution (see \cite{CC,W}). % denoted by $\Gamma(u) ,$
%, the convex envelope of $u$ in $\Omega,$ 
%denoted by $\Gamma(u) ,$  % which  is the largest convex function bounded above by $u, $ on $\Omega, $%has a supporting hyperplane from below at each  $x\in\Omega, $ 
%$$\Gamma(u)(x):=\sup\left\{ l(x) : l(x)\leq u(x),  \,l, \mbox{ an affine function} \right\},\qquad\forall x\in\Omega,$$
%was used .  
%It has been shown that the convex envelope $\Gamma(u)$  for a viscosity supersolution $u$ becomes $C^{1,1},$ so a priori ABP estimate can apply to $\Gamma(u).$ 
%As mentioned above,   %this approach  does not make sense on general manifolds so 
%we  deal with the squared distance function  on Riemannian manifolds   instead of affine functions,  
Replacing affine functions by the squared distance functions  on $M,$ as mentioned above,  we    consider  the sup and inf-convolutions on Riemanninan manifolds defined as follows: %     which  are   defined  in  \eqref{eq-def-inf-conv-p} and \eqref{eq-def-sup-conv-p}.  For example, 
 for $\ve>0,$ let $u_\ve$ denote  the inf-convolution of $u$,     defined as 
  \begin{equation*}%\label{eq-def-inf-conv}
u_{\ve}(x):=\inf_{y\in   \Omega  } \left\{ u(y) +\frac{1}{2\ve}d^2(y, x)\right\}, \qquad\forall x\in\Omega\subset M.
 \end{equation*}
The sup-convolution can be defined in a similar way using concave paraboloids. % in \eqref{eq-def-sup-conv-p}.
We see that the regularized  functions  by the sup and inf-convolutions   are  semi-convex and semi-concave, respectively, % in  Lemma \ref{lem-u-e-alexandrov-p},
 which     imply that they admit the Hessian almost everywhere thanks to  the Aleksandrov theorem  \cite{A, B}.  
     In Lemma \ref{prop-u-u-e-superj-p},   we prove  that  regularized viscosity solutions   solve approximated  equations %in  the same class of the original  equations 
 in the viscosity sense,   provided that  the sectional curvature of $M$ is  bounded from below, and the  operator $F$  is intrinsically uniformly continuous with respect to $x$; see  Definition \ref{def-in-unif-cont}. % Intrinsic uniform continuity of  the operator $F$ (with respect to $x$) means that $F(S_x)$    
% As in the Euclidean case,  we shall  
%Since we consider  a class of all uniformly parabolic/elliptic operators, which can be expressed  in terms of  Pucci's extremal operators, %; see Definition \ref{def-pucci} and the comment after it.
  Intrinsic uniform continuity of  Pucci's operators  is a sufficient condition   for obtaining  the   uniform Harnack  estimates   for viscosity solutions % belonging to $\cS^*_P$     
  since   a class $\cS^*_P$ of all viscosity solutions to the uniformly parabolic equations % for uniformly parabolic equations
    is  invariant under the regularization processes of  sup and inf-convolutions.  Then an application of a priori estimates to the sup and inf-convolutions of   viscosity solutions will yield the uniform   Harnack inequality  for    viscosity solutions.

On the other hand,  assuming    the sectional curvature of $M$ to be   bounded from  below, we  establish   a priori   Harnack inequality  for nonlinear  parabolic operators    in Section \ref{subsec-PHI-C2} influenced by  Wang and Zhang \cite{WZ}, who studied the elliptic case. We introduce   the parabolic contact set $\cA_{a,b}$ for $a,b>0$    in Definition \ref{def-contact-p},  which consists of  a point $(\overline x,\overline  t)\in M\times\R,$ where
a concave paraboloid $$\displaystyle-\frac{a}{2}d_y^2(x) + bt + C \,\,\,\mbox{(for some $C$)}$$ touches $u$ from below at $(\overline x,\overline t)$  in a parabolic neighborhood of $(\overline x,\overline t),$ i.e,   $B_r(\overline x)\times(\overline t-r^2,\overline t]$ for some $r>0.$  Under the assumption that  Ricci  curvature of $M$ is  bounded from  below,    an estimation of   the  Jacobian  of the parabolic normal map on the parabolic contact set $\cA_{a,b}$ is  obtained in Lemma \ref{lem-Jac-normal_map}, which 
   is essential   for  proving   the ABP-Krylov-Tso type estimate. %Harnack estimates.  
   For the heat equation on manifolds  with  a lower bound of the Ricci curvature,   we  can  use  Lemma \ref{lem-Jac-normal_map} and   Bishop-Gromov's  volume comparison theorem  to  deduce the Harnack inequality with help of the Laplacian comparison theorem. % replacing the Hessian comparison (Lemma \ref{lem-hess-dist-sqrd}). 
     In particular, this implies    a global Harnack inequality for heat equation on manifolds with nonnegative Ricci curvature  proved   earlier  by Li and Yau    \cite{LY};  see     Remarks \ref{rmk-li-yau} and \ref{rmk-Harnack-const}.  
    %If we consider the heat equation, we only deal with  a lower bound of the Ricci curvature. 
  %   on  manifolds with nonnegative Ricci curvature,    a lower bound of the sectional curvature is not necessary and  we 
 % With help of the Laplacian comparison theorem replacing the Hessian comparison (Lemma \ref{lem-hess-dist-sqrd}),  we    use  Lemma \ref{lem-Jac-normal_map} to  deduce the Harnack inequality  for the hear equation;  in particular, this implies    a global Harnack inequality for heat equation on manifolds with nonnegative Ricci curvature  (see     \cite{LY, KKL}).
 Regarding a class of nonlinear   operators, we  establish the (locally) uniform  Harnack inequality for uniformly parabolic operators   provided that the sectional curvature of the underlying manifold is bounded from below, where our computation  does not  rely   on the linearity of the operator as in \cite{WZ}. 

Now  we state our main results as follows. In the statements and hereafter, we denote
\[
\fint_Q f:=\frac{1}{|Q|} \int_Q f
\]
and
\[
K_r(x_0,t_0):=B_r(x_0)\times (t_0-r^2,t_0],\quad (x_0, t_0) \in M \times \R,
\]
where $|Q|$ stands for the volume of a set  $Q$ of $M$ or $M\times\R,$  and $B_r(x_0)$  is a geodesic ball of radius $r$ centered at $x_0.$

 \begin{thm}[Parabolic Harnack inequality]\label{thm-PHI}
  Assume that $M$ has  sectional curvature bounded from below by $-\kappa$ for  $\kappa\geq0,$ i.e., $\Sec \geq -\kappa$ on $M.$
 %  and $F$   satisfies $\mathrm{(H1)}.$  %\eqref{Hypo1}.  
Let  $0<R\leq R_0,$  and $f\in C\left(K_{2R}(x_0,4R^2)\right).$  If $u\in C\left(K_{2R}(x_0,4R^2)\right)$ is  a nonnegative viscosity solution  in $\cS_P^*\left(\lambda,\Lambda, f\right)$   %of the equation $F(D^2u)-\p_tu=f$
 in $K_{2R}(x_0,4R^2),$
then  we have
\begin{equation}\label{eq-PHI}
\sup_{K_R(x_0,2R^2)}u\leq C_H\left\{ \inf_{K_R(x_0,4R^2)}u+R^2 \left(\fint_{K_{2R}(x_0,4R^2)} |{f}|^{n\theta +1} \right)^{\frac{1}{n\theta+1}}\right\},
\end{equation} 
where $\theta:=1+\log_2\cosh(8\sqrt{\kappa}R_0)$ and  $C_H>0$  is a uniform constant depending only on $n,\lambda,\Lambda$ and $\sqrt{\kappa}R_0$.
%\item[(ii)] If $$\Ric \geq -\kappa,\,\,\,\mbox{on}\,\,\,B_{2R}(x_0)\quad\mbox{and}\quad F(D^2u)=\La u,$$ then \eqref{eq-PHI} holds with the constants $\theta:=1+\log_2\cosh(2\sqrt{\kappa/(n-1)}R)$ and $C>0$   depending only on $n,\lambda,\Lambda$ and $\sqrt{\kappa}R$.
%\end{itemize}
\end{thm}
 
 {\color{black} 
When $\kappa=0,$ \eqref{eq-PHI}  becomes a global   Harnack inequality
which extends the classical Euclidean theory of Krylov and Safonov \cite{KS}. %(see also \cite{Ca, KKL}).
%recovering  the  result  Krylov and Safonov \cite{KS,CC,W} in the Euclidean space.
Assuming the sectional curvature to be bounded from below,   our Harnack estimate are locally uniform, namely, for a fixed $R_0>0,$ we obtain  uniform Harnack inequalities in any balls of  radius $R$ less than $ R_0.$
When  $n,\lambda$ and $\Lambda$ are given,    
the   uniform   constant $C_H=C_H\left(\sqrt{\kappa} R_0\right)$   in our   estimate \eqref{eq-PHI} %and \eqref{eq-WHI}
  grows       faster than  % the exponential function %of $\sqrt{\kappa}R_0$ % 
$\exp\left(1+ {\kappa}R_0^2\right)$ %as $\sqrt{\kappa}R_0\to+\infty$
   % (for a universal constant $c>0$)
as $\sqrt{\kappa }R_0$ tends to infinity;  see Remark \ref{rmk-Harnack-const}. %This bound  of the  Harnack constant    seems to be      rough   when compared to  those for  2-dimensional model spaces  in \cite[Proposition 11.4.]{WZ2}.
}

\begin{thm}[Weak Harnack inequality]\label{thm-weak-PHI}
   Assume that $\Sec \geq -\kappa\,$ on $M$  for  $\kappa\geq0.$
   % and $F$  satisfies $\mathrm{(H1)}.$   
   Let $0<R\leq R_0,$ and  $f\in C\left(K_{2R}(x_0,4R^2)\right).$ If $u\in C\left(K_{2R}(x_0,4R^2)\right)$ is a nonnegative viscosity supersolution    in $\overline\cS_P\left(\lambda,\Lambda, f\right)$ %of the equation $F(D^2u)-\p_t u=f$
   in $K_{2R}(x_0,4R^2),$
then we have
%\begin{itemize}
%\item[(i)]
\begin{equation}\label{eq-WHI}
\left(\fint_{K_{R}(x_0,2R^2)} u^p \right)^{\frac{1}{p}}\leq C_H\left\{ \inf_{K_{R}(x_0,4R^2)}u+ R^2 \left(\fint_{K_{2R}(x_0,4R^2)}  |{f^+}|^{n\theta+1} \right)^{\frac{1}{n\theta+1}}\right\};\quad f^+:=\max(f,0),
\end{equation} 
where $\theta:=1+\log_2\cosh(8\sqrt{\kappa}R_0),$  and  the positive constants $p\in (0,1)$ and $C_H$ are uniform  depending only on $n,\lambda,\Lambda,$ and $\sqrt{\kappa}R_0.$ 
%\item[(ii)]
%\eqref{eq-WHI} holds with $\theta= 1+\log_2\cosh(2\sqrt{\kappa/(n-1)}R)$ and  constants $p\in (0,1)$ and $C$ are uniform constants depending only on $n,\lambda,\Lambda,$ and $\sqrt{\kappa}R.$
%\end{itemize} 
\end{thm}

In the elliptic setting, we  have  Harnack inequalities for  a  class of viscosity solutions  as below; refer to  Definition \ref{def-visc-sol} for the definitions of   classes $\cS^*_E$ and $ \overline \cS_E$ of  viscosity solutions and supersolutions  to uniformly  elliptic equations.  

\begin{thm}[Elliptic Harnack inequality]\label{thm-EHI}
  Assume that %$M$ has  sectional curvatures bounded from below by $-\kappa$ for  $\kappa\geq0,$ i.e., 
$\Sec \geq -\kappa\,$ on $M$  for  $\kappa\geq0.$
   %and $F$   satisfies $\mathrm{(H1)}.$
     Let $0<R\leq R_0,$  and $f\in C\left(B_{2R}(x_0)\right).$ If $u\in C\left(B_{2R}(x_0)\right)$ is a nonnegative viscosity solution  in $\cS_E^*\left(\lambda,\Lambda, f\right)$ %of the equation $F(D^2u)=f$
      in $B_{2R}(x_0),$ then we have
  \begin{equation*}%\label{eq-EHI}
  \sup_{B_R(x_0)}u\leq C\left\{\inf_{B_R(x_0)}u+R^2\left(\fint_{B_{2R}(x_0)}|f|^{n\theta}\right)^{\frac{1}{n\theta}}\right\},
  \end{equation*} 
  where $\theta:=1+\log_2\cosh\left(8\sqrt{\kappa}R_0\right)$ and $C>0$ is a uniform constant depending only  on $n,\lambda,\Lambda,$ and $\sqrt{\kappa}R_0.$ 
  \end{thm}

\begin{thm}[Weak Harnack inequality]\label{thm-weak-EHI}
   Assume that $\Sec \geq -\kappa\,$ on $M$  for  $\kappa\geq0.$
   % and $F$   satisfies $\mathrm{(H1)}.$ 
    Let $0<R\leq R_0,$ and  $f\in C\left(B_{2R}(x_0)\right).$ If $u\in C\left(B_{2R}(x_0)\right)$ is a nonnegative viscosity supersolution  in $\overline\cS_E\left(\lambda,\Lambda, f\right)$  %of the equation $F(D^2u)=f$
     in $B_{2R}(x_0),$
then we have
%\begin{itemize}
%\item[(i)]
\begin{equation*}%\label{eq-WHI}
\left(\fint_{B_{R}(x_0)} u^p \right)^{\frac{1}{p}}\leq C\left\{ \inf_{B_{R}(x_0)}u+ R^2 \left(\fint_{B_{2R}(x_0)} |{f^+}|^{n\theta} \right)^{\frac{1}{n\theta}}\right\};\quad f^+:=\max(f,0),
\end{equation*} 
where $\theta:=1+\log_2\cosh(8\sqrt{\kappa}R_0)$  and  the positive constants $p\in (0,1)$ and $C$ are uniform   depending only on $n,\lambda,\Lambda,$ and $\sqrt{\kappa}R_0.$ 
%\item[(ii)]
%\eqref{eq-WHI} holds with $\theta= 1+\log_2\cosh(2\sqrt{\kappa/(n-1)}R)$ and  constants $p\in (0,1)$ and $C$ are uniform constants depending only on $n,\lambda,\Lambda,$ and $\sqrt{\kappa}R.$
%\end{itemize} 
\end{thm} 
 
The rest of the paper is organized as follows. 
In Section \ref{sec-pre}, we recall  some results on Riemannian geometry and   viscosity solutions that are used in  the paper.
In Section \ref{sec-supinf-conv}, we investigate basic properties of  the sup and inf-convolutions,  and   the relation between  the viscosity solution and  its  sup and inf-convolutions. %,   which is the key part of the paper.
Section \ref{sec-PHI} is devoted to proving the parabolic Harnack inequalities  of viscosity solutions. 
In Section \ref{sec-EHI}, we prove  Harnack inequalities of  viscosity solutions to the elliptic equations.

%%%%%%%%%%%%%%%%%%%%%%%%%%%%%%%%%%%%%%%%%%%%%%%%%%%%%%%%%%%%%%%%%%%%%%
\section{Preliminaries}\label{sec-pre}

\subsection{Riemannian geometry}

Let  $(M, g)$ be a smooth, complete Riemannian manifold of dimension $n$,  where $g$ is the  Riemannian metric  and  $\vol:=\vol_g$ is  the Riemannian measure on $M$. We denote $\langle X,Y\rangle:=g(X,Y)$  and $|X|^2:=\langle X,X\rangle$ for $X,Y\in T_xM$,  where $T_x M$ is the tangent space at $x\in M$. Let $d(\cdot,\cdot)$ be the distance function on $M$.  For a given point $y\in M$,   $d_y(x)$  denotes  the distance function to $y$,  i.e., $d_y(x):=d(x,y)$.

We recall  the exponential map $\exp: TM \to  M$.    If $\gamma_{x,X}:\R\to M$ is the geodesic        starting at $x\in M$ with velocity $X\in T_xM$,  then the exponential map  is defined by
$$\exp_x(X):=\gamma_{x,X}(1). $$
We observe that the geodesic $\gamma_{x,X}$ is defined for all time since $M$ is complete.    
 For $X\in T_xM$ with $|X|=1$,   we define the cut time $t_c(X)$ as
 $$t_c(X) := \sup\left\{ t >0 :  \exp_x(sX) \,\,\,\mbox{is minimizing between}\,\,\, x \,\,\,\mbox{and $\exp_x(t X) $}\right\}.$$
The cut locus of $x\in M,$ denoted by $\Cut(x),$ is defined  by
 $$ \Cut(x):= \left\{\exp_x(t_c(X) X)  : X\in T_x M \,\,\,\mbox{with $|X|=1,\,\, t_c(X)<+\infty$}\right\}. $$
  %   If $t_c(X)<+\infty$,  $\exp_x(t_c(X) X)$ is a cut point of $x$.   The cut locus of $x$ denoted by $\Cut(x)$ is defined as the set of all cut points of $x$,  namely,
%$$ \Cut(x):= \left\{\exp_x(t_c(X) X)  : X\in T_x M \,\,\,\mbox{with $|X|=1,\,\, t_c(X)<+\infty$}\right\}. $$  
If we define 
$$E_x := \left\{t X \in T_xM :  0\leq t<t_c(X),\,\, X\in T_x M \,\,\,\mbox{with $|X|=1$}\right\}\subset T_xM,$$ 
 it can be proved that  
 $\Cut(x)= \exp_x(\p E_x), M=\exp_x(E_x)\cup \Cut(x),$ and $\exp_x : E_x \to \exp_x(E_x)$ is a diffeomorphism.  We note that $\Cut(x)$ is closed and has measure zero.  \textcolor{black}{Given two points $x $ and $y\notin \Cut(x)$, there exists a unique minimizing geodesic $\exp_x(tX)$ (for   $X\in E_x$)  joining  $x$ to $y$ with $y= \exp_x(X),$ and we will write $X= \exp_x^{-1}(y)$.}    
For any $x\notin \Cut(y)\cup\{y\}$, the distance function  $d_y$ is smooth at $x,$ and the Gauss lemma implies that
$$\D d_y(x)=-\frac{\exp_x^{-1} (y)}{|\exp_x^{-1} (y)|},$$
and $$\D( d^2_y/2)(x)=-\exp_x^{-1} (y).$$
The injectivity radius at $x$ of $M$ is defined as 
$$i_M(x):=\sup \{ r>0 : \,\,\mbox{$\exp_x$ is a diffeomorphism from $B_r(0)$ onto $B_r(x)$} \} .$$
We note that
$i_M(x)>0$ for any $x\in M$ and the map $x\mapsto i_M(x)$ is continuous.  %Let $x,y \in M$ be close   such that   $d(x,y)<\min\left\{ i_M(x), i_M(y)\right\},$ then there exists a unique minimizing geodesic joining $x$ to $y.$

We recall the Hessian of a $C^2$- function $u$ on $M$  defined as 
$$D^2u  \,  (X, Y):=  \left\langle\nabla_X \nabla u, Y\right\rangle,$$
 for any vector fields $X, Y$ on $M,$ where $\D$ denotes the Riemannian connection of $M,$ and   $\nabla u$   is the gradient of $u.$  The Hessian $D^2u$ is a  symmetric 2-tensor in $\Sym TM,$  whose value at $x\in M$ depends only on $u$ and the values $X, Y$ at $x.$   
By a canonical identification of   
  the space of symmetric  bilinear forms on $T_x M$ with the space of symmetric endomorphisms of $T_xM,$  
 the Hessian of  $u$ at $x\in M$  can be also  viewed as a  symmetric endomorphism of $T_xM$: $$
D^2u(x)  \cdot X=  \nabla_X \nabla u,\quad\forall X\in T_xM. 
$$
We will   write $D^2u(x)  \,  (X, X)=\left\langle D^2u (x)\cdot X,\, X\right\rangle$ for $X\in T_xM.$ % as a quadratic form on $T_xM.$

Let $\xi$ be a vector field along a differentiable  curve $\gamma:[0,a]\to M.$ We denote  by $\frac{D \xi}{dt}(t)= \D_{\dot \gamma(t)} \xi(t),$   the covariant derivative  of $\xi$ along $\gamma.$ A vector field $\xi$ along $\gamma$ is said to be parallel along $\gamma$ when $$\frac{D \xi}{dt}(t)\equiv0 \quad\mbox{on}\,\,\,[0,a].$$ 
If $\gamma : [0, 1] \to M$ is  a unique minimizing geodesic  joining $x$ to $y,$ then  
for any $\zeta\in T_x M, $ there exists a unique parallel vector field, denoted by $L_{x,y}\zeta(t),$ along $\gamma$ such that $L_{x,y}\zeta(0)=\zeta.$     
The parallel transport of $\zeta$  from $x$ to $y$ , denoted by $L_{x,y}\zeta,$ is defined as   $$ L_{x,y}\zeta:=L_{x,y}\zeta(1)\in T_yM,$$
 which will  induce    a linear isometry  $L_{x,y}: T_xM \to T_yM.$   
%  from $T_x M$ onto $T_y M.$ 
 We note  that $L_{y,x}=L_{x,y}^{-1}$  and 
 \begin{equation}\label{eq-parallel-vector}
\left\langle  L_{x,y} \zeta ,\nu\right\rangle_y=\left\langle \zeta, L_{y,x}\nu\right\rangle_x,\quad\forall\zeta\in T_xM,\,\,\,\nu\in T_yM.
\end{equation}
We also define the parallel transport of  a  symmetric  bilinear form      along the unique minimizing geodesic;  see \cite[p. 311]{AFS}.

\begin{definition}\label{def-PT-hess}
Let $x,y\in M,$ and let $\gamma:[0,1]\to M$ be  a unique minimizing geodesic  joining $x$ to $y.$  For $S_x\in \Sym TM_x,$ 
the parallel transport of $S_x$ from $x$ to $y,$ denoted by $L_{x,y}\circ S_x,$  is a symmetric bilinear form on $T_yM$ satisfying  
$$\left\langle\left( L_{x,y}\circ S_x\right)\cdot\nu,\nu\right\rangle_y:=\left\langle S_x\cdot\left(L_{y,x}\nu\right) , L_{y,x}\nu\right\rangle_x,\qquad\forall\nu\in T_yM.$$
%where  $L_{x,y}$ is the parallel transport from $T_x M$  to $T_yM$ along   $\gamma.$
\end{definition}
Identifying the space of symmetric   bilinear forms on $T_y M$ with the space of  symmetric endomorphisms of $T_y M,$   $L_{x,y}\circ S_x $  can be considered as a  symmetric endomorphism  of $T_yM$  such that
$$\left(L_{x,y}\circ S_x\right) \cdot \nu = L_{x,y}\left( S_x \cdot\left( L_{y,x}  \nu\right) \right), \qquad \forall \nu\in T_yM.$$ 
 Then it is not difficult  to check  that $S_x$ and $L_{x,y}\circ S_x$ have the same eigenvalues.

Let the Riemannain curvature tensor be defined by  
$$ R(X, Y)Z = \D_X \D_YZ- \D_Y\D_XZ  -\D_{ [X,Y ]}Z.$$
%where $\D$ denotes the Riemannian connection of $M.$ 
For 
 two linearly independent  vectors $X,Y\in T_xM,$ we   define the sectional curvature of the
plane determined by $X$ and $Y$ as 
 $$\Sec(X,Y):=\frac{\langle R(X, Y)X,Y\rangle}{|X|^2|Y|^2-\langle X,Y\rangle^2}.$$
%If $\Sec(X,Y) \geq \kappa$ for all $X,Y \in T_xM$ and $x\in M$ with   $\kappa\in\R,$ we write $\Sec\geq \kappa$ on $M$ for simplicity. 
 Let   $\Ric$    denote the   Ricci curvature  tensor defined as follows: for a unit vector $X\in T_xM$ and    an orthonormal basis $\{X,e_2,\cdots,e_n\}$ of $T_xM,$ $$\Ric(X,X)=\sum_{j=2}^n \Sec(X,e_j).$$  
 As usual,   $\Ric\geq \kappa$ on $M\,\,\,(\kappa\in\R)$ stands  for   $\Ric_x \geq \kappa g_x $ for all $x\in M.$   %We note that if $\Sec_x\geq \kappa\,\,\,(\kappa\in\R)$ for $x\in M,$  then   $\Ric_x\geq (n-1)\kappa.$
  
%We mention the  formulas for the first and second variations of the energy  (see for instance \cite{D}). 
We recall the first and second variations of the energy function (see for instance, \cite{D}).  

\begin{lemma}[First and second variations of energy]%\label{lem-1st-2nd-var-formula}
Let $\gamma :[0,1]\to M$ be a minimizing geodesic,   and $\xi$ be a vector field  along  $\gamma.$  For small  $\e>0,$ let $h: (-\e,\e)\times[0,1]\to M$ be a  variation of  $\gamma$ defined as 
$$h(r,t):=\exp_{\gamma(t)} r\xi(t) .$$ Define    the energy function of the variation %denoted by $E,$ 
$$E(r):=\int_0^1\left|\frac{\p h}{\p t}(r,t)\right|^2dt, \qquad\mbox{for}\,\,\, r\in(-\e,\e).$$ 
Then, we have
\begin{enumerate}[(a)]
\item  $$E(0)=d^2\left(\gamma(0), \gamma(1)\right),$$
\item 
\begin{align*}
\frac{1}{2}E'(0)%&=-\int_0^1\left\langle \xi(t), \frac{D}{dt}\dot\gamma(t)\right\rangle dt+ \left\langle\xi(1), \dot\gamma(1)\right\rangle-\left\langle\xi(0), \dot\gamma(0)\right\rangle\\
&=\left\langle\xi(1), \dot\gamma(1)\right\rangle-\left\langle\xi(0), \dot\gamma(0)\right\rangle,
\end{align*}
\item 
\begin{align*}
\frac{1}{2}E''(0)%&=-\int_0^1\left\langle \xi(t),  \frac{D^2\xi}{dt^2}+  R\left(\dot\gamma(t),\xi(t)\right)\dot\gamma(t) \right\rangle dt  \\
%&\,\,\,\,\,+ \left\langle\frac{D}{dr}\frac{\p h}{\p r}, \dot\gamma\right\rangle(0,1)-\left\langle\frac{D}{dr}\frac{\p h}{\p r}, \dot\gamma\right\rangle(0,0) +\left\langle\xi(1), \frac{D\xi}{dt}(1)  \right\rangle-\left\langle\xi(0), \frac{D\xi}{dt}(0)  \right\rangle\\
&=\int_0^1\left\{\left\langle\frac{D\xi}{dt},\frac{D\xi}{dt}\right\rangle-\left\langle    R\left(\dot\gamma(t),\xi(t)\right)\dot\gamma(t) ,\,\xi(t)\right\rangle \right\}dt.
%&=-\int_0^1\left\langle \xi(t),  \frac{D^2\xi}{dt^2}+  R\left(\dot\gamma(t),\xi(t)\right)\dot\gamma(t) \right\rangle dt   +\left\langle\xi(1), \frac{D\xi}{dt}(1)  \right\rangle-\left\langle\xi(0), \frac{D\xi}{dt}(0)  \right\rangle,
\end{align*}
%where $\frac{D\xi}{dt}$ denotes the covariant derivative of $\xi$ along $\gamma.$
\end{enumerate}
\end{lemma}
In particular, if a vector field $\xi$ is parallel   along $\gamma,$ then we have  $\frac{D\xi}{dt}\equiv0$ and $\langle\xi,\dot\gamma\rangle\equiv C$ (for   $C\in\R$ ) on $[0,1].$  In this case,  we  have the following estimate:
\begin{equation}\label{eq-2nd-var-PT}
E(r)=E(0)-r^2\int_0^1\left\langle    R\left(\dot\gamma(t),\xi(t)\right)\dot\gamma(t) ,\,\xi(t)\right\rangle dt+o\left(r^2\right).
\end{equation}

Now, we state some known results on Riemannian manifolds with a lower bound of the curvature. 
First, we have   the following volume  doubling property   assuming Ricci curvature to be  bounded from below (see \cite{V} for instance).  % the Bishop–Gromov inequality   states that the volume of balls does not increase faster than the volume of balls in the model space. In particular, 

 %(see \cite[Chapter 2]{L} for instance).
 %We have the following doubling property and refer to \cite[Chapter 14]{V}.t
 \begin{thm}[Bishop-Gromov] \label{lem-bishop}
Assume that   $\Ric 
  \geq -(n-1)\kappa$  on $M$ for $\kappa\geq0.$ 
 % \begin{itemize}
  %\item[(i) ] For $x,y\in M$ with $x\notin\Cut(y),$ 
    %$$\La d_y \leq  (n-1)\sqrt{k}\coth\left(\sqrt{k}d_y(x)\right).$$
  %\item[(ii)]{\rm(Bishop-Gromov)}
 For any $0<r < R,$ we have 
  \begin{equation}\label{eq-doubling}
  \frac{\vol(B_{2r}(z))}{\vol(B_r(z))}\leq 2^n\cosh^{n-1}\left(2\sqrt{\kappa}R\right).
  \end{equation}
  %\end{itemize}
  \end{thm} 
We observe  that the doubling property \eqref{eq-doubling}  implies that for any $0<r< R< R_0,$  
  \begin{equation*}%\label{eq-doubling-decay}
  \frac{\vol(B_{R}(z))}{\vol(B_r(z))}\leq  \cD\, \left(\frac{R}{r}\right)^{\log_2\cD},
  \end{equation*}
  where    $\cD:= 2^n\cosh^{n-1}\left(2\sqrt{ \kappa}R_0\right)$  is the so-called doubling constant. 
     Using the volume doubling property, it is easy to prove the following lemma.
 %The Hessian $D^2u$ of $u$ can be computed  by differentiation of  $u$ twice along a geodesic path, that is, if $\gamma(t)$ is a geodesic, then
%\begin{equation*}
%\frac{d^2}{dt^2}u(\gamma(t)) = \langle D^2u(\gamma(t))\dot{\gamma}(t), \dot{\gamma}(t)\rangle.
%\end{equation*}

\begin{lemma}\label{lem-weight-int-ellip}
Assume that   for any  $z\in M$ and $0 <r < 2R_0,$ there exists   a doubling constant $\cD>0$ such that
$$\vol(B_{2r}(z))\leq \cD \vol(B_r(z)).$$
Then we have that  for any $B_{r}(y)\subset B_R(z)$ with $0<r <  R <R_0, $  
\begin{equation}\label{eq-weight-int-ellip}
\left\{\fint_{B_r(y)}\left|r^2f\right|^{n\theta}\right\}^{\frac{1}{n\theta}}\leq  2 \left\{\fint_{B_R(z)}\left|R^2f\right|^{n\theta}\right\}^{\frac{1}{n\theta}};\qquad \theta:= \frac{1}{n}\log_2\cD.
\end{equation} 
%where $\theta:= \frac{1}{n}\log_2\cD.$  
In particular, if  the sectional curvature of $M$ is bounded from below by $-\kappa$ ($\kappa\geq0$),    then \eqref{eq-weight-int-ellip} holds with $\theta := 1+\log_2 \cosh(4\sqrt{\kappa}R_0). $ % $\theta := 1+\frac{n-1}{n}\log_2 \cosh(4\sqrt{\kappa}R_0). $
\end{lemma}
In the parabolic setting, 
it follows from Lemma \ref{lem-weight-int-ellip}  that for any $0<r <  R <R_0, $ 
\begin{equation}\label{eq-weight-int-para}
\left\{\fint_{K_{r,\,\al r^2}(y,s)}\left|r^2f\right|^{n\theta+1}\right\}^{\frac{1}{n\theta+1}}\leq  2 \al^{-\frac{1}{n\theta+1}} \left\{\fint_{K_R(z,t)}\left|R^2f\right|^{n\theta+1}\right\}^{\frac{1}{n\theta+1}},
\end{equation} 
where $K_{r,\al r^2}(y,s):= B_r(y)\times(s-\al r^2,s]\subset K_R(z,t)=B_R(z)\times(t-R^2,t]$ for $\al>0.$

We recall   semi-concavity  of functions on  Riemannian manifolds which is a natural generalization of concavity.   
The work of Bangert \cite{B} concerning  semi-concave functions enables us to deal with    functions that  are not twice differentiable in the usual sense.

%that retains most of the good properties known in convex analysis. 
%We recall   non-smooth analysis  on Riemannian manifolds as convex analysis in the Euclidean space, which allow us to handle the fact that semi concave functions are not twice differentiable in the usual sense.
  
 % has been thought as a  replacement for Riemannian case of convexity in convexity-type estimates on $\R^n.$  
   %    In nonsmooth analysis, convexity-type estimates are often used as a replacement for second- order derivative bounds.
  
  %allow us to handle the fact that c-concave functions are not twice differentiable in the usual sense. almost everywhere second derivatives of semiconcave function

\begin{definition}
Let $\Omega$ be an open set of $M.$  A function $\phi:\Omega\to\R$ is said to be semi-concave at $x_0\in\Omega$  if  there exist a geodesically convex  ball $B_r(x_0)$  with  $0<r<i_M(x_0),$ and a smooth function 
  $ \Psi : B_r(x_0) \to \R$ such that  $\phi+\Psi$ is geodesically concave on  $B_r(x_0).$ A function $\phi$ is semi-concave on $\Omega$ if it is semi-concave at each point in $\Omega.$
\end{definition}

The following   local characterization of semi-concavity is quoted  from \cite[Lemma 3.11]{CMS}. 
\begin{lemma}\label{lem-semiconcave} 
 Let $\phi:\Omega\to\R$ be a continuous function and let  $x_0\in \Omega,$ where $\Omega\subset M$ is open.    
 Assume that there exist a neighborhood $U$ of $x_0,$ and a  constant $C>0$ such that for any $x \in U$ and $X \in T_x M$ with $|X|=1,$  
$$\limsup_{r\to 0}\frac{\phi\left(\exp_x rX \right)+\phi\left(\exp_x -rX \right)-2\phi(x)}{r^2}\leq C. $$
Then $\phi$ is semi-concave at $x_0.$
\end{lemma}

 %The following lemma is concerned with Hessian bound for the squared distance 
Hessian bound for the squared distance function is the following lemma which is proved  in \cite[Lemma 3.12]{CMS}  using the formula for the second variation of energy.  According to the local characterization of semi-concavity combined with Lemma \ref{lem-hess-dist-sqrd},     $d_y^2$ is semi-concave  on  a bounded open set $\Omega\subset M$ for any $y\in M,$ provided that the sectional curvature of $M$ is bounded from below. %, and     an open set $\Omega \subset M$   is  bounded. 
%the closure $\overline\Omega$ of an open set $\Omega \subset M$   is compact.    

 \begin{lemma}\label{lem-hess-dist-sqrd}
 Let $x, y\in M.$  If $\Sec
  \geq -\kappa\,\, ( \kappa\geq0)$ along a minimizing geodesic  joining $x$ to $y,$ then for any $X\in T_xM$ with $|X|=1,$
  $$\limsup_{r\to0}\frac{d_y^2\left(\exp_xrX\right)+d_y^2\left(\exp_x -rX\right)-2d^2_y(x)}{r^2}\leq 2\sqrt{\kappa}d_y(x)\coth\left(\sqrt{\kappa}d_y(x)\right).$$
  \end{lemma}

The following  result  from Bangert is an extension of Aleksandrov's second differentiability theorem  that a convex function has second derivatives almost everywhere in the Euclidean space \cite{A} (see also \cite[Chapter 14]{V}) .  
\begin{thm}[Aleksandrov-Bangert, \cite{B}]\label{thm-AB}
Let $\Omega\subset M$ be an open set  and let $\phi:\Omega\to \R$   be semi-concave. Then  for almost every $x\in\Omega, $ $\phi$ is differentiable at  $x,$ and   there exists a symmetric operator $A(x):T_xM\to T_xM$ %. In particular, for $x\in\Omega$ and $\xi\in T_xM,$
characterized by any one of the two equivalent properties:
%in the sense that   for   $\xi\in T_xM,,$
\begin{enumerate}[(a)]
\item  for   $\xi\in T_xM,\,\,\,$ $A(x)\cdot\xi=\D_\xi\D\phi(x),$ \\
\item
%\begin{equation*}%\label{eq-taylor-exp}
$\phi(\exp_x \xi)=\phi(x)+\left\langle\D\phi(x),\xi\right\rangle+\frac{1}{2}\left\langle A(x)\cdot \xi,\xi\right\rangle+o\left(|\xi|^2\right)\quad\mbox{ as $\xi\to0.$}
$%\end{equation*}\\
\end{enumerate}
%The operator $A(x)$  is denoted by $D^2\phi(x). $  % The  Hessian $D^2\phi(x)$  can be viewed either as a symmetric quadratic form on $T_x M$, or  as a symmetric operator from $T_x M$ to itself. 
%When no confusion is possible, 
The operator $A(x)$   and its associated  symmetric bilinear from  on $T_xM$  are  denoted by $D^2\phi(x)$ and  called the Hessian of $\phi$ at $x$ when   no confusion is possible. 
\end{thm}

%We end this subsection by stating  the area formula.
Let $M$ and $N$ be   Riemannian manifolds of dimension $n$ and  $\phi:M\to N$ be smooth.  
The Jacobian of $\phi$ is  the absolute value of determinant of the differential $d \phi$,  i.e., $$\Jac\phi(x):=|\det d\phi(x)|\quad\mbox{for $x\in M$}. $$  
The following is the area formula, which follows easily from  the area formula in Euclidean space and  a partition of unity. 
\begin{lemma}[Area formula]
 For any smooth function $\phi:M\times \R\to M\times \R$ and any measurable set $E\subset M\times\R$, we have
\begin{equation*}
\int_E \Jac\phi (x,t)dV(x,t)=\int_{M\times\R} \cH^0[E\cap\phi^{-1}(y,s)] dV(y,s),\end{equation*}
where $\cH^0$ is the counting measure.
\end{lemma}

%%%%%%%%%%%%%%%%%%%%%%%%%%%%%%%%%%%%%%%%%%%%%%%%%%%%%%%%%%%%%%%%%%%%%%%%%%%%%%%%%%%%%%%%%%%
\subsection{Viscosity solutions}%\label{subsec-viscosity-sol}
%%%%%%%%%%%%%%%%%%%%%%%%%%%%%%%%%%%%%%%%%%%%%%%%%%%%%%%%%%%%%%%%%%%%%%%%%%%%%%%%%%%%%%%%%%%%%%
%We recall the definition of  viscosity solutions of parabolic equations on Riemannian manifolds. 
In this subsection, we consider a  refined definition of viscosity solutions  to parabolic equations slightly different from the usual definition in  \cite{Z}; see \cite{W} for the Euclidean case.

\begin{definition} Let $\Omega\subset M $ be open and $T>0.$ Let  $u: \Omega\times(0,T]\to\R$ be a lower semi-continuous function.
We say that  $u$ has  a local  minimum  at $(x_0,t_0)\in \Omega\times(0,T]$  in the parabolic sense  if there exists    $r>0$ such that
\begin{align*}
u(x,t)\geq u(x_0,t_0)\quad\mbox{for all $(x,t)\in K_r(x_0,t_0):=B_r(x_0)\times(t_0-r^2,t_0].$}
\end{align*}
\end{definition}
Similarly, we can define a local maximum in the parabolic sense.  
\begin{definition}[Viscosity sub and super- differentials]
Let $\Omega\subset M$ be open and $T>0.$ Let  $u: \Omega\times(0,T]\to\R$ be a lower semi-continuous function. We define the second order parabolic  subjet of $u$ at $(x,t)\in \Omega\times(0,T] $ by
\begin{align*}
\cP^{2,-}u(x,t):=&\left\{\left( \p_t\varphi(x,t),\D\varphi(x,t),D^2\varphi(x,t)\right)\in \R\times T_xM\times\Sym TM_x: \varphi\in C^{2,1}\left( \Omega\times(0,T]\right), \right.\\
&\left.\,\, u-\varphi \,\,\,\mbox{has a local minimum at $(x,t)$ in the parabolic sense}\right\}.
\end{align*}
%where $\Sym TM$ stands for the  symmetric 2-tensor at $x.$ 
If $(p,\zeta,A)\in \cP^{2,-}u(x,t),$ then $(p,\zeta)$ and $A$ are  called a first order subdifferential  (with respect to $(t,x)$), and a second order subdifferential  (with respect  to $x$) of $u$  at $(x,t),$ respectively.   

In a similar way,  for an upper semi-continuous function
 $u: \Omega\times(0,T] \to\R,$ we define the second order parabolic  superjet of $u$ at $(x,t)\in \Omega\times(0,T] $ by
\begin{align*}
\cP^{2,+}u(x,t):=&\left\{\left(\p_t\varphi(x,t),\D\varphi(x,t),D^2\varphi(x,t)\right)\in \R\times T_xM\times\Sym TM_x: \varphi\in C^{2,1}\left( \Omega\times(0,T]\right),\right.\\
&\left.\,\, u-\varphi \,\,\,\mbox{has a local maximum at $(x,t)$ in the parabolic sense}\right\}. 
\end{align*}
\end{definition}

 The following   characterization of  the parabolic subjet $\cP^{2,-}u$  can be obtained by a simple modification of \cite[Proposition 2.2]{AFS}, \cite[Proposition 2.2]{Z}.  
 \begin{lemma}\label{lem-char-superjet}
 Let $u:\Omega\times(0,T]\to\R $ be a lower semi-continuous function and let $(x,t)\in \Omega\times(0,T] .$
The following statements are equivalent:
\begin{enumerate}[(a)]
\item $(p,\zeta,A)\in \cP^{2,-}u(x,t),$
\item  for $\xi\in T_xM$ and $\sigma\leq0,\,\,$
$$u\left(\exp_x\xi,t+\sigma\right)\geq u(x,t)+\langle\zeta,\xi\rangle+ \sigma p+\frac{1}{2}\left\langle A\cdot \xi, \xi\right\rangle+o\left(|\xi|^2+|\sigma|\right)\,\,\,\mbox{as $(\xi,\sigma)\to(0,0)$}.$$%  as $|\xi|^2+|\sigma|\to0.$
\end{enumerate}
\end{lemma}

%See \cite{Z} and \cite{Ju}

\begin{definition}[Viscosity solution]\label{def-visc-sol-p} Let $F:M\times\R\times \R\times TM\times \Sym TM\to \R,$  and  let $\Omega\subset M$ be  open and  $T>0.$
We say that an  upper semi-continuous function $u: \Omega\times(0,T] \to\R$ is a parabolic viscosity subsolution of the equation $\p_t u =F(x,t,u , \D u , D^2u)$ in $\Omega\times(0,T]$ if $$p-F\left(x,t,u(x,t),\zeta,A\right)\leq 0$$
for any $(x,t)\in\Omega\times(0,T]$ and $(p,\zeta,A)\in \cP^{2,+}u(x,t).$ 
Similarly,  a  lower semi-continuous function $u: \Omega\times(0,T] \to\R $  is said to be 
a parabolic viscosity supersolution of the equation $\p_t u =F(x,t,u , \D u , D^2u)$ in $\Omega\times(0,T]$ 
  if  $$p-F\left(x,t,u(x,t),\zeta,A\right)\geq 0$$ for any $(x,t)\in\Omega\times(0,T]$ and $(p,\zeta,A)\in \cP^{2,-}u(x,t).$ 
We say that $u$ is  a parabolic viscosity solution if $u$ is  both a parabolic viscosity subsolution and a parabolic viscosity supersolution.
\end{definition}

%We remark that a  refined definition of viscosity solutions  for parabolic equations is considered  (see \cite{W} for the Euclidean case)  different from the definition in  \cite{Z}. 
 We remark that   parabolic viscosity solutions  at the present time will not be influenced by what is to happen in the future.  In the Euclidean space, Juutinen  \cite{Ju} showed that a refined definition of  parabolic viscosity solutions  is equivalent to the usual one if   comparison principle holds.   Whenever we refer to a  ``viscosity (sub or super) solution"  to  parabolic equations in this paper, we always mean a ``parabolic viscosity (sub or super) solution"  for simplicity.
%  a parabolic viscosity solution will be   abbreviated  into just  a viscosity solution.

\begin{comment}
\begin{definition}[Viscosity solution]\label{def-visc-sol-p} Let $F:  \Sym TM\to \R$ and let $\Omega$ be an open set in $M,$ $T>0.$
We say that a  upper semicontinuous function $u: \Omega\times(0,T] \to\R$ is a viscosity subsolution of the equation $F( D^2u)-\p_t u=f$ on $\Omega\times(0,T]$ if $$F\left( A\right)-p\geq f(x,t)$$
for any $(x,t)\in\Omega\times(0,T]$ and $(p,\zeta,A)\in \cP^{2,+}u(x,t).$ 
Similarly, a viscosity supersolution of  $F(  D^2u)-\p_t u=f$ on $\Omega\times(0,T]$ 
  is a  lower semicontinuous function $u $ satisfying  if $$F\left( A\right)-p\leq f(x,t)$$ for any $(x,t)\in\Omega\times(0,T]$ and $(p,\zeta,A)\in \cP^{2,-}u(x,t).$ 
We say $u$ is  a viscosity solution if $u$ is  both a viscosity subsolution and a viscosity supersolution.
\end{definition}

\begin{lemma}\label{lem-visc-semi-a.e.} Let $u$ be a viscosity supersolution of \eqref{eq-main} in $\Omega\times(T_1,T_2],$ where $\Omega$ is an open set in $M$ and $T_1<T_2.$ If $u$ satisfies  at $(x_0,t_0)\in\Omega\times(T_1,T_2]$, 
\begin{equation*}%\label{eq-taylor-exp}
u(\exp_{x_0} \xi,t_0+\sigma)=u(x_0,t_0)+\left\langle\D u(x_0,t_0),\xi\right\rangle+\sigma\p_tu(x_0,t_0)+\frac{1}{2}\left\langle D^2u(x_0,t_0)\xi,\xi\right\rangle+o\left(|\xi|^2+|\sigma|\right)\end{equation*}
 as $(\xi,\sigma)\in T_{x_0}M\times(-\infty,0]\to0,$ 
 then we have 
$$F\left(D^2u(x_0,t_0)\right)-\p_t u(x_0,t_0)\leq f(x_0,t_0).$$ 
\end{lemma}
\begin{proof}
\end{proof}
\end{comment}

We end this subsection by recalling  Pucci's extremal operators  and their properties. 
We refer to \cite{CC} for the proof.

\begin{definition}[Pucci's extremal operator]%\label{def-pucci}
  For $0<\lambda\leq \Lambda$ (called ellipticity constants), the Pucci's extremal operators are defined as follows: for any $x\in M,$ and  $S_x\in\Sym TM_x,$
\begin{align*}
\cM^+_{\lambda,\Lambda}(S_x):=\cM^+(S_x)=\lambda\sum_{e_i<0}e_i+\Lambda\sum_{e_i>0}e_i,\\
\cM^-_{\lambda,\Lambda}(S_x):=\cM^-(S_x)=\Lambda\sum_{e_i<0}e_i+\lambda\sum_{e_i>0}e_i,
\end{align*}
where $e_i=e_i(S_x)$ are the eigenvalues of $S_x.$
\end{definition}
In the special case when $\lambda=\Lambda=1,$  the Pucci's extremal operators $\cM^\pm$ simply coincide with the trace operator, that is, $\cM^{\pm}_{{1,1}}(D^2u)=\La u.$%the Laplace operator. 

\begin{lemma}\label{lem-prop-pucci-op}
Let $\Sym(n)$ denote the set of $n\times n$ symmetric matrices.  For $S,P\in \Sym(n),$ the followings hold: 
\begin{enumerate}[(a)]
\item 
%\begin{align*} 
 $$\cM^+(S)=\sup_{A\in\cS_{\lambda,\Lambda}}\mathrm{trace} (AS),\quad\mbox{and}\quad  \cM^-(S)=\inf_{A\in\cS_{\lambda,\Lambda}} \mathrm{trace} (AS),$$
%\end{align*}
where $\cS_{\lambda,\Lambda}$ consists of  positive definite symmetric matrices in $\Sym(n),$   whose eigenvalues  lie in $[\lambda,\Lambda]$.
\item  $\cM^-(-S)=-\cM^+(S)$.
\item $\cM^-(S+P)\leq  \cM^-(S)+\cM^+(P)\leq \cM^+(S+P)\leq \cM^+(S)+\cM^+(P)$.
\end{enumerate}
\end{lemma}
%The proof of Lemma \ref{lem-prop-pucci-op} can be found in \cite{CC}.

\begin{notation}
%Let us summarize  the notations  %and definitions 
%that  will be used.
%\begin{itemize}
%\item 
Let $r>0, \rho>0$,  $z_0\in M$ and $t_0\in\R$.  We denote   
$$K_{r,\,\rho}(z_0,t_0):=B_r(z_0)\times(t_0-\rho,t_0], $$  
where $B_r(z_0)$ is a geodesic ball of radius $r$ centered at $z_0$. 
%\item 
In particular, we denote  $$K_{r}(z_0,t_0):= K_{r, \,r^2}(z_0,t_0).$$
%\item  We say that  a constant $C$ is uniform if $C$ depends only on $n, \lambda,\Lambda$ and $a_L$.
%\item We denote $\fint_Q f:=\frac{1}{\vol(Q)}\int_Q f.$
%\item We denote $|Q|:=\vol(Q)$. 
%\item We denote  the trace by $\tr$.
%\end{itemize}
\end{notation}

%%%%%%%%%%%%%%%%%%%%%%%%%%%%%%%%%%%%%%%%%%%%%%%%%%%%%%%%%%%%%%%%%%%%%%
\section{Sup and inf-convolutions }\label{sec-supinf-conv}
 
In this section, we study the sup and inf-convolutions  introduced by Jensen\cite{J} (see also \cite{JLS},  \cite[Chapter 5]{CC}) to regularize continuous viscosity solutions. Let $\Omega\subset M $  be  a bounded open set, and  let  $u$ be a continuous function on $\overline \Omega\times[T_0,T_2]$  for $T_2>T_0.$ For $\ve>0,$ we define  the inf-convolution  of $u$   (with respect to $ \Omega\times(T_0,T_2]$),  denoted by  $u_{\ve},$   as follows: for $(x_0,t_0)\in  \overline\Omega\times[T_0,T_2],$
\begin{equation}\label{eq-def-inf-conv-p}
u_{\ve}(x_0,t_0):=\inf_{(y,s)\in  \overline\Omega\times[T_0,T_2] } \left\{ u(y,s) +\frac{1}{2\ve}\left\{d^2(y, x_0)+|s-t_0|^2\right\}\right\}.
 \end{equation}
 
%If $(\tilde x_0,\tilde t_0)$ is the point  to realize the above infimum, then we have 
%\begin{align*}
%u_\ve(x_0,t_0)&=u(\tilde x_0,\tilde t_0)+\frac{1}{2\ve}\left\{d^2(\tilde x_0, x_0)+|\tilde t_0-t_0|^2\right\}\\&\leq u(y,s) +\frac{1}{2\ve}\left\{d^2(y, x_0)+|s-t_0|^2\right\}\quad \forall y\in H,\end{align*} which means $u$ has a  touching  paraboloid  $-\frac{1}{2\ve}d^2(y, x_0)+u(y_0)+\frac{1}{2\ve}d^2(y_0, x_0)$ at $y_0$ from below.

 \begin{lemma} \label{lem-visc-u-u-e-0-p}For $u\in C\left(\overline \Omega\times[T_0,T_2]\right),$ let $u_\ve$ be  the inf-convolution of $u$ with respect to $  \Omega\times(T_0,T_2]$.  Let $(x_0,t_0)\in\overline\Omega\times[T_0,T_2].$%, where $H$ is an open set such that $\overline H\subset \Omega.$  
 \begin{enumerate}[(a)]
 \item If $0<\ve<\ve', $ then $\,\, u_{\ve'}(x_0,t_0)\leq u_\ve(x_0,t_0)\leq u(x_0,t_0).$
 \item  There exists $\displaystyle \, (y_0,s_0)\in \overline \Omega\times[T_0,T_2]$ such that $u_\ve(x_0,t_0)=u(y_0,s_0)+\frac{1}{2\ve}\left\{d^2(y_0, x_0)+|s_0-t_0|^2\right\}.$
 \item $\displaystyle d^2(y_0,x_0)+|s_0-t_0|^2\leq 2\ve |u(x_0,t_0)-u(y_0,s_0)|\leq 4\ve ||u||_{L^\infty\left(\Omega\times(T_0,T_2]\right)}.$
\item $u_\ve \uparrow u$ uniformly in $\overline\Omega\times[T_0,T_2].$ 
 \item $u_\ve$ is Lipschitz continuous in $\overline\Omega\times[T_0,T_2]$:   for $(x_0,t_0), (x_1,t_1)\in \overline\Omega\times[T_0,T_2],$
 \begin{equation}\label{eq-u-e-Lip}
 |u_{\ve}(x_0,t_0)-u_\ve(x_1,t_1)|\leq \frac{3}{2\ve}\diam(\Omega) d(x_0,x_1)+\frac{3}{2\ve}(T_2-T_0)|t_0-t_1|.
 \end{equation}
 \end{enumerate}
\end{lemma}
\begin{proof}
 From the definition of $u_\ve,$ $(a)$ and $(b)$ are obvious. 
From $(a)$ and $(b),$ it follows that 
$$\frac{1}{2\ve}\left\{d^2(y_0, x_0)+|s_0-t_0|^2\right\}= u_\ve(x_0,t_0)-u(y_0,  s_0) \leq u(x_0,t_0)-u(y_0,   s_0),$$ proving $(c).$ To show $(d),$ we observe that 
$$0\leq u(x_0,t_0)-u_\ve(x_0,t_0)\leq u(x_0,t_0)-u(y_0,   s_0).$$ We use $(c)$ and the uniform continuity of $u$ on $\overline\Omega\times[T_0,T_2]$  to deduce that $u_\ve$ converges  to $u$  uniformly on $\overline\Omega\times[T_0,T_2].$ 

Now we prove $(e).$ For $(y,s)\in \overline\Omega\times[T_0,T_2],$ we have
\begin{align*}
u_\ve(x_0,t_0)&\leq u(y,s)+\frac{1}{2\ve}\left\{d^2(y,x_0)+|s-t_0|^2\right\}\\
&\leq u(y,s) +\frac{1}{2\ve}\left\{\left(d(y,x_1)+d(x_1,x_0)\right)^2+\left(|s-t_1|+|t_1-t_0|\right)^2\right\}\\
&= u(y,s) +\frac{1}{2\ve}\left\{d^2(y,x_1)+d^2(x_0,x_1)+2d(y,x_1)d(x_0,x_1)+\left(|s-t_1|+|t_0-t_1|\right)^2\right\}\\
&\leq u(y,s)+\frac{1}{2\ve}\left\{d^2(y,x_1)+|s-t_1|^2\right\}+\frac{3}{2\ve}\diam(\Omega)d(x_0,x_1)+ \frac{3}{2\ve}(T_2-T_0)|t_0-t_1|. 
\end{align*}
Taking the infimum of the right hand side, we conclude \eqref{eq-u-e-Lip}, that is,  $u_\ve$ is Lipschitz continuous on $\overline\Omega\times[T_0,T_2].$    \end{proof}

Now, we show the semi-concavity of the inf-convolution,   and hence
the inf-convolution is twice differentiable  almost everywhere in the sense of Aleksandrov and Bangert's Theorem \ref{thm-AB}.

 \begin{lemma}\label{lem-u-e-alexandrov-p} 
  Assume that $$\Sec
  \geq -\kappa\quad\mbox{on $M$},\quad\mbox{for $\kappa\geq0$}$$%, where $H$ is an open set such that $\overline H\subset \Omega.$ Let $x_0,x_1\in H.$
For $u\in C\left(\overline\Omega\times[T_0,T_2]\right),$ let $u_\ve$ be  the inf-convolution of $u$ with respect to $ \Omega\times(T_0,T_2], $ where $\Omega\subset M$ is a  bounded open set, and $T_0<T_2.$ 
\begin{enumerate}[(a)]
 \item $u_\ve$ is semi-concave in $\Omega\times(T_0,T_2).$  Moreover,  for almost every $(x,t)\in\Omega\times(T_0,T_2),$ $u_\ve$ is differentiable at $(x,t),$ and there exists  the Hessian $D^2u_\ve (x,t)$ (in the sense of Aleksandrov-Bangert's Theorem \ref{thm-AB})   such that
 \begin{equation}\label{eq-u-e-taylor}
 u_\ve\left(\exp_x \xi,t+\sigma\right)=u_\ve(x,t)+\left\langle\D u_\ve(x,t),\xi\right\rangle+\sigma\p_tu_\ve(x,t)+\frac{1}{2}{\left\langle D^2u_\ve(x,t)\cdot\xi,\xi\right\rangle}+o\left(|\xi|^2+|\sigma|\right)
 \end{equation}  as $(\xi,\sigma)\in T_xM\times\R\to(0,0).$  % The  operator  $A $ is denoted by $D^2u_\ve,$ and called the Hessian of $u_\ve.$ 
 \item $\displaystyle D^2u_\ve(x,t)\leq \frac{1}{\ve}{\sqrt{\kappa}\diam(\Omega)\coth\left(\sqrt \kappa \diam(\Omega)\right)}\, g_x\,\,\,$ a.e. in $\Omega\times(T_0,T_2).$
 \item  Let $H\times(T_1,T_2]$  be a    subset such that $  \overline H\times[T_1,T_2]\subset\Omega\times(T_0,T_2],$ where $H$ is open,  and  $T_0<T_1<T_2.$
 Then, there exist a smooth function $\psi$ on $M\times(-\infty,T_2]$ satisfying    $$\mbox{$0\leq\psi\leq1$ on $M\times (-\infty,T_2]$},\,\,\,\,\,\psi\equiv1 \,\,\,\,\mbox{in $\overline H\times[T_1,T_2]\,\,\,\,\,$ and }\,\,\,\,\supp\psi\subset \Omega\times(T_0,T_2], $$ and a sequence $\{w_k\}_{k=1}^\infty$ of smooth functions on $M\times(-\infty,T_2]$  such that
\begin{equation*} 
\left\{
\begin{array}{ll}
 w_k\to \psi u_\ve\qquad &\mbox{uniformly in $M\times(-\infty,T_2]\,\,$  as $k\to+\infty,$}\\ 
 |\D w_k|+|\p_t w_k|\leq C \qquad &\mbox{in $M \times(-\infty,T_2]$, }\\
   \p_tw_k\to \p_tu_\ve\quad&\mbox{a.e. in $H\times(T_1,T_2)\,\,$ as $k\to+\infty,$}\\
     D^2 w_k\leq C  g\qquad &\mbox{in $M \times(-\infty,T_2]$, }\\
 D^2 w_k\to D^2u_\ve\quad&\mbox{a.e. in $H\times(T_1,T_2)\,\,$ as $k\to+\infty,$}\\
 \end{array}\right. \qquad \qquad
\end{equation*}
where the constant $C>0$ is independent of $k$.
  \end{enumerate}
\end{lemma}
\begin{proof}
To prove   semi-concavity of $u_\ve$ in $\Omega\times(T_0,T_2),$ %Since $d_y^2$ is semi-concave in $H$  (uniformly in $y\in H$)
 we    fix   $(x_0,t_0)\in \Omega\times(T_0,T_2),$  and   find  $(y_0,s_0)\in \overline\Omega\times[T_0,T_2]$ satisfying 
$$u_\ve(x_0,t_0)=u(y_0,s_0)+\frac{1}{2\ve}\left\{d^2(y_0, x_0)+|s_0-t_0|^2\right\}.$$
For  any  $\xi\in  T_{x_0}M  $ with $|\xi|=1,$ and for small $ r\in\R,$ it follows from   the definition of the inf- convolution $u_\ve$    that 
\begin{align*}
 u_\ve&\left(\exp_{x_0} r\xi,\,t_0+r\right)+u_\ve\left(\exp_{x_0} -r\xi,\,t_0-r\right)-2u_\ve(x_0,t_0)\\
&\leq  u(y_0,s_0)+\frac{1}{2\ve}\left\{d^2\left(y_0, \exp_{x_0} r\xi\right)+|s_0-(t_0+r)|^2\right\}\\
&\quad+u(y_0,s_0)+\frac{1}{2\ve}\left\{d^2\left(y_0, \exp_{x_0} -r\xi\right)+|s_0-(t_0-r)|^2\right\}
-2u_\ve(x_0,t_0)\\
&\leq \frac{1}{2\ve} \left\{ {d_{y_0}^2\left(\exp_{x_0} r\xi\right)+d_{y_0}^2\left(\exp_{x_0} -r\xi\right)-2d_{y_0}^2(x_0)}\right\} +\frac{1}{\ve}r^2.
% &\frac{u_\ve\left(\exp_{x_0} r\xi,\,t_0+r\right)+u_\ve\left(\exp_{x_0} -r\xi,\,t_0-r\right)-2u_\ve(x_0,t_0)}{r^2}\\
%&\leq \frac{1}{2\ve} \frac{d_y^2\left(\exp_x r\xi\right)+d_y^2\left(\exp_x -r\xi\right)-2d_y^2(x)}{r^2} +\frac{1}{\ve}.
%&=: \frac{1}{2\ve}\frac{d_{y_0}^2\left(\exp_x r\xi\right)+d_{y_0}^2\left(\exp_x -r\xi\right)-2d_{y_0}^2(x)}{r^2} +\frac{1}{\ve}.%\\
%&\leq\frac{1}{\ve} \sqrt{K}d_y(x)\coth\left(\sqrt{K}d_y(x)\right)\\
%&\leq\frac{1}{\ve}{\sqrt{K}\diam(H)\coth\left(\sqrt K \diam(H)\right)}
\end{align*}
%where $\overline\Omega$ is compact and  $y_0\in\overline\Omega.$
 Then, we use  Lemma \ref{lem-hess-dist-sqrd} to obtain that for any  $\xi\in  T_{x_0}M  $ with $|\xi|=1,$
 \begin{equation}\label{eq-dist-2-sc-p}
 \begin{split}
&\limsup_{r\to 0} \frac{u_\ve\left(\exp_{x_0} r\xi,\,t_0+r\right)+u_\ve\left(\exp_{x_0} -r\xi,\,t_0-r\right)-2u_\ve({x_0},t_0)}{r^2}\\
&\leq \limsup_{r\to 0}\frac{1}{2\ve}\frac{{d_{y_0}^2\left(\exp_{x_0} r\xi\right)+d_{y_0}^2\left(\exp_{x_0} -r\xi\right)-2d_{y_0}^2(x_0)}}{r^2}+\frac{1}{\ve}\\
&\leq\frac{1}{\ve}  \sqrt{\kappa}d_{y_0}(x_0)\coth\left(\sqrt{\kappa}d_{y_0}(x_0)\right)+\frac{1}{\ve}\\
&\leq\frac{1}{\ve}{\sqrt{\kappa}\diam(\Omega)\coth\left(\sqrt {\kappa} \diam(\Omega)\right)}+\frac{1}{\ve},
\end{split}
\end{equation}
where we note that $\tau\coth(\tau)$ is nondecreasing with respect to $\tau\geq0.$ We recall that $u_\ve$ is Lipschitz continuous on $\overline\Omega\times[T_0,T_2]$ according to Lemma \ref{lem-visc-u-u-e-0-p}. 
Since $(x_0,t_0)\in\Omega\times(T_0,T_2)$ % and $\xi\in T_{x_0}M$ with $|\xi|=1$ are
is  arbitrary,  \eqref{eq-dist-2-sc-p} and Lemma \ref{lem-semiconcave} imply that 
  $u_\ve$ is semi-concave on $\Omega\times(T_0,T_2).$   %  according to \cite[Lemma 3.11]{CMS}
 Thus,  
$u_\ve$  admits  the Hessian  almost everywhere in $\Omega\times(T_0,T_2)$ satisfying \eqref{eq-u-e-taylor} from Aleksandrov and Bangert's Theorem  \ref{thm-AB}.  The upper bound of the Hessian 
  %(see also \cite[Chapter 14]{V}).   
in $(b)$  follows from \eqref{eq-u-e-taylor}  
 and \eqref{eq-dist-2-sc-p}. % and a second order Taylor expansion  \eqref{eq-u-e-taylor}.
    
We use a standard mollification and a partition of unity to approximate $\psi u_\ve$ by a sequence $\{w_k\}_{k=1}^\infty$ of smooth functions in $(c), $ where a mollifier  is   supported in $ (-\delta,0]$ with respect to time  (for small $\delta>0$), not in $ (-\delta,\delta).$   By using  Lipschitz continuity of $u_\ve$ on $\overline\Omega\times[T_0,T_2]$ and semi-concavity  on $\Omega\times(T_0,T_2),$   it is not difficult to prove the properties of $w_k.$   For the  details, we  refer to the proof of Lemma 5.3 in   \cite{Ca}. % of the proof of $(c).$ %(see also  \cite[Lemmas 4.2,4.5]{KKL}). 
\end{proof} 
%$$$$

%refer to \cite{CIL} for the Euclidean case.

%Alexandrov's second differentiability theorem for semiconvex functions with a quadratic modulus of semiconvexity on Riemannian manifolds \cite{V}. 
%see also \cite{AF}.

Next, 
we shall prove that  if $u$ is a  viscosity supersolution to  \eqref{eq-main}, then the  inf-convolution  $u_\e$ is also a viscosity supersolution provided that the sectional curvature of the underlying manifold is bounded from below;  see   \cite[Lemma A.5]{CIL} for the Euclidean case.
  \begin{prop}\label{lem-visc-u-u-e-p}  
% {\color{black} Assume that $\Sec  \geq -k\,\, ( k\geq0)$ on $M.$} 
%Under the same assumption as Lemma \ref{lem-u-e-alexandrov-p}, 
Assume that $$\Sec
  \geq -\kappa\quad\mbox{on $M,\quad$ for $\kappa\geq0$.}$$
  Let $H$  and $\Omega$  be bounded  open sets in  $M$    such that $  \overline H\subset\Omega,$  and  $T_0<T_1<T_2.$
Let  $u\in C\left(\overline\Omega\times[T_0,T_2]\right),$
 and let $\omega$ denote  a modulus of continuity of $u$ on $\overline \Omega\times[T_0,T_2],$ which is nondecreasing  on $(0,+\infty)$ with $\omega(0+)=0.$
For $\ve>0,$  let $u_\ve$ be  the inf-convolution of $u$ with respect to $ \Omega\times(T_0,T_2].$   %and let $m:=||u||_{L^{\infty}\left(\overline\Omega\times[T_1,T_2]\right)}.$    
Then, there exists $\ve_0>0$ depending only  on $||u||_{L^{\infty}\left(\overline\Omega\times[T_0,T_2]\right)} , H,\Omega ,T_0,$ and $T_1,$   such that   if $0<\ve<\ve_0,$  then the following statements  hold: let  $(x_0,t_0)\in \overline H\times[T_1,T_2],$ and     let  $(y_0,s_0) \in\overline\Omega\times[T_0,T_2]$ satisfy 
$$u_\ve(x_0,t_0)=u(y_0,s_0)+\frac{1}{2\ve}\left\{d^2(y_0,x_0)+|s_0-t_0|^2\right\}.$$

\begin{enumerate}[(a)]
%\item  $y_0:=\exp_{x_0}(-\ve\zeta)\in \overline \Omega$ and $$u_\ve(x_0,t_0)=u(y_0,s_0)+\frac{1}{2\ve}\left\{d^2(y_0,x_0)+|s_0-t_0|^2\right\}$$for some $s_0\in[t_0-\ve p, T_2]$
\item 
We have  that    $$(y_0,s_0)\in\Omega\times(T_0,T_2],$$ and there is a unique minimizing geodesic joining  $x_0$ to $y_0.$

\item If $(p,\zeta,A)\in \cP^{2,-}u_\ve(x_0,t_0),$ % for $(x_0,t_0)\in \overline H\times[T_1,T_2]$   and  $(y_0,s_0)\in\overline\Omega\times[T_0,T_2]$ realizes the infimum of the inf- convolution $u_\ve$  at $(x_0,t_0)$, i.e.,
%$$u_\ve(x_0,t_0)=u(y_0,s_0)+\frac{1}{2\ve}\left\{d^2(y_0,x_0)+|s_0-t_0|^2\right\},$$
then  we have 
$$y_0=\exp_{x_0}(-\ve\zeta),\quad\mbox{and}\,\,\,\, s_0\in[t_0-\ve p, T_2].$$
\item If $(p,\zeta,A)\in \cP^{2,-}u_\ve(x_0,t_0),$  then we have %there exists $0<\delta_0<1$ such that
 $$ \left(p,\, L_{x_0,y_0}\zeta, \,L_{x_0,y_0}\circ A-{\color{black}2\kappa\,\omega\left(2\sqrt{\ve\,||u||_{L^\infty\left(\overline\Omega\times[T_0,T_2]\right)} }\right)}\,  g_{y_0}\right)\,\in \cP^{2,-}u(y_0,s_0),$$
where $L_{x_0,y_0}$ stands for the parallel transport along the unique minimizing geodesic joining $x_0$ to $y_0=\exp_{x_0}(-\ve\zeta).$ 
 \end{enumerate}
\end{prop}

\begin{proof} %For $(x_0,t_0)\in \overline H\times[T_1,T_2],$ let  $(y_0,s_0) \in\overline\Omega\times[T_0,T_2]$ satisfy 
%$$u_\ve(x_0,t_0)=u(y_0,s_0)+\frac{1}{2\ve}\left\{d^2(y_0,x_0)+|s_0-t_0|^2\right\}.$$
 By recalling Lemma \ref{lem-visc-u-u-e-0-p}, $(c),$  we see that 
 \begin{align*}
 (y_0,s_0) &\in \overline B_{2\sqrt{m\ve}}(x_0)\times\left(\left[t_0-2\sqrt{m\ve},\,t_0+2\sqrt{m\ve}\right] \cap \left[T_0,T_2\right]\right), 
 \end{align*} 
 for  $ m:=||u||_{L^{\infty}\left(\overline\Omega\times[T_0,T_2]\right)}.$ Since the distance between $\overline H$ and $\p\Omega$ is positive, we select $\ve_0>0$ so small     that
 $$2\sqrt{m\ve_0} <\min\left\{\,d(\overline H,\p\Omega),\,T_1-T_0\right\}=:\delta_0, $$ 
 where   $d(\overline H,\p\Omega)$ means the distance between $\overline H$ and $\p\Omega.$
For $0<\ve<\ve_0,$   we have that   $$(y_0,s_0)\in \Omega\times(T_0,T_2]$$  since  $(x_0,t_0)\in \overline H\times [T_1,T_2].$ %  proving the first statement in $(a).$   
We observe   that $$ i_{\overline \Omega}:=\inf\left\{\, i_M(x) : x\in\overline \Omega\,\right\}>0$$ 
  since  $\overline \Omega$ is compact from Hopf- Rinow Theorem and the map $x\mapsto i_M(x)$ is continuous.
Now, we select    $$\ve_0:= \frac{1}{8||u||_{L^{\infty}\left(\overline\Omega\times[T_0,T_2]\right)}}\min\left\{\,\delta_0^2,\,i_{\overline\Omega}^2\,\right\}.$$
Then  we have  that for $0<\ve<\ve_0,$ $$d^2(x_0,y_0)\leq 4\ve ||u||_{L^{\infty}\left(\overline\Omega\times[T_0,T_2]\right)}< 4\ve_0 ||u||_{L^{\infty}\left(\overline\Omega\times[T_0,T_2]\right)} < i_{\overline \Omega}^2,$$
 and hence  {\color{black} $d(x_0,y_0)<i_{\overline \Omega}\leq \min\left\{i_M(x_0),\,i_M(y_0)\right\},$}
   which  implies  the uniqueness of a minimizing geodesic joining $x_0$ to $y_0.$ 
    This  finishes the proof of $(a).$ 
%Therefore, we select and fix   $$\ve_0:= \frac{1}{8||u||_{L^{\infty}\left(\overline\Omega\times[T_0,T_2]\right)}}\min\left\{\,\delta_0^2,\,i_{\overline\Omega}^2\,\right\},$$ which finishes the proof of $(a).$ 

  From $(a),$ there exists a unique vector $X\in T_{x_0}M$ such that
 $$ y_0=\exp_{x_0}X,\quad\mbox{and}\quad |X|=d(x_0,y_0).$$ 
 First, we claim  that if $(p,\zeta,A)\in \cP^{2,-}u_\ve(x_0,t_0),$ then  $y_0= \exp_{x_0}(-\ve\zeta),$ namely, $X=-\ve\zeta.$   
   %:=\inf_{y\in H} \left\{ u(y) +\frac{1}{2\ve}d^2(y, x_0)\right\}$$
 Since $(p,\zeta,A)\in \cP^{2,-}u_\ve(x_0,t_0),$ we have that  for any $\xi\in T_{x_0}M$ with $|\xi|=1,$  small $r\in \R,  \sigma\leq0$ and for any $(y,s)\in \overline \Omega\times[T_0,T_2],$  
 \begin{equation}\label{ineq-u-u-e-p}
 \begin{split}
 u(y,s) &+\frac{1}{2\ve}\left\{d^2\left(y,\exp_{x_0}r\xi\right)+|s-(t_0+\sigma)|^2\right\}
\\
&\geq u_\ve\left(\exp_{x_0}r\xi, t_0+\sigma\right)\\
&\geq u_\ve(x_0,t_0)+r\langle\zeta,\xi\rangle+\sigma p+\frac{r^2\langle A\cdot \xi,\xi\rangle}{2}+o\left(r^2+|\sigma|\right)\\
&=u(y_0,s_0)+\frac{1}{2\ve}\left\{d^2(y_0,x_0)+|s_0-t_0|^2\right\}+r\langle\zeta,\xi\rangle+\sigma p+\frac{r^2\langle A\cdot \xi,\xi\rangle}{2}+o\left(r^2+|\sigma|\right).
 \end{split}
 \end{equation}
 
 When $(y,s)=(y_0,s_0)$ and $\sigma=0$ in  \eqref{ineq-u-u-e-p}, we see that  for small $r\geq0,$
 \begin{align*}
 \frac{1}{2\ve}\left\{d(y_0,x_0)+r\right\}^2&\geq
 \frac{1}{2\ve}d^2\left(y_0,\exp_{x_0}r\xi\right)\\
&\geq  \frac{1}{2\ve}d^2(x_0,y_0)+r\langle\zeta,\xi\rangle+\frac{r^2\langle A\cdot\xi,\xi\rangle}{2}+o\left(r^2\right),
 \end{align*} 
 and hence for small  $r\geq0,$
 \begin{equation}\label{eq-1-p}
 rd(x_0,y_0)\geq  r\left\langle \ve\zeta,\xi\right\rangle +O\left(r^2\right),\qquad\forall \xi\in T_{x_0}M\,\,\,\,\mbox{with $|\xi|=1.$}
 \end{equation} 
 If $X=0,$ \eqref{eq-1-p} implies that $\langle \ve\zeta,\xi\rangle=0$ for all $\xi\in T_{x_0}M.$ Thus we deduce that $\zeta=0$  and    $y_0=\exp_{x_0}0=\exp_{x_0}(-\ve \zeta).$ 
 
Now, we assume that $X\neq0.$  
If  $(y,s)=(y_0,s_0), \,\sigma=0,$ and $\xi=X/|X|=X/d(x_0,y_0)$ in \eqref{ineq-u-u-e-p}, then we have that  for small $r\geq0,$
  \begin{align*}
 \frac{1}{2\ve}\left\{d(x_0,y_0)-r\right\}^2&\geq
   \frac{1}{2\ve}d^2(x_0,y_0)+r\left\langle\zeta,\xi\right\rangle+\frac{r^2\langle A\cdot\xi,\xi\rangle}{2}+o\left(r^2\right)
 \end{align*} 
 and hence  for small $r\geq0,$
 \begin{equation}\label{eq-2-p}
 -r d(x_0,y_0)\geq r \left\langle \ve\zeta,X/|X|\right\rangle +O\left(r^2\right).
   \end{equation} 
 For small $r\geq0,$   \eqref{eq-1-p} and \eqref{eq-2-p}  imply that $$\langle -\ve\zeta,\xi\rangle\leq |X|=d(x_0,y_0),\qquad \forall\xi\in T_{x_0}M \quad\mbox{with $|\xi|=1,$}$$ and $$\langle -\ve\zeta,X/|X|\rangle=|X|=d(x_0,y_0).$$  Then, it follows  that 
  $-\ve\zeta=X$ and hence
%  $z_0=\exp_{x_0}X=\exp_{x_0}(-\ve\zeta)=y_0$ and 
 $y_0=\exp_{x_0}X=\exp_{x_0}(-\ve\zeta)$  for $X\neq0.$  Thus we have proved that $y_0=\exp_{x_0}(-\ve\zeta).$ 
 
 When $(y,s)=(y_0,s_0)$ and $r=0$ in \eqref{ineq-u-u-e-p}, we have that for small $\sigma\leq0,$
 $$\frac{1}{2\ve}|s_0-t_0-\sigma|^2\geq\frac{1}{2\ve}|s_0-t_0|^2+\sigma p+o(|\sigma|),$$ which  implies that $s_0\geq t_0-\ve p.$ This proves $(b).$

To show $(c),$ 
 we recall that  there is a unique minimizing geodesic joining $x_0$ to $y_0,$ and $(y_0,s_0)\in\Omega\times(T_0,T_2]$ according to  $(a).$ 
Using the parallel transport, we rewrite  \eqref{ineq-u-u-e-p} as follows:  for any $\nu\in T_{y_0}M$  with  $|\nu|=1,$ and  small $r\in\R,\sigma\leq0,$  and for $(y,s)\in \overline\Omega\times[T_0,T_2], $  
  %$r\in \R,\sigma\leq0$ and $(y,s)\in \overline\Omega\times[T_1,T_2],$
 \begin{equation*}
 \begin{split}
 u(y,s)&\geq u(y_0,s_0) +r\left\langle\zeta,L_{y_0,x_0}\nu\right\rangle_{x_0}+\sigma p+\frac{r^2 }{2}\left\langle A\cdot \left( L_{y_0,x_0}\nu\right), L_{y_0,x_0}\nu\right\rangle_{x_0}\\
 &+\frac{1}{2\ve}\left\{d^2(y_0,x_0)-d^2\left(y,\exp_{x_0} rL_{y_0,x_0}\nu\right)\right\}+\frac{1}{2\ve}\left\{|s_0-t_0|^2-|s-t_0-\sigma|^2\right\}+o (r^2+|\sigma|).
\end{split} 
\end{equation*}
By setting $(y,s):=\left(\exp_{y_0}r\nu, s_0+\sigma\right)$ for small $r\in\R,\,\sigma\leq0,$ we claim  that% for any $\nu\in T_{y_0}M$  with  $|\nu|=1,$
 \begin{equation}\label{eq-2nd-var-p}
 \begin{split}
 u\left(\exp_{y_0}r\nu,s_0+\sigma \right)
 &\geq u(y_0,s_0)+r\left\langle L_{x_0,y_0}\zeta,\nu\right\rangle_{y_0}+\sigma p+\frac{r^2}{2}\left\langle \left(L_{x_0,y_0}\circ A\right)\cdot\nu,\nu\right\rangle_{y_0}\\
 &\,\,\,\,+{\color{black}\frac{1}{2\ve}\left\{d^2(y_0,x_0)- d^2\left(\exp_{y_0}r\nu,\,\exp_{x_0}L_{y_0,x_0}r\nu\right)\right\}}+o\left(r^2+|\sigma|\right)\\ 
 &\geq u(y_0,s_0)+r\left\langle L_{x_0,y_0}\zeta,\nu\right\rangle_{y_0}+\sigma p+\frac{r^2}{2}\left\langle \left(L_{x_0,y_0}\circ A\right)\cdot\nu,\nu\right\rangle_{y_0}\\
 &\,\,\,\, -\frac{1}{2\ve} r^2\kappa\, d^2(x_0,y_0) +o\left(r^2+|\sigma|\right).\\ 
% &\geq u(y_0)+s\left\langle P_{x_0,y_0}\zeta, \xi\right\rangle_{y_0}+\frac{s^2\left\langle P_{x_0,y_0} A \xi,\xi\right\rangle_{y_0}}{2} 
% &+{\color{black} \frac{1}{2\ve} s^2\int_{0}^1 K\left(\xi(t),\dot\gamma(t)\right)dt}+o\left(s^2\right),\\
 %-\frac{s^2}{2\ve} K_0 d^2(x_0,y_0) +o(s^2).
%&\geq u(y_0)+s\left\langle P_{x_0,y_0}\zeta, \xi\right\rangle_{y_0}+\frac{s^2\left\langle P_{x_0,y_0} A \xi,\xi\right\rangle_{y_0}}{2}+o\left(s^2\right),\\
\end{split} 
\end{equation}
%where $\gamma(t):=\exp_{x_0}(-t\ve\zeta)$ and $\xi(t):=P_{y_0,\gamma(t)}\xi.$
The first inequality is immediate from  \eqref{eq-parallel-vector} and Definition \ref{def-PT-hess}. 
To prove the second inequality in \eqref{eq-2nd-var-p}, 
we consider a unique minimizing geodesic $$\gamma(t):=\exp_{x_0}(-t\ve\zeta)$$  joining $\gamma(0)=x_0$ to $\gamma(1)=y_0=\exp_{x_0}(-\ve\zeta).$ For a given $\nu\in T_{y_0}M$ with $|\nu|=1,$  define  a variational field $$\nu(t):=L_{y_0,\gamma(t)}\nu\,\in T_{\gamma(t)}M$$ along $\gamma, $ where $\nu(0)=L_{y_0,x_0} \nu,$ and $\nu(1)=\nu.$
For small $\e>0,$ we define  a variation  
 $h:(-\e,\e)\times[0,1]\to M, $ of $\gamma,$    
 $$h(r,t):= \exp_{\gamma(t)} r\nu(t) .$$%  where $\gamma(t):=\exp_{x_0}(-t\ve\zeta)$ and $\nu(t):=L_{y_0,\gamma(t)}\nu.$
 The energy is defined  as 
$$E(r):=\int_0^1\left|\frac{\p h}{\p t}(r,t)\right|^2dt.$$ 
We use the second variation of energy formula \eqref{eq-2nd-var-PT}  to obtain 
$$E(r)=E(0)-r^2\int_{0}^1 \left\langle R\left(\dot\gamma(t),\nu(t)\right)\dot\gamma(t),\nu(t)\right\rangle dt+o\left(r^2\right) $$
since $\gamma$ is a  unique minimizing geodesic, and $\nu(t)$ is  parallel transported along $\gamma.$ %(see \cite{dC}).
Since $|\nu(t)|=|\nu|=1,$ and $|\dot\gamma(t)|=|\dot\gamma(0)|=d(x_0,y_0)$ for   $t\in[0,1],$ we have  that
\begin{align*}
  E(0)-E(r)
 &=r^2\int_{0}^1 \left\langle R\left(\dot\gamma(t),\nu(t) \right)\dot\gamma(t), \nu(t)\right\rangle dt+o\left(r^2\right)\\
 %&=r^2\int_{0}^1 |\dot\gamma(t)|^2\left\langle    R\left(\frac{\dot\gamma(t)}{|\dot\gamma(t)|},\nu(t) \right)\frac{\dot\gamma(t)}{|\dot\gamma(t)|}, \nu(t)\right\rangle dt+o\left(r^2\right)\\
 &=r^2 \int_{0}^1 \Sec\left(\dot\gamma(t),\nu(t)\right) \cdot \left(|\dot\gamma(t)|^2-\left\langle\dot\gamma(t),\nu(t)\right\rangle^2\right)dt+o\left(r^2\right)\\
  &\geq -r^2 \int_{0}^1 \kappa \left(|\dot\gamma(t)|^2-\left\langle\dot\gamma(t),\nu(t)\right\rangle^2\right)dt+o\left(r^2\right)\\
%&\geq r^2\int_{0}^1 |\dot\gamma(t)|^2 \Sec\left(\dot\gamma(t)/|\dot\gamma(t)|,\nu(t)\right)dt+o\left(r^2\right)\\
 &\geq - r^2 \kappa\, |\dot\gamma(0)|^2+o\left(r^2\right)\,\,=-r^2 \kappa \,d^2(x_0,y_0) +o\left(r^2\right).  %&\geq -2\ve K \osc_H u\cdot s^2+o(s^2)\\
\end{align*}
Recalling that $E(0)=d^2(x_0,y_0),$ and $$E(r)\geq d^2\left(\exp_{\gamma(0)}r\nu(0),\exp_{\gamma(1)}  r\nu(1)\right)=d^2\left(\exp_{x_0}L_{y_0,x_0}r\nu,\,\exp_{y_0} r\nu\right),$$   we obtain
 \begin{align*}
d^2(x_0,y_0)- d^2\left(\exp_{x_0}L_{y_0,x_0}r\nu,\exp_{y_0}  r\nu\right) &\geq E(0)-E(r)\\
 &\geq -r^2 \kappa \,d^2(x_0,y_0) +o\left(r^2\right),%&\geq -2\ve K \osc_H u\cdot s^2+o(s^2)\\
\end{align*}
which proves  the second inequality of \eqref{eq-2nd-var-p}.

 Since  
$d^2(x_0,y_0)+|t_0-s_0|^2\leq 4\ve||u||_{L^\infty\left(\overline\Omega\times[T_0,T_2]\right)} $  from Lemma \ref{lem-visc-u-u-e-0-p},  $(c)$, 
it follows that
\begin{equation}\label{eq-dist-bd-omega}
d^2(x_0,y_0)\leq 2\ve \,|u(x_0,t_0)-u(y_0,s_0)|\leq2\ve\,\omega\left(2\sqrt{\ve  ||u||_{L^\infty\left(\overline\Omega\times[T_0,T_2]\right)}}\right), 
\end{equation} where $\omega$ is a modulus of continuity of $u$ on $\overline \Omega\times[T_0,T_2].$  
Therefore,  we use  \eqref{eq-2nd-var-p}  and \eqref{eq-dist-bd-omega} to conclude that  for any  $\nu\in T_{y_0}M$ with $|\nu|=1,$ and for small $ r\in\R, \sigma\leq0,$  
\begin{equation*} 
 \begin{split}
 u\left(\exp_{y_0} r\nu,s_0+\sigma\right)  &\geq u(y_0,s_0)+r\left\langle L_{x_0,y_0}\zeta, \nu\right\rangle_{y_0}+\sigma p+\frac{r^2}{2}\left\langle \left(L_{x_0,y_0}\circ A\right)\cdot \nu,\nu\right\rangle_{y_0}\\
 &\quad-r^2\kappa\,\omega\left(2\sqrt{\ve ||u||_{L^\infty\left(\overline\Omega\times[T_0,T_2]\right)}}\right)+o\left(r^2+|\sigma|\right).
\end{split} 
\end{equation*} 
Therefore, Lemma \ref{lem-char-superjet} implies 
$$\left(p,\,L_{x_0,y_0}\zeta,\, L_{x_0,y_0}\circ A-2\kappa\,\omega\left(2\sqrt{\ve  ||u||_{L^\infty\left(\overline\Omega\times[T_0,T_2]\right)}}\right)\,g_{y_0}\right)\,\in \cP^{2,-}u(y_0,s_0).$$
\end{proof}

We recall  from \cite{AFS}  the  intrinsic uniform continuity of the operator with respect to $x, $ which  is  a natural extension of the Euclidean notion of uniform continuity of  the operator with respect to $x.$
\begin{definition}\label{def-in-unif-cont} The operator $F: \Sym TM\to \R$ is said to be {\color{black}  intrinsically uniformly continuous with respect to $x$} if there exists a modulus of continuity $\omega_F:[0,+\infty)\to[0,+\infty)$ with $\omega_F(0+)=0$  such that 
\begin{equation}\label{Hypo2} \tag{H2}
F\left(S\right)-F\left(L_{x,y}\circ S\right)\leq \omega_F\left(d(x,y)\right)
\end{equation}
for  any $S\in\Sym TM_x,$ and    $x,y\in M$ {\color{black} with $d(x,y)<\min\left\{i_M(x),\,i_M(y)\right\}. $}
\end{definition}
We may assume that $\omega_F$ is nondecreasing  on $(0,+\infty).$  Recall  some examples of the  intrinsically uniformly continuous operator from  \cite{AFS}. 

\begin{remark} \label{rmk-AFS}
{\rm 
%\begin{enumerate}[(a)]
%\item  

(a) When $M=\R^n,$  we have $L_{x,y} \circ S\equiv S$ so \eqref{Hypo2} holds.    
%\item 

(b) In general, we   consider the operator $F,$  which depends only on the eigenvalues of $S\in\Sym TM,$ of the form :
\begin{equation}\label{eq-F-eigenvalue}
F(S)=G\left(\,\mathrm{eigenvalues \,\,of } \,\,S \,\right)\quad\mbox{for some $G.$}
\end{equation} 
%we canonically identify the space of symmetric bilinear forms on T Mx with the space of self-adjoint linear mappings from T Mx into T Mx ,
Since $S$ and $L_{x,y}\circ S$ have the same eigenvalues, the operator $F$ satisfies  intrinsic uniform continuity with respect to $x$    (with  $\omega_F\equiv0$). The trace and determinant of $S$  are typical examples   of the  operator satisfying \eqref{eq-F-eigenvalue}. 
%\item  

(c) Pucci's extremal operators $\cM^{\pm}$  satisfy \eqref{eq-F-eigenvalue}, \eqref{Hypo2} and  \eqref{Hypo1}. %$\mathrm{(H1)}.$ 

}
\end{remark}

\begin{lemma}\label{prop-u-u-e-superj-p}
 Under the same assumption  as Proposition \ref{lem-visc-u-u-e-p}, we also assume that  $F$ satisfies \eqref{Hypo1}  and \eqref{Hypo2}. 
For ${\color{black}f\in C\left( \Omega\times(T_0,T_2]\right)},$  let $u\in C\left(\overline\Omega\times[T_0,T_2]\right)$ be a viscosity supersolution of $$F(D^2u)-\p_t u=f\quad \mbox{in $\Omega\times(T_0,T_2] .$}$$   If $0<\ve<\ve_0,$  then the inf-convolution $u_\ve$  (with respect to $\Omega\times(T_0,T_2]$) is a viscosity supersolution of 
  $$F(D^2u_\ve)-\p_t u_\ve=f_\ve\quad\mbox{on $H\times(T_1,T_2]$},$$
  where 
  $\ve_0>0$ is the constant as in Proposition \ref{lem-visc-u-u-e-p}, and 
  $$f_{\ve}(x,t):= \sup_{\overline B_{2\sqrt{m\ve}}(x)\times \left[t-2\sqrt{m\ve},\,\min\left\{t+2\sqrt{m\ve},\,T_2\right\}\right]}f+{\color{black}\omega_F}\left(2\sqrt{m\ve}\right)+2n\Lambda {\color{black}\kappa\,\omega}\left(2\sqrt{m \ve}\right)$$
  for $m:=||u||_{L^{\infty}\left(\overline\Omega\times[T_0,T_2]\right)}.$    
  Moreover,  we have 
   $$F(D^2u_\ve)-\p_tu_\ve\leq f_\ve\quad\mbox{ a.e. in $H\times(T_1,T_2).$}$$
\end{lemma}
\begin{proof}Fix $0<\ve<\ve_0.$ 
Let $\varphi\in C^{2,1}\left(H\times(T_1,T_2]\right)$ be a function such that  $u_\ve-\varphi$ has a local minimum at $(x_0,t_0)\in H\times(T_1,T_2]$ in the parabolic sense. Then we have $$\left(\p_t \varphi(x_0,t_0),\D\varphi(x_0,t_0), D^2\varphi(x_0,t_0)\right)\in \cP^{2,-}u_\ve(x_0,t_0).$$  We apply Proposition \ref{lem-visc-u-u-e-p} to have that 
\begin{equation*}
\left(\p_t\varphi(x_0,t_0),L_{x_0,y_0}\D\varphi(x_0,t_0),L_{x_0,y_0}\circ D^2\varphi(x_0,t_0) -2\kappa \omega\left(2\sqrt{m\ve}\right)\,g_{y_0} \right)\in \cP^{2,-}u(y_0,s_0)
\end{equation*}
for $$y_0:=\exp_{x_0}\left(-\ve\D\varphi(x_0,t_0)\right)\in  \overline B_{2\sqrt{m\ve}}(x_0)\subset\Omega,$$ and  some 
$ s_0\in \left[t_0-2\sqrt{m\ve},\min\left\{t_0+2\sqrt{m\ve},T_2\right\}\right]\subset (T_0,T_2].$ 
Since $u$ is a viscosity supersolution  in $\Omega\times(T_0,T_2]$, we see that
\begin{align*}
f(y_0,s_0)&\geq F\left(L_{x_0,y_0}\circ D^2\varphi(x_0,t_0) -2\kappa\,\omega\left(2\sqrt{m\ve}\right)\,g_{y_0}\right)-\p_t\varphi(x_0,t_0)\\
&\geq  F\left(L_{x_0,y_0}\circ D^2\varphi(x_0,t_0)\right)-n\Lambda\cdot 2\kappa\,\omega\left(2\sqrt{m\ve}\right)- \p_t\varphi(x_0,t_0)\\
&\geq F\left( D^2\varphi(x_0,t_0)\right)-\omega_F\left(d(x_0,y_0)\right)-2n\Lambda \kappa\,\omega\left(2\sqrt{m\ve}\right)-\p_t\varphi(x_0,t_0) 
 \end{align*}
 using  the uniform ellipticity and intrinsic uniform continuity  of $F.$ Thus, we deduce that 
 \begin{align*}
  F\left( D^2\varphi(x_0,t_0)\right)-\p_t\varphi(x_0,t_0)&\leq f(y_0,s_0)+\omega_F\left(d(x_0,y_0)\right)+2n\Lambda \kappa\,\omega\left(2\sqrt{m\ve}\right)\\
  &\leq f_\ve(x_0,t_0).
 \end{align*}
% since $d^2(x_0,y_0)\leq 4m\ve. $ % from Lemma \ref{lem-visc-u-u-e-p}.
 Therefore, $u_\ve$ is a viscosity supersolution of $F(D^2u_\ve)-\p_t u_\ve=f_\ve$ in $H\times(T_1,T_2].$ 
 
According to Lemma \ref{lem-u-e-alexandrov-p}, $u_\ve$ admits the Hessian almost everywhere in $\Omega\times(T_0,T_2)$ satisfying \eqref{eq-u-e-taylor}. For almost every $(x,t)\in\Omega\times(T_0,T_2),$  we use  \eqref{eq-u-e-taylor} and Lemma \ref{lem-char-superjet} to deduce
$$\left(\p_t u_\ve(x,t), \,\D u_\ve(x,t),\,D^2u_\ve(x,t)\right)\in \cP^{2,-}u_\ve(x,t)\cap  \cP^{2,+}u_\ve(x,t).$$ 
Therefore, we conclude that 
$$F(D^2u_\ve)-\p_t u_\ve\leq f_\ve\quad\mbox{ a.e.  in }\,\,\, H\times(T_1,T_2),$$
since  $u_\ve$ is a viscosity supersolution in $H\times(T_1,T_2].$   
\end{proof}

For a viscosity subsolution, we can  employ  similar argument as in   Lemmas \ref{lem-visc-u-u-e-0-p},\ref{lem-u-e-alexandrov-p},  \ref{prop-u-u-e-superj-p}, and Proposition \ref{lem-visc-u-u-e-p} using the sup-convolution: 
\begin{equation*}%\label{eq-def-sup-conv-p}
u^{\ve}(x_0,t_0):=\sup_{(y,s)\in   \overline\Omega\times[T_0,T_2] } \left\{ u(y,s) -\frac{1}{2\ve}\left\{d^2(y, x_0)+|s-t_0|^2\right\}\right\}\quad \mbox{for $(x_0,t_0)\in   \overline\Omega\times[T_0,T_2],$}
 \end{equation*}
under the assumption that  the sectional curvature is bounded from below.

%For a viscosity subsolution, we can  obtain similar results  to Lemmas \ref{lem-visc-u-u-e-0-p},\ref{lem-u-e-alexandrov-p},  \ref{prop-u-u-e-superj-p}, and Proposition \ref{lem-visc-u-u-e-p} using the sup-convolution: 
%\begin{equation*}%\label{eq-def-sup-conv-p}
%u^{\ve}(x_0,t_0):=\sup_{(y,s)\in   \overline\Omega\times[T_0,T_2] } \left\{ u(y,s) -\frac{1}{2\ve}\left\{d^2(y, x_0)+|s-t_0|^2\right\}\right\}\quad \mbox{for $(x_0,t_0)\in   \overline\Omega\times[T_0,T_2],$}
 %\end{equation*} under the assumption that  the sectional curvature is bounded from below.

\section{Parabolic Harnack inequality }\label{sec-PHI}

\subsection{A priori   estimate}\label{subsec-PHI-C2}

In this subsection, we shall  prove Proposition \ref{lem-decay-est-1-step},  which is a main ingredient of a priori Harnack estimate,      by making use of the   ABP-Krylov-Tso type estimate in Lemma \ref{lem-abp-type}.  %assuming the  sectional curvature of $M$ is  bounded from below. 
  We begin  with the   definition of the contact set  for the elliptic case   from \cite{WZ}.
 \begin{definition}
  Let $\Omega$ be a bounded open set in $M$ and let $u\in C(\Omega).$ For a given $a>0$ and a compact set $E\subset M,$ the contact set associated with $u$ of opening $a$ with  vertex set $E$ is defined by
  $$\cA_a(E;\Omega;u):=\left\{x\in\Omega : \exists y\in E \,\,s.t.\,\,\inf_{ \Omega}\left(u+\frac{a}{2}d_y^2\right)=u(x)+\frac{a}{2}d_y^2(x) \right\}.$$
 \end{definition}
The following result 
 is essentially contained  in    \cite[Proof of Theorem 1.2]{WZ}  and \cite[Proof of Lemma 4.1]{Ca};  see also \cite[Proposition 2.5]{CMS} and \cite[Chapter 14]{V}.
 
\begin{lemma}\label{lem-WZ-jac}
Assume that
\begin{equation*}%\label{eq-ricci-bdd}
\Ric \geq -\kappa  \quad \mbox{on $M,$} \quad\mbox{for $\kappa\geq0$}.
\end{equation*} 
Let $\Omega$ be a bounded  open set  in $M$ and $E$ be a  compact set  in $M.$ For $a>0$ and a smooth function  $u $  on $\Omega,$
we define  the map $\tilde\phi:\Omega\to M  $ as
  $$\tilde\phi(x):=\exp_{x}a^{-1}\D u(x).$$
If $x\in\cA_a(E;\Omega;u),$ then we  have the following.   
\begin{enumerate}[(a)]
\item If $y\in E$  satisfies
     $$\displaystyle\inf_{ \Omega}\left(u+\frac{a}{2}d_y^2\right)=u(x)+\frac{a}{2}d_y^2(x),$$
 then $ x\notin \Cut(y),\quad\frac{1}{a}\D u(x) =-d_y(x)\D d_y(x),$ and 
 $$y=\tilde\phi(x)=\exp_{x}a^{-1}\D u(x). $$
\item\begin{equation*}
\Jac\tilde\phi(x)\leq  {\sS}^n\left(\sqrt{\frac{\kappa}{n}}\frac{|\D u|}{a}\right)\left\{\sH\left(\sqrt{\frac{\kappa}{n}}\frac{|\D u|}{a}\right)+\frac{\La u}{na}\right\}^n,
\end{equation*}  
where
$$\sH(\tau):=\tau\coth(\tau),\qquad \sS(\tau):=\sinh(\tau)/\tau,\qquad \tau\geq0.$$
%In particular, if $\Ric\geq0,$ i.e., $k=0,$ we have  $$\Jac\phi(x)\leq    \left( 1+\frac{\La u}{na}\right)^n.$$
\end{enumerate}
\begin{comment}

 if $\cA_a(E;\Omega;u)\subset \Omega,$ then  
we have
\begin{equation}
|E|\leq \int_{\cA_a(E;\Omega;u)} \sS^n\left(\sqrt{\frac{k}{n}}\frac{|\D u|}{a}\right)\left\{\sH\left(\sqrt{\frac{k}{n}}\frac{|\D u|}{a}\right)+\frac{\La u}{na}\right\}^n,
\end{equation}  
where
$$\sH(\tau):=\tau\coth(\tau),\qquad \sS(\tau):=\sinh(\tau)/\tau,\qquad \tau\geq0.$$
In particular, if $\Ric\geq0,$ i.e., $k=0,$ we have 
$$|E|\leq \int_{\cA_a(E;\Omega;u)}  \left( 1+\frac{\La u}{na}\right)^n.$$

\end{comment}
\end{lemma} 
%We remark that if $E\subset \Omega, $ the above lemma holds under the assumption that $\Ric \geq -k$ on  $\Omega$    instead of  \eqref{eq-ricci-bdd}.

Now  we define a parabolic version of the contact set associated with  $u$ which  contains a point $(\overline x,\overline  t)\in M\times\R,$ where a continuous function $u$ has a tangent, concave   paraboloid $\displaystyle-\frac{a}{2}d_y^2(x) + bt + C $ (for some $a,b>0$ and $C$)  at $(\overline x,\overline t)$  from below in a parabolic neighborhood of  $(\overline x,\overline t),$ i.e., in $K_r(\overline x,\overline t)$ for some $r>0.$

%a concave paraboloid $\displaystyle-\frac{a}{2}d_y^2(x) + bt + C $ (for some $a,b>0$ and $C$) touches $u$ from below at $(\overline x,\overline t)$  in a parabolic neighborhood of  $(\overline x,\overline t),$ i.e., in $K_r(\overline x,\overline t)$ for some $r>0.$

  \begin{definition}\label{def-contact-p}
  Let $\Omega$ be a bounded open set in $M $ and let $u\in C(\Omega\times(0,T])$ for   $T>0.$ For  given $a,b>0$ and a compact set $E\subset M,$ the parabolic contact set associated with $u$   is defined by
  \begin{align*}
  &\cA_{a,b}(E;\Omega\times(0,T];u)\\&:=\left\{(x,t)\in\Omega\times(0,T] : \exists y\in E \,\,s.t.\,\,\inf_{ \Omega\times(0,\,t]}\left(u(z,\tau)+\frac{a}{2}d_y^2(z)-b\tau \right)=u(x,t)+\frac{a}{2}d_y^2(x)-bt\right\}.\end{align*}
 \end{definition}

 As in \cite{KKL},  for $u\in C^{2,1}\left(\Omega\times(0,T]\right),$
  we define  the map $\phi:\Omega\times(0,T]\to M  $ by 
  $$\phi(x,t):=\exp_{x}a^{-1}\D u(x,t),$$
and  define the parabolic normal map $\Phi:\Omega\times(0,T]\ra M\times \R$  by
$$\Phi(x,t):=\left(\phi(x,t), -\frac{1}{2}d^2\left(x, \phi(x,t)\right)-a^{-1}\left\{u(x,t)-bt\right\}\right)  . $$

\begin{lemma}\label{lem-Jac-normal_map}
Assume that
\begin{equation*}%\label{eq-ricci-bdd}
\Ric \geq -\kappa  \quad \mbox{on  $M$},\quad\mbox{for $\kappa\geq0$}.
\end{equation*} 
Let $\Omega$ be a bounded  open set  in $M,$   and let  $u$ be a smooth function on $\Omega\times(0,T]$ for $T>0.$   For any compact set $E\subset M,$ $a,b>0, $  and {\color{black}$0<\tilde\lambda\leq1,$} %if $\cA_{a,b}(E;\Omega\times(T_1,T_2];u)\subset \Omega\times(T_1,T_2],$ then 
 we have that  if  $(x,t)\in \cA_{a,b}(E;\Omega\times(0,T];u),$  then
%\begin{equation}
%\Jac\Phi(x,t)\leq \frac{1}{{(n+1)}^{n+1}}\left\{\cS\left(\sqrt{\frac{\kappa}{n}}\frac{|\D u|}{a}\right)\frac{\La u}{a}-\frac{u_t}{a}+n \cC\left(\sqrt{\frac{\kappa}{n}}\frac{|\D u|}{a}\right)+\frac{b}{a}\right\}^{n+1},
%\end{equation}  
\begin{equation*}
\Jac\Phi(x,t)\leq \frac{1}{{(n+1)}^{n+1}}\sS^{n+1}\left(\sqrt{\frac{\kappa}{n}}\frac{|\D u|}{a}\right)\left[\left\{\frac{\tilde\lambda\La u-\p_tu}{\tilde\lambda a }+\frac{b}{\tilde\lambda a} +n \sH\left(\sqrt{\frac{\kappa}{n}}\frac{|\D u|}{a}\right)\right\}^+\right]^{n+1},
\end{equation*}  
%$$ \frac{1}{{(n+1)}^{n+1}}\left[\sS\left(\sqrt{\frac{k}{n}}\frac{|\D u|}{a}\right)\left\{\frac{\lambda\La u-\p_tu+b}{\lambda a } +n \sH\left(\sqrt{\frac{k}{n}}\frac{|\D u|}{a}\right)\right\}\right]^{n+1}$$
where
$$\sH(\tau)=\tau\coth(\tau),\qquad \sS(\tau)=\sinh(\tau)/\tau,\qquad \tau\geq0.$$
%In particular, if $\Ric\geq0,$ i.e., $k=0,$ and if $\lambda=1,$  we have that for $(x,t)\in \cA_{a,b}(E;\Omega\times(T_1,T_2];u),$ 
%$$\Jac\Phi(x,t)\leq  \frac{1}{(n+1)^{n+1}} \left( \frac{\La u- u_t}{a}+\frac{b}{a}+n\right)^{n+1}.$$

\end{lemma}
\begin{proof}
Let $(x,t)\in \cA_{a,b}:=\cA_{a,b}(E;\Omega\times(0,T];u)\subset \Omega\times(0,T].$  From the definition of the parabolic  contact set, there exists a  vertex $y\in E$ such that 
$$\inf_{\Omega\times(0,t]}\left(u(z,\tau)+\frac{a}{2}d_y^2(z)-b\tau \right)=u(x,t)+\frac{a}{2}d_y^2(x)-bt.$$
According to Lemma \ref{lem-WZ-jac},  we have that 
 $$ y=\phi(x,t)=\exp_{x}a^{-1}\D u(x,t),  \quad x\notin \Cut(y),\quad\mbox{and}\quad\frac{1}{a}\D u(x,t) =-d_y(x)\D d_y(x)$$ 
since  
$$\inf_{\Omega }\left(u(z,t)+\frac{a}{2}d_y^2(z) \right)=u(x,t)+\frac{a}{2}d_y^2(x).$$ 
We notice that $D^2\left( u+\frac{a}{2}d_y^2\right)(x,t)\geq0$ and $\p_tu(x,t)-b\leq 0.$
Now we set $$\tilde\phi:=\phi(\cdot,t)\,:\Omega\owns z\,\mapsto \,\exp_{z}a^{-1}\D u(z,t)\in M$$ to obtain  from Lemma \ref{lem-WZ-jac} that
\begin{equation}\label{eq-jac-tildephi}
\Jac\tilde \phi(x)\leq \sS^n\left(\sqrt{\frac{\kappa}{n}}\frac{|\D u|}{a}\right)\left\{\sH\left(\sqrt{\frac{\kappa}{n}}\frac{|\D u|}{a}\right)+\frac{\La u}{na}\right\}^n(x,t).
\end{equation}

 \begin{comment}
 We notice that $D^2_x\left(a^{-1}u+{d_y^2}/{2}\right)(x,t)\geq0$ and $u_t(x,t)-b\leq 0.$
According to the proof of \cite[Theorem 1.2]{WZ}, we have that  $$\,x\notin \Cut(y), \,y=\phi(x,t), \,\,\,\mbox{and}\,\,\,a^{-1}\D u(x,t) =-d_y(x)\D d_y(x)$$ 
since we also have  that
$$\inf_{\Omega }\left(u(z,t)+\frac{a}{2}d_y^2(z) \right)=u(x,t)+\frac{a}{2}d_y^2(x).$$
Thus it follows  from Lemma \ref{lem-WZ-jac} that
$$\Jac\phi(\cdot,t)(x)\leq \sS^n\left( \frac{\sqrt{k}|\D u|}{a}\right)\left\{\sH\left( \frac{\sqrt{k}|\D u|}{a}\right)+\frac{\La u}{na}\right\}^n$$ since $\sS$ and $\sH$ are nondecreaing for $\tau\geq0.$

%Now we shall  estimate the Jacobian of  $\Phi$ at $(x,t).$ 
%We may assume that $\D_x v(x,t)\neq 0$, which is equivalent to $x\neq y$.
%For $(\xi,\sigma)\in T_xM\times \R\bs \{(0,0)\},$ and let $\gamma=(\gamma_1,\gamma_2)$ be the geodesic  with  $\gamma (0)=(x,t)$ and $\gamma'(0)=(\xi,\sigma)$. We note that $\gamma_1(\tau)=\exp_x\tau\xi$ and $\gamma_2(\tau)=t+\sigma\tau$.  
\end{comment}

By a  simple calculation, we have that for $(\xi,\sigma)\in T_xM\times \R\bs \{(0,0)\},$
\begin{align*}
d\Phi(x,t)\cdot(\xi,\sigma)&=\left(d\tilde \phi\cdot\xi+\sigma\frac{\p\phi}{\p t}, \,\,\, -\left\langle\D\left(d_x^2/2\right)(y),  \,d\tilde\phi\cdot\xi+\sigma\frac{\p\phi}{ \p t}\right\rangle_y-a^{-1}\sigma\left( \p_tu-b\right)\right),
\end{align*}
where $\frac{\p \phi}{\p t}(x,t)=\left.\frac{d}{d\tau}\right|_{\tau=0}\phi(x,t+\tau)\in T_{y}M$ and we used $\D\left(d_y^2/2\right)(x)=-a^{-1}\D  u(x,t).$  
To compute the Jacobian of $\Phi$,  we introduce an orthonormal basis $\{e_1, \cdots, e_n\}$ of $T_xM$ and an orthonormal basis $\{\overline e_1,\cdots,\overline e_n\}$ of $T_yM=T_{\phi(x,t)}M$. By setting for $i,j=1,\cdots,n$, 
\begin{align*}
A_{ij}:=\left\langle \overline e_i, \,d\tilde\phi \cdot e_j \right\rangle,\,\,
b_i:=\left\langle \overline e_i, \,\frac{\p\phi}{ \p t} \right\rangle, \,\,
\mbox{and}\,\, c_i:=\left\langle \overline e_i, \,\D\left(d_x^2/2\right)(y)\right\rangle,
\end{align*} the Jacobian matrix of  $\Phi$ at $(x,t)$   is 
\begin{equation*}\displaystyle
 \left(
\begin{array}{cc}
 A_{ij} &   b_i\\% \hline
-c_{k}A_{kj} &-c_kb_k+a^{-1}\left(b-\p_tu\right)
\end{array}
\right).
\end{equation*}
Using  the row operations and \eqref{eq-jac-tildephi},  we   deduce that
\begin{align*}\displaystyle
\Jac\Phi(x,t)&=\left|\det \left(
\begin{array}{cc}
 A_{ij} &   b_i\\% \hline
0 & a^{-1}\left(b-\p_tu\right)
\end{array}
\right)\right| =a^{-1}\left(b-\p_tu\right)\Jac\tilde\phi(x)\\
&\leq a^{-1}\left(b-\p_tu\right)\sS^n\left(\sqrt{\frac{\kappa}{n}}\frac{|\D u|}{a}\right)\left\{\sH\left(\sqrt{\frac{\kappa}{n}}\frac{|\D u|}{a}\right)+\frac{\La u}{na}\right\}^n(x,t),
\end{align*}
where we note that $\left(b-\p_tu\right)(x,t)\geq0$ and  $\Jac\tilde\phi(x)\geq0.$ %$\displaystyle\left\{\sH\left( \sqrt{\frac{k}{n}}\frac{|\D u|}{a}\right)+\frac{\La u}{na}\right\}\geq0$ at $(x,t).$
According to the geometric and arithmetic means inequality,
we conclude that
\begin{align*}
\Jac\Phi(x,t)&\leq \frac{1}{{(n+1)}^{n+1}}\left[n\sS\left(\sqrt{\frac{\kappa}{n}}\frac{|\D u|}{a}\right)\left\{\sH\left(\sqrt{\frac{\kappa}{n}}\frac{|\D u|}{a}\right)+\frac{\La u}{na}\right\}
+\frac{b-\p_tu}{a}\right]^{n+1}\\
&=\frac{1}{{(n+1)}^{n+1}}\left[\sS\left(\sqrt{\frac{\kappa}{n}}\frac{|\D u|}{a}\right)\left\{\frac{\tilde\lambda \La u-\p_tu+b}{\tilde\lambda a} +n\sH\left(\sqrt{\frac{\kappa}{n}}\frac{|\D u|}{a}\right)\right\}\right.\\ 
&\qquad\qquad\qquad+\left.\left\{ 1-\frac{1}{\tilde\lambda}\sS\left(\sqrt{\frac{\kappa}{n}}\frac{|\D u|}{a}\right)\right\} \frac{b-\p_tu}{ a }\right]^{n+1}\\
&\leq \frac{1}{{(n+1)}^{n+1}}\left[\sS\left(\sqrt{\frac{\kappa}{n}}\frac{|\D u|}{a}\right)\left\{\frac{\tilde\lambda\La u-\p_tu+b}{\tilde\lambda a } +n \sH\left(\sqrt{\frac{\kappa}{n}}\frac{|\D u|}{a}\right)\right\}\right]^{n+1}
 \end{align*}
since $\left(b-\p_tu\right)(x,t)\geq0$ and $\sS(\tau)=\sinh(\tau)/\tau\geq 1\geq\tilde\lambda$ for all $\tau\geq0.$
 \end{proof}
 
\begin{remark}\label{rmk-li-yau}
{\rm
As mentioned in the introduction,  we can obtain the (locally) uniform Harnack inequality for  the heat equation on manifolds  with  a lower bound  $-\kappa$ ($\kappa\geq0$) of   Ricci curvature, which was    established earlier in  \cite{LY, Y1,Y2,BQ} using Li-Yau type gradient estimates.    Our   proof   follows closely the approach given  in \cite{KKL}  by making use of 
Lemma \ref{lem-Jac-normal_map} and a (locally) uniform volume doubling property   with the help of  the Laplacian comparison   replacing  the Hessian comparison (Lemma \ref{lem-hess-dist-sqrd}),  
%we   follow the approach in \cite{KKL} to  
where the uniform  constant $C_H>0$ in \eqref{eq-PHI} depends only on $n$ and $\sqrt{\kappa} R_0$; see \cite{WZ} for the elliptic case.   Hence, in the case when $\kappa=0, $ this implies    a global Harnack inequality for heat equation on manifolds with nonnegative Ricci curvature which was   proved first  by Li and Yau \cite{LY} and recently in \cite{KKL}. % Instead of  a direct calculation of the Jacobian of the normal map in \cite{KKL}, we use a standard theory of Jacobi fields so   
Lastly, we mention that  a global  Harnack estimate  in \cite{KKL} for linear, uniformly  parabolic operators with bounded measurable coefficients does not follow   from our approach since we replace    a direct calculation of the Jacobian of the normal map  by an estimation using a standard theory of Jacobi fields. 
}
\end{remark}

Assuming the  sectional curvature of $M$ to be bounded from below,  we have the   ABP-Krylov-Tso type estimate  regarding the Pucci operator as below, which 
will play a key role to estimate sublevel sets of $u$ in Proposition \ref{lem-decay-est-1-step}.

%\begin{equation}\label{cond-sec-curv}
%\Sec(X,Y)\geq -K,\quad\forall X,Y\in TM.
%\end{equation}

\begin{lemma}\label{lem-abp-type}
Assume that 
  $$\Sec\geq  -\kappa\quad\mbox{on $M,$}\quad\mbox{for  $\kappa\geq0.$}$$ 
  Let   $R_0>0$ and  $0<\eta<1.$ %and  {\color{red} $0<\lambda\leq n.$}
%\marginpar{\color{blue}condition on $M  :$\\ $Jac\exp_{(x,t)}$,   \\$D^2d_y$ }
For $z_0\in M$,  and $0<R\leq R_0,$    let $u$ be a smooth function in $K_{\alpha_1 R,\, \alpha_2R^2}(z_0,0) \subset M\times\R$ such that 
\begin{equation}\label{cond-abp}
u\geq 0\quad\mbox{in}\quad K_{\alpha_1R,\, \alpha_2R^2}(z_0,0)\bs K_{\beta_1R,\,\beta_2R^2}(z_0,0) \quad\mbox{
and} \quad\inf_{K_{2R}(z_0,0)}u\leq1,\end{equation}
where $\alpha_1:=\frac{11}{\eta}$, $\alpha_2:=4+\eta^2+\frac{\eta^4}{4}$, $\beta_1:=\frac{9}{\eta}$, and $\beta_2:=4+\eta^2$.  
  Then we have
 \begin{align}\label{eq-abp-type}
|B_R(z_0)|\cdot R^2\leq  \int_{\{u\leq M_\eta\}\cap K_{\beta_1R,\,\beta_2R^2}(z_0,0)}{\sS^{n+1} } \cdot\left[\left\{\frac{R^2}{2\lambda}\left\{{\cM^- (D^2u)-\p_tu}\right\} +\frac{6}{\lambda\eta^2} +\left(n+1\right)\frac{\Lambda}{\lambda} \sH \right\}^+\right]^{n+1},
\end{align}
where
 the constant  $M_\eta>0$ depends only on $\eta>0,$ and $$\sS:=\sS\left(2\al_1\sqrt{\kappa} R_0\right),\quad \sH:=\sH\left(2\al_1\sqrt{\kappa} R_0\right)$$  for  $ \sS(\tau)=\sinh(\tau)/\tau,$  and  $\sH(\tau)=\tau\coth(\tau).$  
% \item[(ii)] If $$\Ric\geq -(n-1)\kappa\quad\mbox{on}\,\,\,K_{\beta_1R,\,\beta_2R^2}(z_0,0),\,\,\,\, \kappa\geq0,\quad\mbox{and}\quad\lambda=\Lambda=1,$$  \eqref{eq-abp-type} holds with $\sS:=\sS\left(2\beta_1\sqrt{\kappa} R_0\right) , \,\sH:=\sH\left(2\beta_1\sqrt{\kappa} R_0\right).$
% \begin{align}\label{eq-abp-type}
%|B_R(z_o)|\cdot R^2\leq C(\eta,n,\lambda)  \int_{\{u\leq M_\eta\}\cap K_{\beta_1R,\,\beta_2R^2}(z_o,0)}{\cS^{n+1}\left(2\sqrt{K} \beta_1R_0\right)} \left[\left\{{R^2}\left({\cM^- (D^2u)-u_t}\right) + {n\Lambda} \cH\left( 2\sqrt{K}\beta_1R_0\right)+1\right\}^+\right]^{n+1}
%\end{align}where the constants  $M_\eta, C_\eta>0$ depend only on $\eta>0$. 
%\end{itemize}
   \end{lemma}

\begin{figure}%[b] 
  \centering
    \includegraphics[width=0.92\textwidth]{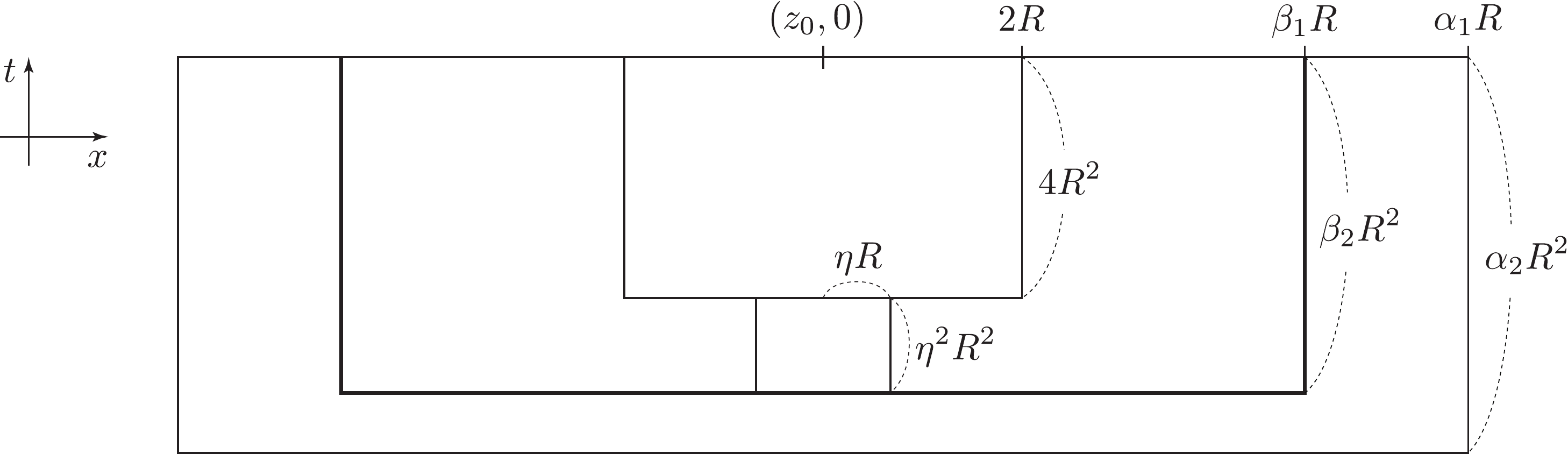}
      \caption{  $\alpha_1:=\frac{11}{\eta},\,\, \alpha_2:=4+\eta^2+\frac{\eta^4}{4}, 
      \,\,\beta_1:=\frac{9}{\eta},\,\,\beta_2:=4+\eta^2$}\label{boxes}
\end{figure}

\begin{proof}

%We set $a:=\frac{2}{R^{2}}$ and $b:=\frac{12}{\eta^2R^2}, $ and  
%As in  \cite[Lemma 3.2]{KKL}, 
We  consider the parabolic contact set
$$\cA_{a,b}\left(\overline B_R(z_0);  K_{\al_1R,\,\al_2R^2}(z_0,0);u\right)\quad\mbox{ for  $\,a:=\frac{2}{R^{2}}\,$ and $\,b:=\frac{12}{\eta^2R^2},$} $$  which will be  denoted by $\cA$ for simplicity.
% we begin with  a similar argument of  the proof of \cite[Lemma 3.2]{KKL}. 
As in the proof of  \cite[Lemma 3.2]{KKL},  for any $\overline y\in B_R(z_0)$,  we define  
$$w_{\overline y}(x,t):=\frac{1}{2}R^2u(x,t)+\frac{1}{2}d_{\overline y}^2(x)-C_{\eta}t,\quad \,C_\eta:=\frac{b}{a}=\frac{6}{\eta^2}.$$
%with $C_\eta:=\frac{6}{\eta^2}$. % \frac{1}{2\eta}\left(3+\frac{9}{\eta}\right)$.
From the assumption \eqref{cond-abp},  we see that 
$$\inf_{K_{2R}(z_0,0)}w_{\overline y}\leq% R^2+\frac{1}{\eta}(3R)^2+C_\eta(2R)^2=
\left( 5+\frac{24}{\eta^2}\right)R^2=:A_\eta R^2, $$
and $$w_{\overline y}\geq\left(6+\frac{24}{\eta^2}\right)R^2=(A_\eta +1)R^2\quad\mbox{on}\,\,\,K_{\alpha_1R,\,\alpha_2R^2}(z_0,0)\bs K_{\beta_1R,\,\beta_2R^2}(z_0,0). $$
Then we deduce  that  
for any $\displaystyle \left(\overline y, \overline h\right)\in B_R(z_0)\times\left(A_\eta R^2,(A_\eta +1)R^2\right)$,  
there exists a time  $\displaystyle \overline t\in\left(-\beta_2R^2,0\right)$ such that 
$$\overline h=\inf_{B_{\al_1R}(z_0)\times\left(-\al_2R^2,\overline t\right]}w_{\overline y}\, =\,w_{\overline y}\left(\overline x,\overline t\right),$$
where  the infimum is achieved at an interior point $\overline x\in B_{\beta_1 R}(z_0).$
This means that  $(\overline x,\overline t)$ is a parabolic contact point, i.e.,  $(\overline x,\overline t)\in\cA .$  
According to Lemma \ref{lem-WZ-jac}, we observe that 
$$\overline y=\exp_{\overline x}\left(\frac{1}{2}R^2\,\D u\left(\overline x, \overline t\right)\right),\quad\mbox{and }\quad \overline x\not\in\Cut(\overline y). $$

Now,  we define  the map $\phi:K_{\alpha_1R,\,\alpha_2R^2}(z_0,0)\to M  $ as 
  $$\phi(x,t):=\exp_x \left(\frac{1}{2}R^2\D u(x,t)\right),$$
and  the map $\Phi:K_{\alpha_1R,\,\alpha_2R^2}(z_0,0)\to M\times\R$    as
\[
\Phi(x,t):=\left(\phi(x,t), -\frac{1}{2}d^2\left(x,\phi(x,t) \right)-\frac{1}{2}R^2u(x,t)+C_\eta t\right).
\]
% that contains a point $(x,t),$ where a concave paraboloid $-\frac{1}{2} d_{y}^2(x)+C_\eta t+C $ (for some $C$) touches $\frac{1}{2}R^2u$ from below. 
We also define  
\begin{align*}
\tilde\cA:=\left\{(x,t)\in K_{\beta_1R,\,\beta_2R^2}(z_0,0)\,:\, \exists y\in B_{R}(z_0)\,\,\,\mbox{s.t.}\,\,\,\displaystyle w_{y}(x,t)=\inf_{B_{\al_1R}(z_0)\times\left(-\al_2R^2,t\right]}w_y\leq(A_\eta+1 )R^2  \,\right\}.\end{align*}
According to  the argument above,  we have proved that for any $(y,s)\in B_R(z_0)\times\left(-(A_\eta +1)R^2,-A_\eta R^2\right)$,  there exists a point $(x,t)\in \tilde\cA$ such that
$(y,s)=\Phi(x,t),$  that is,    $$B_R(z_0)\times\left(-(A_\eta +1)R^2,-A_\eta R^2\right)\subset \Phi(\tilde\cA) .$$  Thus, the area formula provides
\begin{align}\label{eq-abp-area}
|B_R(z_0)|\cdot  R^2\leq \int_{M\times\R} \cH^0\left[\tilde \cA\cap\Phi^{-1}(y,s)\right] dV(y,s)= \int_{\tilde\cA}\Jac \Phi (x,t)dV(x,t).%\int_{\{u\leq 11\}\cap Q_{6R}}\Jac\Phi,
\end{align}
We note that   \begin{align}\label{eq-abp-contact-set}
\tilde\cA \subset \cA\cap K_{\be_1R,\be_2R^2}(z_0,0)\cap \{u\leq M_\eta\}
\end{align}
 for $M_\eta:=2(A_\eta +1)$  since $\frac{1}{2}R^2u(x,t)\leq w_y(x,t) \leq (A_\eta +1)R^2$ for $(x,t)\in\tilde\cA.$ 

Next, we   claim that for $(x,t)\in\cA,$
\begin{equation}\label{eq-abp-jac-est}
\begin{split}
\Jac\Phi(x,t)&\leq\frac{1}{(n+1)^{n+1}} \sS^{n+1}\left( 2\al_1\sqrt{{\kappa}}\,R_0 \right)\left[\left\{\frac{nR^2}{2\lambda}\left({\frac{\lambda}{n}\La u-\p_tu}\right) +\frac{6n}{\lambda \eta^2}+n \sH\left (2\al_1\sqrt{ {\kappa} } \,R_0\right)\right\}^+\right]^{n+1}.
\end{split}\end{equation}
 From Lemma \ref{lem-WZ-jac},  if $(x,t)\in\cA,$  then we have
\begin{equation*}
\frac{R^2}{2}\D u(x,t) =-d_y(x)\D d_y(x)\quad\mbox{for }\,\,\,
y:=\phi(x,t)\in \overline B_R(z_0)%=\exp_x\left(\frac{R^2}{2}\D u(x,t)\right)
;\quad x\notin \Cut(y),
\end{equation*}
and hence
 \begin{equation}\label{eq-gradient-contact-pt}
\frac{R^2}{2}|\D u(x,t)|=d_y(x)\leq d(y,z_0)+d(z_0,x)\leq R+\al_1R\leq 2\al_1R_0.
\end{equation}
%from the proof of \cite[Theorem 1.2]{WZ},    
%we have that for $(x,t)\in \cA,$
%\begin{equation}\label{eq-gradient-contact-pt}
%\frac{R^2}{2}|\D_x u(x,t)|=d_y(x)\leq d(y,z_0)+d(z_0,x)\leq 2\beta_1R\leq 2\beta_1R_0.
%\end{equation}%and hence $$\sqrt{\frac{\kappa}{n}}\frac{R^2|\D_x u|}{2}\leq \sqrt{K}2\beta_1R\leq 2\sqrt K \beta_1 R_0.$$    
Using Lemma \ref{lem-Jac-normal_map} (with $\tilde\lambda=\lambda/n$) and \eqref{eq-gradient-contact-pt}, we deduce that for $(x,t)\in \cA,$
\begin{equation*}%\label{eq-abp-jac-est}
\begin{split}
\displaystyle(n+1)&\Jac\Phi(x,t)^{\frac{1}{n+1}}\\&\leq\sS \left(\sqrt{\frac{(n-1)\kappa}{n}}\frac{R^2|\D u|}{2}\right)\left\{\frac{nR^2}{2\lambda}\left({\frac{\lambda}{n}\La u-\p_tu}\right) +\frac{6n}{\lambda \eta^2} +n \sH\left(\sqrt{\frac{(n-1)\kappa}{n}}\frac{R^2|\D u|}{2}\right)\right\}^+ \\
 &\leq 
\sS\left( 2\al_1\sqrt{{\kappa}}\,R_0 \right)\left\{\frac{nR^2}{2\lambda}\left({\frac{\lambda}{n}\La u-\p_tu}\right) +\frac{6n}{\lambda \eta^2}+n \sH\left (2\al_1\sqrt{ {\kappa} } \,R_0\right)\right\}^+,
\end{split}\end{equation*}
 since   $\sH(\tau)$ and $\sS(\tau)$ are nondecreasing for  $\tau\geq0.$   This proves \eqref{eq-abp-jac-est}.

Lastly, %we replace the Laplacian by the Pucci's operator $\cM^-$ in \eqref{eq-abp-jac-est} so 
we shall show that  for $(x,t)\in \cA,$
\begin{equation}\label{eq-abp-la-pucci}
\frac{\lambda}{n} \La u\leq  \cM^-(D^2u)+\frac{2n\Lambda}{R^2}\sH\left( 2\al_1\sqrt{\kappa}\,R_0\right). 
\end{equation}
Indeed,  for $(x,t)\in \cA, $  %Recalling Lemma \ref{lem-WZ-jac} again for $\tilde\phi:=\phi(\cdot,t)$, we have 
we recall Lemma \ref{lem-WZ-jac} again  to see 
$$D^2\left(u+\frac{1}{R^2}d_y^2\right)(x,t)\geq0 \quad\mbox{  for $y:=\phi(x,t)$} ;\quad x\notin \Cut(y),$$ 
i.e., the Hessian of $R^2 u+d_y^2$ at $(x, t)$ is positive semidefinite.  
From  Lemma \ref{lem-hess-dist-sqrd} and \eqref{eq-gradient-contact-pt}, it follows that  
$$D^2u(x,t)\geq -\frac{2}{R^2}D^2\left(\frac{1}{2}d_y^2\right)(x)\geq -\frac{2}{R^2}\sH\left(\sqrt{\kappa}\,d_y(x)\right)\,g_x  \geq -\frac{2}{R^2}\sH\left( 2\al_1\sqrt{\kappa}\,R_0\right)g_x.$$
Let $\mu_1$ be the largest eigenvalue of $D^2 u(x,t).$
If  $\mu_1\geq0,$  then  we have
\begin{align*}
\cM^-(D^2u(x,t))%&\geq \lambda e_1-(n-1)\Lambda \frac{2}{R^2}\sH\left(\sqrt{K}d_y(x)\right)\\
&\geq \lambda \mu_1-(n-1)\Lambda \frac{2}{R^2}\sH\left( 2\al_1\sqrt{\kappa}R_0\right)\\
&\geq \frac{\lambda}{n} \La u-n\Lambda \frac{2}{R^2}\sH\left( 2\al_1\sqrt{\kappa}R_0\right).
\end{align*} 
If $\mu_1<0,$ then we have
\begin{align*}
\cM^-(D^2u(x,t))=\Lambda \La u &\geq -n\Lambda \frac{2}{R^2}\sH\left( 2\al_1\sqrt{\kappa}R_0\right)\geq \frac{\lambda}{n} \La u-n\Lambda \frac{2}{R^2}\sH\left( 2\al_1\sqrt{\kappa}R_0\right),
\end{align*} 
 which proves  \eqref{eq-abp-la-pucci}  for $(x,t)\in\cA.$ 
Therefore, the ABP-Krylov-Tso type estimate  \eqref{eq-abp-type} follows from \eqref{eq-abp-area}, \eqref{eq-abp-contact-set} \eqref{eq-abp-jac-est}  and \eqref{eq-abp-la-pucci}.
%since $\tilde\cA\subset \cA\cap\left\{u\leq M_\eta\right\}\subset  \left\{u\leq M_\eta\right\}\cap K_{\beta_1R,\,\beta_2R^2}(z_0,0).$ 
\end{proof}

%%%%%%%%%%%%%%%%%%%%%%%%%%%%%%%%%%%%%%%%%%%%%%%%%%%%%%%%%%%%%%%%%%%%%%%%%

As in \cite{KKL}, we modify the barrier function of \cite{W}  to construct a barrier function in the Riemannian setting.  First, we fix some constants that will be used frequently (see Figure \ref{boxes}); for a given $0<\eta<1$, 
$$\alpha_1:=\frac{11}{\eta},\,\alpha_2:=4+\eta^2+\frac{\eta^4}{4},\, \beta_1:=\frac{9}{\eta}\,\,\,\,\mbox{and}\,\,\,\beta_2:=4+\eta^2. $$   
\begin{lemma}\label{lem-barrier} %Suppose that $M$ satisfies the condition %\eqref{cond-M-1},  \eqref{cond-sec-curv}. 
  Assume that 
  $$\Sec\geq -\kappa\quad\mbox{on $M,$}\quad\mbox{for $\kappa\geq0. $}$$ 
Let $R_0>0$ and $0<\eta<1$.  For  $z_0\in M,$ and   $0<R\leq R_0, $ 
%  or $$\Ric\geq -(n-1)k\quad\mbox{on}\,\,\,B_{\beta_1R}(z_0),\,\,\,\, k\geq0,\quad\mbox{and}\quad\lambda=\Lambda=1.$$
 %\begin{itemize}
 %\item[(i)] Assume that $$\Sec\geq -K\quad\mbox{on}\,\,\,K_{\beta_1R,\,\beta_2R^2}(z_0,0),\quad K\geq0.$$
 there exists a continuous function $v_{\eta}(x,t)$ in $\displaystyle K_{\alpha_1R,\,\alpha_2R^2}(z_0,\beta_2R^2) %:=B_{\frac{11}{\eta}R}(z_o)\times\left((-2+\eta^2)R^2,(4+\eta^2)R^2\right)
 $,  which is smooth in $\left(M\bs \Cut (z_0)\right)\cap K_{\alpha_1R,\,\alpha_2R^2}(z_0,\beta_2R^2),$ such that  %\marginpar{\color{red} Please provide a proof of this lemma.}
\begin{enumerate}[(a)]
\item $v_{\eta}(x,t)\geq 0$ in $K_{\alpha_1R,\,\alpha_2R^2}(z_0, \beta_2R^2)\,\bs\,  K_{\beta_1R,\,\beta_2R^2} (z_0,\beta_2R^2)$, 
\item $v_{\eta}(x,t)\leq 0$ in $K_{2R}(z_0,\beta_2R^2)$,
\item {  $R^2\left\{\cM^+(D^2 v_{\eta})-\p_t v_\eta\right\}+\frac{12}{\eta^2}+2(n+1)\Lambda\sH\left( 2\al_1\sqrt{\kappa}\,R_0\right)\leq0 $} a.e. in $K_{\beta_1R, \,\beta_2R^2}(z_0, \beta_2R^2)\bs  K_{\frac{\eta}{2} R} (z_0, \frac{\eta^2}{4}R^2), $
\item $R^2\left\{\cM^+(D^2 v_{\eta})-\p_t v_\eta\right\}\leq C_{\eta}\quad a.e. $ in $K_{\beta_1R,\, \beta_2R^2}(z_0, \beta_2R^2)$,
\item  $v_{\eta}(x,t)\geq -C_{\eta}$ in $K_{\alpha_1R,\,\alpha_2R^2}(z_0,  \beta_2R^2)$,
\end{enumerate}
where   $\sH(\tau)=\tau\coth(\tau),$  
 and  the constant $C_{\eta}>0$  depends only on $\eta, n,\lambda,\Lambda, \sqrt{\kappa}R_0$ (independent of  $R$ and $z_0$).
%\item[(ii)]  If $$\Ric\geq -(n-1)\kappa\quad\mbox{on}\,\,\,K_{\beta_1R,\,\beta_2R^2}(z_0,0),\,\,\,\, \kappa\geq0,\quad\mbox{and}\quad\lambda=\Lambda=1,$$  there exists  $v_\eta$ satisfying $(a)-(e)$ with   the constant $C_{\eta}>0,$ that  depends only on $\eta, n,\lambda,\Lambda, \sqrt{\kappa}R_0.$
%\end{itemize}
\end{lemma}
\begin{proof}% Fix $0<\eta<1$.  
As in   \cite[Lemma 3.22]{W} and \cite[Lemma 4.1]{KKL}, we consider  $$h(s,t):=-Ae^{-mt}\left(1-\frac{s}{\beta_1^2}\right)^l\frac{1}{(4\pi t)^{n/2}}\exp\left(-\alpha\frac{s}{t}\right)\quad\mbox{for $t>0$,}$$ 
 and   define 
  $$\vp_\eta(s,t):= h(s,t)+\widetilde C t\quad\mbox{in $[0,\beta_1^2]\times[0,\beta_2]\bs [0,\frac{\eta^2}{4}]\times[0,\frac{\eta^2}{4}]$, }$$ where $\widetilde C:=12/\eta^2+2(n+1)\Lambda\sH\left( 2 \al_1\sqrt{\kappa}R_0\right),$  and  the positive constants $A, m, l, \al $ (depending only on  $\eta,n,\lambda,\Lambda, \sqrt{\kappa}R_0$) will be chosen
later.  
  Extending $\vp_\eta$ smoothly   in $[0, \alpha_1^2]\times[-\frac{\eta^4}{4}, \beta_2]$ to satisfy $(a)$ and $(e),$ we define
$$v_\eta(x,t) :=\vp_\eta\left(\frac{d^2_{z_0}(x)}{R^2},\frac{t}{R^2}\right) \quad\mbox{for }\,\,\,\,(x,t)\in K_{\alpha_1R,\,\alpha_2R^2}(z_0, \beta_2R^2) , $$ where $d_{z_0}$ is the  distance function to $z_0.$ We may assume that $\vp_\eta(s,t)$ is nondecreasing with respect to $s$ in $[0, \alpha_1^2]\times[-\frac{\eta^4}{4}, \beta_2].$

We recall that  
\begin{align*}
\left\langle D^2\left(d_{z_0}^2/2\right)(x)\cdot \xi,\xi\right\rangle=\left\langle d_{z_0}D^2d_{z_0}(x)\cdot \xi,\xi\right\rangle+\left\langle\D d_{z_0}(x),\xi\right\rangle^2,\quad\forall\xi\in T_xM,\,\,\, x\not\in \Cut(z_0),
\end{align*}
and  $$\cM^+\left(D^2\left(d_{z_0}^2/2\right)(x)\right)\leq n\Lambda \sH\left(\sqrt\kappa d_{z_0}(x)\right)\leq n\Lambda \sH\left(\al_1\sqrt\kappa R_0\right),\quad\forall  x\in B_{\beta_1 R}(z_0)\bs\Cut(z_0),
$$
from Lemma \ref{lem-hess-dist-sqrd}.
 Following  the proof of \cite[Lemma 4.1]{KKL}, and using Lemma \ref{lem-prop-pucci-op} $(a),$    we can select  positive constants $A,m,l,\al$,  depending only on $\eta, n,\lambda,\Lambda, \sqrt{\kappa}R_0,$ such that  $(b),(c),$ and $(d)$ hold.  For the details, we refer to the proof of   \cite[Lemma 4.1]{KKL} (see also  \cite[Lemma 3.22]{W}).
 \end{proof}

We   approximate the barrier function  $v_\eta$ by a sequence of  smooth functions as Cabr\'{e}'s approach in \cite{Ca} since $v_\eta(x,t)=\vp_\eta\left(\frac{d_{z_0}^2(x)}{R^2},\frac{t}{R^2}\right)$ is not smooth on  $\Cut(z_0).$  
 We note that the cut locus of $z_0$ is  closed and has measure zero.  It is not hard to verify the following lemma, and        refer to \cite[ Lemmas 5.3, 5.4]{Ca} for the elliptic case. 

\begin{lemma}\label{lem-barrier-approx}
 Let $z_0\in M,\, R>0$ and let $\vp:\R^+\times[0,T]\to\R$ be a smooth  function such that %$\p_s\psi(0,t)=0$  and 
$\vp(s,t)$ is nondecreasing with respect to $s$ for any $t\in[0,T]$. Let $v(x,t):=\vp\left({d^2_{z_0}(x)} , t\right)$. Then there exist a smooth function $\psi$ on $M$ satisfying  
$$0\leq\psi\leq1\quad\mbox{on $M,$}\quad\psi\equiv1\quad\mbox{in $B_{\beta_1R}(z_0),$\quad and }\quad \supp\psi\subset B_{\frac{10}{\eta}R}(z_0), $$ and a sequence $\{w_k\}_{k=1}^\infty$ of smooth functions in $M\times[0,T] $  such that
\begin{equation*} 
\left\{
\begin{array}{ll}
 w_k\to \psi v\qquad &\mbox{uniformly in $M\times[0,T]$, }\\ 
 \p_tw_k\to \psi\p_tv\qquad &\mbox{uniformly in $M\times[0,T]$, }\\
 D^2 w_k\leq C  g_x\qquad &\mbox{in $M\times[0,T]$, }\\
 D^2w_k\to D^2v\qquad &\mbox{a.e. in $B_{\beta_1R}(z_0)\times[0,T]$, }
 \end{array}\right. \qquad \qquad
\end{equation*}
where the constant $C>0$ is independent of $k$.
\end{lemma}

The following measure estimate of the  sublevel set  is obtained by 
applying  Lemma \ref{lem-abp-type} to $u+v_\eta$ with $v_\eta,$ constructed in Lemma \ref{lem-barrier} and   translated in time,  with the help of the approximation  lemma above.

\begin{prop} \label{lem-decay-est-1-step}  Assume that 
  $$\Sec\geq -\kappa\quad\mbox{on $M$},\quad\mbox{for $\kappa\geq0, $} $$ and  that   $F$ satisfies 
\eqref{Hypo1} with $F(0)=0.$ 
      Let   $0<\eta<1,$  $0<R\leq R_0,$ and $K_{\alpha_1R,\,\alpha_2R^2}(z_0, 4R^2)\subset K_{R_0}(x_0,t_0)\subset M\times\R.$   
Let $u$ be a smooth function on  $K_{\alpha_1R,\, \alpha_2R^2}(z_0,4R^2)$ such that $$F(D^2u)-\p_tu \leq f\quad\mbox{in}\quad K_{\be_1R,\, \be_2R^2}(z_0,4R^2), $$
$$u\geq 0\quad\mbox{in}\quad K_{\alpha_1R,\,\alpha_2R^2}(z_0, 4R^2)\bs  K_{\beta_1R,\,\beta_2R^2}(z_0,4R^2), $$
and $$\inf_{K_{2R}(z_0,4R^2)}u\leq1. $$
 Then, there exist uniform constants $M_\eta>1, 0<\mu_\eta<1$,  and $0<\e_\eta<1$ such that 
\begin{equation*}%\label{eq-decay-est-1-step}
\frac{\left|\left\{u\leq M_{\eta}\right\}\cap K_{\eta R}(z_0 ,0)\right|}{\left|K_{\be_1R,\, \be_2R^2}(z_0, 4R^2)\right|}\geq \mu_\eta,
\end{equation*}
provided 
\begin{equation*}%\label{cond-f}
\left( \fint_{K_{\be_1R,\,\be_2R^2}(z_0,4R^2)} \left|\be_1^2R^2f^+\right|^{n\theta+1}\right)^{\frac{1}{n\theta+1}}\leq \e_\eta;\quad f^+:=\max(f,0),
%R_0^2\left(\fint_{K_{R_0}(y_0,s_0)} |f^+|^{n\theta+1}\right)^{\frac{1}{n\theta+1}}\leq \e_\eta,%\\\frac{R^2}{|K_{\alpha_1R,\,\alpha_2R^2}(z_o, 4R^2)|^{1/(n+1)}}||f^+||_{L^{n+1}\left(K_{\alpha_1R,\,\alpha_2R^2}(z_o, 4R^2)\right)}\leq \ve_\eta,
\end{equation*}
where  {\color{black}$\theta:=1+\log_2\cosh (4\sqrt\kappa R_0),$}
% $\theta:=1+\log_2\cosh (4\sqrt\kappa R_0),$
 %{\color{red}$\theta:=1+\frac{1}{n}\cosh^{n-1}(4\sqrt{k_0}R_0)$} 
and $ M_\eta>0,\, 0<\mu_\eta, \e_\eta <1 $ depend only on $\eta, n, \lambda,\Lambda$ and $\sqrt{\kappa}R_0.$  
\end{prop}
\begin{proof}
Let $v_\eta$ be the barrier function  as in Lemma \ref{lem-barrier} after translation in time (by $-\eta^2R^2$) and let $\{w_k\}_{k=1}^\infty$ be a sequence of smooth functions approximating $v_\eta$ from  Lemma \ref{lem-barrier-approx}. We notice that $u+v_\eta\geq0 $ in $K_{\alpha_1R, \,\alpha_2R^2}(z_0, 4R^2)\bs K_{\beta_1R, \,\beta_2R^2}(z_0, 4R^2)$ and $\displaystyle\inf_{K_{2R}(z_0,4R^2)} (u+v_{\eta})\leq1$. %for $K_{2R}:=B_{2R}(z_o)\times(0, 4R^2]$.   
 We can apply Lemma  \ref{lem-abp-type} to  $u+w_k$  after a  slight   modification as in the proof of \cite[Lemma 4.3]{KKL},   
 and use the dominated convergence theorem to let $k$ go to $+\infty$ due to  Lemma \ref{lem-barrier-approx}. Thus  we obtain 
\begin{align*}%\label{eq-abp-type}
|B_R(z_0)|\cdot R^2\leq C_1\int_{\{u+v_\eta\leq M_\eta\}\cap K_{\beta_1R, \beta_2R^2}(z_0, 4R^2)}\left[\left\{R^2\left\{\cM^-(D^2u+D^2v_\eta)-\p_t(u+v_\eta)\right\}+ C_2 \right\}^+ \right]^{n+1},
%\\&\leq C_1\int_{E_1\cup E_2}\left[\left\{R^2\left\{\cM^- (D^2u )-\p_tu\right\}+R^2\left\{\cM^+ (D^2v_\eta )-\p_tv_\eta\right\}+ C_2 \right\}^+ \right]^{n+1},
 \end{align*}
where
$C_1:=\sS^{n+1}\left(2\sqrt{\kappa} R_0\right)/(2\lambda)^{n+1},\,\,$ and $ C_2:=12/\eta^2+2(n+1)\Lambda\sH\left( 2\sqrt{\kappa}\,R_0\right).$
Using      Lemma \ref{lem-prop-pucci-op},  \eqref{Hypo1'} and the properties $(c),(d)$ of $v_\eta$ in Lemma \ref{lem-barrier},  we have 
\begin{align*}%\label{eq-abp-type}
|B_R(z_0)|\cdot R^2
%&\leq C_1\int_{\{u+v_\eta\leq M_\eta\}\cap K_{\beta_1R, \beta_2R^2}(z_0, 4R^2)}\left[\left\{R^2\left\{\cM^-(D^2u+D^2v_\eta)-\p_t(u+v_\eta)\right\}+ C_2 \right\}^+ \right]^{n+1},
%\\
&\leq C_1\int_{E_1\cup E_2}\left[\left\{R^2\left\{\cM^- (D^2u )-\p_tu\right\}+R^2\left\{\cM^+ (D^2v_\eta )-\p_tv_\eta\right\}+ C_2 \right\}^+ \right]^{n+1}\\
&\leq C_1\int_{E_1\cup E_2}\left[\left\{R^2\left\{F (D^2u )-\p_tu\right\}+R^2\left\{\cM^+ (D^2v_\eta )-\p_tv_\eta\right\}+ C_2 \right\}^+ \right]^{n+1}\\
&\leq C_1\int_{K_{\be_1R,\,\be_2R^2}(z_0,4R^2)}\left|R^2f^++(C_\eta+C_2)\,\chi_{E_2} \right|^{n+1},
 \end{align*}
where   $E_1:=\{u+v_\eta\leq M_\eta\}\cap \left(K_{\beta_1R,\,\beta_2R^2}(z_0,4R^2)\bs K_{\eta R}(z_0,0)\right) $ and $E_2:=\{u+v_\eta\leq M_\eta\}\cap K_{\eta R}(z_0,0)$.
%Using the properties $(c),(d)$ of $v_\eta$ in Lemma \ref{lem-barrier}, 
Then, it follows that 
 \begin{align*}%\label{eq-abp-type}
\frac{|B_R(z_0)|\cdot R^2}{\left|K_{\be_1R,\,\be_2R^2}(z_0,4R^2)\right| } %&\leq C_1\fint_{K_{\alpha_1R,\,\alpha_2R^2}(z_0,4R^2)}\left|R^2f^++C_3\chi_{E_2} \right|^{n+1}\\
&\leq C_3\fint_{K_{\be_1R,\,\be_2R^2}(z_0,4R^2)}\left|\be_1^2R^2f^++\chi_{E_2} \right|^{n+1}\\
&\leq C_3\left( \fint_{K_{\be_1R,\,\be_2R^2}(z_0,4R^2)} \left|\be_1^2R^2f^+\right|^{n\theta+1}\right)^{\frac{n+1}{n\theta+1}}
+C_3\frac{\left|E_2\right|^\frac{n+1}{n\theta+1}}{{\left|K_{\be_1R,\,\be_2R^2}(z_0,4R^2)\right|^{\frac{n+1}{n\theta+1}} } }
%+C_3\frac{\left|\left\{u+v_\eta\leq M_\eta\right\}\cap K_{\eta R }(z_0,0)\right|^\frac{n+1}{n\theta+1}}{{\left|K_{\alpha_1R,\,\alpha_2R^2}(z_0,4R^2)\right|^{\frac{n+1}{n\theta+1}} } }
 \end{align*}
for {\color{black}$\theta:=1+\log_2\cosh (4\sqrt\kappa R_0)\geq 1,$}  where a uniform constant $C_3>0$  depending only on $\eta, n,\lambda,\Lambda$ and $\sqrt{\kappa}R_0$   may change from line to line. Therefore, Bishop-Gromov's   Theorem \ref{lem-bishop}
implies that \begin{align*}%\label{eq-abp-type}
\frac{\left|E_2\right|^\frac{n+1}{n\theta+1}}{{\left|K_{\be_1R,\,\be_2R^2}(z_0,4R^2)\right|^{\frac{n+1}{n\theta+1}} } }
+\left( \fint_{K_{\be_1R,\,\be_2R^2}(z_0,4R^2)} \left|\be_1^2R^2f^+\right|^{n\theta+1}\right)^{\frac{n+1}{n\theta+1}}
&\geq \frac{1}{C_3} \frac{|B_R(z_0)|\cdot R^2}{\left|K_{\be_1R,\,\be_2R^2}(z_0,4R^2)\right| } \\
&\geq \frac{1}{C_3}\frac{1}{\cD}\left(\frac{1}{\be_1}\right)^{\log_2\cD}\frac{1}{\be_2}=: 2\mu_\eta^\frac{n+1}{n\theta+1}
 \end{align*}
for {\color{black}$\cD:=2^{n}\cosh^{n-1}(2\sqrt\kappa R_0).$}  %and   for a uniform constant $C_3>0$  depending only on $\eta, n,\lambda,\Lambda$ and $\sqrt{\kappa}R_0.$
 By selecting $\e_\eta:=\mu_\eta^{\frac{1}{n\theta+1}},$ we conclude that 
\begin{align*}%\label{eq-abp-type}
\mu_\eta \leq \frac{\left|\left\{u+v_\eta\leq M_\eta\right\}\cap K_{\eta R }(z_0,0)\right| }{\left|K_{\be_1R,\,\be_2R^2}(z_0,4R^2)\right| }
 \leq \frac{\left|\left\{u \leq \tilde M_\eta\right\}\cap K_{\eta R }(z_0,0)\right| }{\left|K_{\be_1R,\,\be_2R^2}(z_0,4R^2)\right| }
 \end{align*}
 for   $\tilde M_\eta:=M_\eta+C_\eta$ depending only on $\eta,n,\lambda,\Lambda$ and $\sqrt{\kappa}R_0$ since $v_\eta\geq -C_\eta$ in $K_{\be_1R,\,\be_2R^2}(z_0,4R^2)$ from Lemma \ref{lem-barrier}.
  \end{proof}
  
  \begin{remark}\label{rmk-Harnack-const}  
 {\rm 
 Let         $n,\lambda, \Lambda,$   and $\eta\in(0,1)$  be given, where  $\eta$   is fixed as a universal  constant in the sequel. Constructing  the  barrier function  in Lemma \ref{lem-barrier},  we have chosen   a  uniform constant  $C_\eta=C_\eta\left(\sqrt{\kappa}R_0\right)$   such that
  $$C_{\eta}\left(\sqrt{\kappa}R_0\right)\gtrsim \exp\left(  c\left(1+\sqrt{\kappa}R_0\right)^2\right), $$ for large $  \sqrt{\kappa}R_0,$
 where      $c>0$ is  a universal constant  and we notice that $\sH(\tau)\leq 1+\tau$ for $\tau\geq0.$  
 Due to this choice of $C_{\eta}\left(\sqrt{\kappa}R_0\right),$ 
    the positive constants $M_\eta\left(\sqrt{\kappa}R_0\right), \mu_\eta\left(\sqrt{\kappa}R_0\right),$  $ \e_\eta\left(\sqrt{\kappa}R_0\right) $  in  
  Proposition \ref{lem-decay-est-1-step}   have been  selected so   that for large $\sqrt{\kappa}R_0>1, $ 
   $$M_\eta\left(\sqrt{\kappa}R_0\right)\gtrsim \exp\left(  c\left(1+\sqrt{\kappa}R_0\right)^2\right),$$   
  $$ \mu_ \eta\left(\sqrt{\kappa}R_0\right) \lesssim \exp\left(-  c\left(1+\sqrt{\kappa}R_0\right)^3\right) $$ and  
   $$ \e_ \eta\left(\sqrt{\kappa}R_0\right) \lesssim \exp\left(-  c\left(1+\sqrt{\kappa}R_0\right)^2\right). $$ 
   Once Proposition \ref{lem-decay-est-1-step} is established,   a priori Harnack estimate follows from  the  same procedure as in \cite{KKL}  using  Bishop and Gromov's volume comparison theorem. %a (locally) uniform volume doubling property. %,  and hence 
%  Following  the  proof given in \cite{KKL},  A
Thus an examination of the procedure  % in \cite{KKL} %convinces us the same argument may be employed, with only minor modifications, to establish the following result. Thus  fulfilling  the procedure     
asserts    
%Thus we can see
 that the  uniform constant 
  $C_H=C_H\left(\sqrt{\kappa}R_0\right)$ in Theorems  \ref{thm-PHI} and \ref{thm-weak-PHI}   grows      faster  than   $  \exp \left(1+ {\kappa}R_0^2\right)$ as $\sqrt{\kappa}R_0\to+\infty$. In fact, $C_H(\sqrt{\kappa}R_0)$ behaves like the exponential composed with a polynomial.  
It is a    rough Harnack estimate  compared to    the results of      \cite{Y1,Y2, BQ}       which  obtained   differential Haranck   estimates  for the heat equation  on a Rimannian manifold $M$ with Ricci curvature bounded by   $-\kappa$ ($\kappa\geq0$)   along the line of Li and Yau \cite{LY}.      A sharp Li-Yau type  Harnack  inequality for  the heat equation on such a manifold   \cite{LX} states that for  any $x_1,x_2\in M$ and $0<t_1<t_2<+\infty$
  $$u(x_1,t_1)\leq u(x_2,t_2)\left(\frac{t_2}{t_1}\right)^{\frac{n}{2}}\exp\left(\frac{d^2(x_2,x_1)}{4(t_2-t_1)}\left(1+\frac{1}{3}\kappa(t_2+t_1)\right)+\frac{n}{4}\kappa(t_2-t_1)\right),$$ 
  from which  a  sharp Haranck   constant    $C_H(\sqrt{\kappa}R_0)$   for the heat equation   has asymptotic growth rate of  $ \exp \left(   1+ {\kappa}R_0^2\right)$ as $\sqrt{\kappa}R_0$ tends to infinity.   
  %   It seems natural  that  the uniform constant $C_H$ increases in an exponential fashion  as $\sqrt{\kappa}R_0\to+\infty$ due to the assumption of a negative   lower bound $-\kappa$ of the curvature of the underlying manifold.  It is an interesting question  to find a sharp  estimate for the constant $C_H.$  
  %,  which  
  %This  seems to be    a very   rough  Harnack estimate. % bound  of the  Harnack constant.  % when compared to  those for  2-dimensional model spaces  in \cite[Proposition 11.4.]{WZ2}.

  %$$$$
%$$\frac{1}{C(\sqrt{\kappa} R)}\in o\left({e^{-{\sqrt{\kappa} R}}}\right)\qad \mbox{as $\sqrt{\kappa}R\to+\infty,$}$$
 %that is ,$C$ grows faster that $e^{\sqrt{\kappa} R}.$

% In the sequel the constant $\eta>0$ is selected uniformly.  $\epsilon,\mu$}
}
 \end{remark}

 \subsection{Parabolic Harnack inequalities for viscosity solutions}
  Now we  prove  Proposition \ref{lem-decay-est-1-step-viscosity}  which   is a
  counterpart of Proposition \ref{lem-decay-est-1-step}  for viscosity solutions,  and a 
   key ingredient in proving Theorems \ref{thm-PHI} and \ref{thm-weak-PHI}. 
  
%  will give  a decay estimate of the distribution function of     continuous viscosity supersolutions    using  the  inf- convolutions. %in Section \ref{sec-supinf-conv}. 

 \begin{prop} \label{lem-decay-est-1-step-viscosity} 
   Assume that 
  $$\Sec\geq -\kappa\quad\mbox{on $M,$}\quad\mbox{for $\kappa\geq0,$} $$
  and  that   $F$ satisfies \eqref{Hypo1} with $F(0)=0.$ %$\mathrm{(H1)}.$     
  Let  $0<\eta<1,$ $0<R\leq R_0,$ and $K_{\alpha_1R,\,\alpha_2R^2}(z_0, 4R^2)\subset K_{R_0}(x_0,t_0)\subset M\times\R.$ % with $0<\al_1R<R_0.$
     For  $f\in C\left( K_{R_0}(x_0,t_0)\right),$ let $u\in C\left(  K_{R_0}(x_0,t_0)\right)$ be a viscosity supersolution of     $$F(D^2u)-\p_tu = f\quad\mbox{in}\,\,\,\,  K_{\alpha_1R,\, \alpha_2R^2}(z_0,4R^2), $$ such that 
$$u\geq 0\quad\mbox{in}\quad K_{\alpha_1R,\,\alpha_2R^2}(z_0, 4R^2)\bs  K_{\beta_1R,\,\beta_2R^2}(z_0,4R^2), $$
and $$\inf_{K_{2R}(z_0,4R^2)}u\leq1. $$
 Then, there exist uniform constants $M_\eta>1, 0<\mu_\eta<1$,  and $0<\e_\eta<1$ such that 
\begin{equation*}%\label{eq-decay-est-1-step}
\frac{\left|\left\{u\leq M_{\eta}\right\}\cap K_{\eta R}(z_0 ,0)\right|}{\left|K_{\alpha_1R,\, \alpha_2R^2}(z_0, 4R^2)\right|}\geq \mu_\eta,
\end{equation*}
provided 
\begin{equation}\label{cond-f}
\left(\fint_{K_{R_0}(x_0,t_0)} \left|R_0^2f^+\right|^{n\theta+1}\right)^{\frac{1}{n\theta+1}}\leq \e_\eta,
%R_0^2\left(\fint_{K_{R_0}(y_0,s_0)} |f^+|^{n\theta+1}\right)^{\frac{1}{n\theta+1}}\leq \e_\eta,%\\\frac{R^2}{|K_{\alpha_1R,\,\alpha_2R^2}(z_o, 4R^2)|^{1/(n+1)}}||f^+||_{L^{n+1}\left(K_{\alpha_1R,\,\alpha_2R^2}(z_o, 4R^2)\right)}\leq \ve_\eta,
\end{equation}
where   
 $\theta:=1+\log_2\cosh (4\sqrt\kappa R_0), $ 
and $ M_\eta>0,\, 0<\mu_\eta, \e_\eta <1 $ depend only on $\eta, n, \lambda,\Lambda$ and $\sqrt{\kappa}R_0.$  \end{prop}
 
   \begin{proof}
 It suffices to prove the proposition for $F=\cM^-$   owing to the uniform ellipticity \eqref{Hypo1'}. %from $\mathrm{(H1)}$ (or  (H1')).      
     % so we assume that $F=\cM^-.$    
 %since    $\al:=\al_2/\al_1^2\leq1 $ in  \eqref{eq-weight-int-para}. 
 %There exists  $R_1>0$ such that $0<\al_1R<R_1<R_0$  and 
 %$$K_{\al_1R,\,\al_2R^2}(z_0,4R^2)\subset K_{R_1}(x_0,t_0)\subset K_{R_0}(x_0,t_0).$$  
Setting $\tilde\al_1:=(\al_1+\be_1)/2,$ and $\tilde \al_2:=(\al_2+\be_2)/2, $
   we define % $\tilde\al_j:=(\al_j+\beta_j)/2 $ for $j=1,2,$ and 
 $$\Omega\times(T_0,T_2]:=K_{\tilde\al_1R,\,\tilde\al_2R^2}(z_0,4R^2),\,\,\,\,\mbox{and }\,\,\,\, H\times(T_1,T_2]:= K_{\be_1R,\,\be_2R^2}(z_0,4R^2).$$   
 We note that $u$ and $f $ belong to $C\left(\overline\Omega\times[T_0,T_2]\right),$ and  we denote by 
 $\omega$ the modulus of continuity of $u$ on $\overline\Omega\times[T_0,T_2],$ which is nondecreasing with $\omega(0+)=0.$

  For   $\ve>0,$ let $u_\ve$ be  the inf-convolution  of $u$ with respect to  $\Omega\times(T_0,T_2]$ as in \eqref{eq-def-inf-conv-p}. 
   According to  Lemma \ref{prop-u-u-e-superj-p},
   there exists $\ve_0>0$ such that if $0<\ve<\ve_0,$ then  $u_\ve$ satisfies 
\begin{equation*}%\label{eq-pucci-u-e-p}
\cM^-(D^2u_\ve)-\p_tu_\ve\leq f_\ve \quad\mbox{ \,a.e. in $\,\,  K_{\be_1R,\,\be_2R^2}(z_0,4R^2),$}
\end{equation*} 
 where  $f_{\ve}$ is defined  as follows: 
  for $(x,t)\in K_{\be_1R,\,\be_2R^2}(z_0,4R^2),$
 $$f_{\ve}(x,t):= \sup_{\overline B_{2\sqrt{m\ve}}(x)\times \left[t-2\sqrt{m\ve},\min\left\{t+2\sqrt{m\ve},T_2\right\}\right]}f+2n\Lambda {\color{black}\kappa\,\omega}\left(2\sqrt{m \ve}\right)%+{\color{blue}\omega_F}\left(2\sqrt{m\ve}\right)
 ;\,\,m:=||u||_{L^{\infty}\left(\overline K_{\tilde\al_1R,\,\tilde\al_2R^2}(z_0,4R^2)\right)},$$
 and we recall that $\cM^-$ is intrinsically uniformly continuous with respect to $x$ with $\omega_{\cM^-}\equiv0.$ 
Using  \eqref{eq-weight-int-para} and \eqref{cond-f},  we have    that
 \begin{equation*}%\label{eq-f-av-vol}
 \begin{array}{ll}
\left(\fint_{K_{\be_1R,\,\be_2R^2}(z_0, 4R^2)} \left|\be_1^2R^2f^+\right|^{n\theta+1}\right)^{\frac{1}{n\theta+1}} &\leq 2\left(\frac{\be_2}{\be_1^2}\right)^{-\frac{1}{n\theta+1}}\left(\fint_{K_{R_0}(x_0,t_0)} |R_0^2f^+|^{n\theta+1}\right)^{\frac{1}{n\theta+1}}\\&\leq 2\left(\frac{\be_2}{\be_1^2}\right)^{-\frac{1}{n\theta+1}}\e_\eta\leq 2\left(\frac{\be_2}{\be_1^2}\right)^{-\frac{1}{n +1}}\e_\eta=:\tilde\e_\eta,
  \end{array}
  \end{equation*} 
and hence  for small $\ve>0,$
\begin{equation}\label{eq-f-av-vol-ue}
 \begin{array}{ll}
 \left(\fint_{K_{\be_1R,\,\be_2R^2}(z_0, 4R^2)} \left|\be_1^2R^2\left\{\cM^-(D^2u_\ve)-\p_tu_\ve\right\}^+\right|^{n\theta+1}\right)^{\frac{1}{n\theta+1}}
 &\leq  \left(\fint_{K_{\be_1R,\,\be_2R^2}(z_0, 4R^2)} \left|\be_1^2R^2f_\ve^+\right|^{n\theta+1}\right)^{\frac{1}{n\theta+1}} \\
 &\leq  2\tilde\e_\eta,
\end{array}
\end{equation}
since $f_\ve$ converges uniformly to $f$ in     $K_{\be_1R,\,\be_2R^2}(z_0,4R^2).$   
For a fixed $\delta>0,$ % and  define  
% $$\tilde u_\ve:=\frac{u_{\ve}+2\delta}{1+3\delta}.$$  
we may assume that  for small $\ve>0, $
 $$u_\ve\geq- \delta  \quad\mbox{in}\quad K_{\tilde\alpha_1R,\,\tilde\alpha_2R^2}(z_0, 4R^2)\bs  K_{\beta_1R,\,\beta_2R^2}(z_0,4R^2),$$ and  $$\displaystyle\inf_{K_{2R}(z_0,4R^2)}  u_\ve\leq 1+\delta$$   since $u_\ve$ converges uniformly to $u$ in $K_{\tilde\al_1R,\,\tilde\al_2R^2}(z_0,4R^2)$ from  Lemma \ref{lem-visc-u-u-e-0-p}. 

Now, we   fix a small $\ve>0.$
According to Lemma \ref{lem-u-e-alexandrov-p}, $(c),$ there is a smooth function $\psi$ on $M\times(-\infty,T_2]$ satisfying    $0\leq\psi\leq1$ on $M\times (-\infty,T_2]$,
$$\psi\equiv1 \quad\mbox{in $\overline K_{\beta_1R,\,\beta_2R^2}(z_0,4R^2),\quad$ and }\quad\supp\psi\subset K_{\tilde\alpha_1R,\,\tilde\alpha_2R^2}(z_0, 4R^2), $$  and  we find a sequence $\{ w_k\}_{k=1}^\infty$ of  smooth functions on $M\times(-\infty,T_2]$ satisfying 
\begin{equation*} 
\left\{
\begin{array}{ll}
 w_k\to  \psi u_\ve\qquad &\mbox{uniformly in $M\times (-\infty,T_2]$  as $k\to+\infty,$}\\ 
 |\D w_k|+|\p_t w_k|\leq C \qquad &\mbox{in $M \times(-\infty,T_2]$, }\\
   \p_tw_k\to \p_tu_\ve\quad&\mbox{a.e. in $K_{\beta_1R,\,\beta_2R^2}(z_0,4R^2)$ as $k\to+\infty,$}\\
     D^2w_k\leq C  g\qquad &\mbox{in $M \times(-\infty,T_2]$, }\\
 D^2 w_k\to D^2u_\ve\quad&\mbox{a.e. in $K_{\beta_1R,\,\beta_2R^2}(z_0,4R^2)$ as $k\to+\infty,$}\\
 \end{array}\right. \qquad \qquad
\end{equation*}
where the constant $C>0$ is independent of $k$. 
For large $k,$ we may assume that
$$ w_k\geq -2\delta\,\,\,\mbox{in $K_{\alpha_1R,\,\alpha_2R^2}(z_0, 4R^2)\bs  K_{\beta_1R,\,\beta_2R^2}(z_0,4R^2)$},\quad\displaystyle\inf_{K_{2R}(z_0,4R^2)}\frac{w_k+2\delta}{1+4\delta}\leq 1,$$
 and 
 $$\left(\fint_{K_{\be_1R,\,\be_2R^2}(z_0, 4R^2)}
  \left|\be_1^2R^2\left\{\cM^-(D^2w_k)-\p_tw_k\right\}^+\right|^{n\theta+1}\right)^{\frac{1}{n\theta+1}}\leq 4\tilde\e_\eta,$$
where we used  the dominated convergence theorem to obtain the last estimate from \eqref{eq-f-av-vol-ue}.   

Selecting $\e_\eta>0$ small enough, we apply Proposition \ref{lem-decay-est-1-step} to $\displaystyle \frac{w_k+2\delta}{1+4\delta} $   (for large $k$) to  obtain  
  \begin{equation*}
  \frac{\left|\left\{w_k+2\delta\leq(1+4\delta) M_{\eta}\right\}\cap K_{\eta R}(z_0 ,0)\right|}{\left|K_{\be_1R,\, \beta_2R^2}(z_0, 4R^2)\right|}\geq \mu_\eta.
  \end{equation*} 
%  where we used the Jensen inequality.
By letting $k\to+\infty, $ we have 
  \begin{equation*}
\frac{\left|\left\{u_\ve+\delta\leq (1+4\delta)M_{\eta}\right\}\cap K_{\eta R}(z_0 ,0)\right|}{\left|K_{\be_1R,\, \be_2R^2}(z_0, 4R^2)\right|}\geq \mu_\eta. 
\end{equation*}  
Since $u_\ve$ converges uniformly to $u$ in $K_{\tilde\alpha_1R,\, \tilde\alpha_2R^2}(z_0, 4R^2),$   %Using Lemmas \ref{lem-visc-u-u-e-0-p},\ref{lem-u-e-alexandrov-p},    
  we let $ \ve\to0$ and  $\delta\to0,$ and use Bishop-Gromov's Theorem \ref{lem-bishop} to deduce that 
  \begin{equation*}
\frac{\left|\left\{u\leq M_{\eta}\right\}\cap K_{\eta R}(z_0 ,0)\right|}{\left|K_{\alpha_1R,\, \alpha_2R^2}(z_0, 4R^2)\right|}\geq\frac{1}{\cD}\left(\frac{\be_1}{\al_1}\right)^{\log_2\cD}\frac{\be_2}{\al_2}\frac{\left|\left\{u\leq M_{\eta}\right\}\cap K_{\eta R}(z_0 ,0)\right|}{\left|K_{\be_1R,\, \be_2R^2}(z_0, 4R^2)\right|}\geq \frac{1}{\cD}\left(\frac{\be_1}{\al_1}\right)^{\log_2\cD}\frac{\be_2}{\al_2}\mu_\eta>0
\end{equation*} for {\color{black}$\cD:=2^{n}\cosh^{n-1}(4\sqrt\kappa R_0),$}   which finishes the proof.
 \end{proof}

 \textbf{Sketch of proof of Theorems  \ref{thm-PHI} and \ref{thm-weak-PHI}  }

  Theorems  \ref{thm-PHI} and \ref{thm-weak-PHI}  follow from   Proposition \ref{lem-decay-est-1-step-viscosity} and 
   a standard covering argument 
using  Bishop and Gromov's Theorem \ref{lem-bishop}.  %,  Proposition \ref{lem-decay-est-1-step-viscosity}, and  Lemma \ref{lem-prop-pucci-op} regarding   Pucci's extremal operators 
 %  (see \cite{KKL,W} for instance).
Indeed,  we first prove a decay estimate for the  distribution function of a viscosity supersolution $u$ to $\cM^-(D^2u)-\p_t u\leq f^+$ in $K_{2R}(x_0,4R^2).$ The main tools of the proof are     Proposition \ref{lem-decay-est-1-step-viscosity}, and  a parabolic version of the Calder\'on-Zygmund decomposition in \cite{Ch} according to   Bishop and Gromov's Theorem \ref{lem-bishop}.   Then, the weak Harnack inequality in Theorem \ref{thm-weak-PHI}   follows.   To complete the proof of Theorem \ref{thm-PHI} for  the viscosity solution $u\in\cS_P^*(f)$, we apply  
Proposition \ref{lem-decay-est-1-step-viscosity}, and  obtain the same decay estimate for  $w:=C_1-C_2u$ (for   $C_1,C_2>0$), which satisfies 
\begin{align*}
\cM^-(D^2w)-\p_tw= -C_2 \left\{\cM^+(D^2u)- \p_tu\right\}\leq C_2|f| %-C_1\left\{F(D^2u)- \p_tu\right\}=-C_1f
\end{align*}
in the viscosity sense. 
For the detailed proofs, we refer to \cite{KKL,W}.

%%%%%%%%%%%%%%%%%%%%%%%%%%%%%%%%%%%%%%%%%%%%%%%%%%%%%%%%%%%%%%%%%%%%%%%%%%%%%%%%%%%%%%%%%%%%%%%%%%%%%%%%%%%%%%%%%%%%%%%%
\section{Elliptic Harnack inequality }\label{sec-EHI}
  Using  the sup and  inf-convolutions, we  will prove  Harnack inequalities of continuous viscosity solutions to elliptic equations from a priori estimates. 
We recall viscosity solutions for uniformly elliptic operators. 

\begin{definition}[Viscosity sub and super-differentials, \cite{AFS}] Let $\Omega\subset M$ be open and 
let $u: \Omega\to\R$ be a lower semi-continuous function. We define the second order subjet of $u$ at $x\in \Omega $ by
\begin{align*}
\cJ^{2,-}u(x):=&\left\{\left(\D\varphi(x),D^2\varphi(x)\right)\in T_xM\times\Sym TM_x: \varphi\in C^2(\Omega), \right.\\
&\left. u-\varphi \,\,\,\mbox{has a local minimum at $x$}\right\}.
\end{align*}
%where $\Sym TM$ stands for the  symmetric 2-tensor at $x.$ 
If $(\zeta,A)\in \cJ^{2,-}u(x),$ $\zeta$ and $A$ are  called a first order subdifferential  and a second order subdifferential  of $u$ at $x,$ respectively.  

Similarly, for a upper semi-continuous function
 $u: \Omega\to \R,$ we define the second order superjet of $u$ at $x\in \Omega$ by
\begin{align*}
\cJ^{2,+}u(x):=&\left\{\left(\D\varphi(x),D^2\varphi(x)\right)\in T_xM\times\Sym TM_x: \varphi\in C^2(\Omega), \right.\\
&\left.u-\varphi \,\,\,\mbox{has a local maximum at $x$}\right\}. 
\end{align*}
\end{definition}
  
We quote the following  local   characterization of  $\cJ^{2,-}u$  from \cite[Proposition 2.2]{AFS}. 
\begin{lemma} Let $u: \Omega\to \R$ be a lower semi-continuous function and $x\in M.$
The following statements are equivalent:
\begin{enumerate}[(a)]
\item $(\zeta,A)\in \cJ^{2,-}u(x).$
\item  $u\left(\exp_x\xi\right)\geq u(x)+\langle\zeta,\xi\rangle+\frac{1}{2}\langle A\xi, \xi\rangle+o(|\xi|^2)\,\,$ as $\,\, T_xM\owns\xi\to0.$
\end{enumerate}
\end{lemma}

\begin{definition}[Viscosity solution] \label{def-visc-sol}
%\begin{enumerate}[(i)]
%\item 
(i) Let $F:M\times \R\times TM\times \Sym TM\to \R$ and let $\Omega\subset M$ be an open set.
We say that a  upper semi-continuous function $u: \Omega\to\R$ is a viscosity subsolution of the equation $F(x,u,\D u, D^2u)=0$ on $\Omega$ if $$F\left(x,u(x),\zeta,A\right)\geq 0$$
for any $x\in\Omega$ and $(\zeta,A)\in \cJ^{2,+}u(x).$  
Similarly,  a  lower semi-continuous function $u: \Omega  \to\R $  is said to be 
a  viscosity supersolution of the equation $ F(x,u , \D u , D^2u)=0$ on $\Omega$ if  
$$F\left(x,u(x),\zeta,A\right)\leq 0$$ 
for any $x\in\Omega$ and $(\zeta,A)\in \cJ^{2,-}u(x).$  
We say $u$ is  a viscosity solution if $u$ is  both a viscosity subsolution and a viscosity supersolution.
%\item 

(ii) Let $\Omega\subset M$  be open, and  
let $0<\lambda\leq \Lambda. $
 We denote by    $
\overline \cS_E\left(\lambda,\Lambda,f\right)$ a class of a viscosity  supersolution    $u\in C(\Omega)$  satisfying
$$\cM^-(D^2u)\leq f\quad\mbox{in $\Omega$} $$  in the viscosity sense. 
Similarly,  a class $\underline\cS_E\left(\lambda,\Lambda,f\right)$ of subsolutions  is defined as the set of     $u\in C(\Omega)$ such that 
$$\cM^+(D^2u)\geq f\quad\mbox{ in $\Omega$} $$ in the viscosity sense. 
We also define $$\cS_E^*\left(\lambda,\Lambda,f\right):=\overline \cS_E\left(\lambda,\Lambda,|f|\right)\cap \underline \cS_E\left(\lambda,\Lambda,-|f|\right).$$ We write shortly $\overline \cS_E(f),\underline \cS_E(f), $and $  \cS_E^*(f)$ for  $\overline \cS_E\left(\lambda,\Lambda,f\right),\underline \cS_E\left(\lambda,\Lambda,f\right),  $ and $\cS_E^*\left(\lambda,\Lambda,f\right),$   
 respectively. 
%\end{enumerate}
\end{definition}

As in the parabolic case, we use  the sup and inf-convolutions   to approximate continuous viscosity solutions.  Let $\Omega\subset M$  be  a bounded open set,
and    $u$ be a continuous function on $\overline \Omega.$    For $\ve>0,$  let $u_\ve$ denote 
  the inf-convolution  of $u$   (with respect to $ \Omega $),   defined as follows: for $x_0\in  \overline\Omega,$
\begin{equation*}%\label{eq-def-inf-conv}
u_{\ve}(x_0 ):=\inf_{y\in  \overline\Omega  } \left\{ u(y) +\frac{1}{2\ve}d^2(y, x_0)\right\}.
 \end{equation*}
If $y_0$ is a point  to realize the above infimum, then we have 
$$u_\ve(x_0)=u(y_0)+\frac{1}{2\ve}d^2(y_0, x_0)\leq  u(y) +\frac{1}{2\ve}d^2(y, x_0)\quad \forall y\in \Omega,$$ which means 
$u$ has a  touching  paraboloid  $-\frac{1}{2\ve}d^2(y, x_0)+u(y_0)+\frac{1}{2\ve}d^2(y_0, x_0)$ at $y_0$ from below.

Now we state  the elliptic analogue of the results in  Section \ref{sec-supinf-conv}   without proof. 

\begin{lemma} \label{lem-visc-u-u-e-0}For $u\in C\left(\overline \Omega \right),$ let $u_\ve$ be  the inf-convolution of $u$ with respect to $  \Omega $ and let $x_0\in\overline\Omega .$%, where $H$ is an open set such that $\overline H\subset \Omega.$  
 \begin{enumerate}[(a)]
 \item If $0<\ve<\ve',$ then  $u_{\ve'}(x_0)\leq u_\ve(x_0) \leq u(x_0).$
 \item  There exists $\displaystyle \, y_0\in \overline \Omega $ such that $u_\ve(x_0)=u(y_0)+\frac{1}{2\ve}d^2(y_0, x_0).$
 \item $\displaystyle d^2(y_0,x_0) \leq 2\ve |u(x_0)-u(y_0)|\leq 4\ve ||u||_{L^\infty(\Omega)}.$
\item $u_\ve \uparrow u$ uniformly in $\overline\Omega.$ 
 \item $u_\ve$ is Lipschitz continuous in $\overline\Omega$:   for $x_0, x_1\in \overline\Omega,$
 $$|u_{\e}(x_0)-u_\ve(x_1)|\leq \frac{3}{2\ve}\diam(\Omega) d(x_0,x_1).$$
 \end{enumerate}
\end{lemma}

\begin{lemma}\label{lem-u-e-alexandrov}
  Assume that $$\Sec
  \geq -\kappa\quad\mbox{on $M,\quad$  for $\kappa\geq0$.}$$%, where $H$ is an open set such that $\overline H\subset \Omega.$ Let $x_0,x_1\in H.$
For $u\in C\left(\overline\Omega\right),$ let $u_\ve$ be  the inf-convolution of $u$ with respect to $ \Omega, $ where $\Omega\subset M$ is a bounded open set.
\begin{enumerate}[(a)]
 \item $u_\ve$ is semi-concave in $\Omega.$  Moreover,  for almost every $x\in\Omega,$ $u_\ve$ is differentiable at $x,$ and there exists  the Hessian  $D^2u_\ve(x)$ (in the sense of Aleksandrov-Bangert's Theorem \ref{thm-AB})  such that
 \begin{equation*}%\label{eq-u-e-taylor-e}
 u_\ve\left(\exp_x \xi \right)=u_\ve(x)+\left\langle\D u_\ve(x),\xi\right\rangle+\frac{1}{2}{\left\langle A(x)\cdot\xi,\xi\right\rangle}+o\left(|\xi|^2\right)
 \end{equation*} 
  as $\xi\in T_xM\to 0.$  
   \item $\displaystyle D^2u_\ve(x)\leq \frac{1}{\ve}{\sqrt{\kappa}\diam(\Omega)\coth\left(\sqrt \kappa \diam(\Omega)\right)}\, g_x\,\,\,$ a.e. in $\Omega.$
 \item  Let $H$  be an  open set such that $  \overline H\subset\Omega.$ 
  Then, there exist a smooth function $\psi$ on $M$ satisfying    $$\mbox{$0\leq\psi\leq1$ on $M$},\,\,\,\,\,\psi\equiv1 \,\,\,\,\mbox{in $\overline H\,\,\,\,\,$ and }\,\,\,\,\supp\psi\subset \Omega, $$ and a sequence $\{w_k\}_{k=1}^\infty$ of smooth functions on $M$  such that
\begin{equation*} 
\left\{
\begin{array}{ll}
 w_k\to \psi u_\ve\qquad &\mbox{uniformly in $M\,\,$  as $k\to+\infty,$}\\ 
 |\D w_k| \leq C \qquad &\mbox{in $M$, }\\
      D^2w_k\leq C  g\qquad &\mbox{in $M$, }\\
 D^2 w_k\to D^2u_\ve\quad&\mbox{a.e. in $H\,\,$ as $k\to+\infty,$}\\
 \end{array}\right. \qquad \qquad
\end{equation*}
where the constant $C>0$ is independent of $k$.
  \end{enumerate}

\end{lemma}

  \begin{prop}\label{lem-visc-u-u-e}  
% {\color{black} Assume that $\Sec  \geq -k\,\, ( k\geq0)$ on $M.$} 
%Under the same assumption as Lemma \ref{lem-u-e-alexandrov-p}, 
Assume that $$\Sec
  \geq -\kappa\qquad\mbox{on $M,\qquad$ for $\kappa\geq0$.}$$
  Let $H\subset M$  be a bounded  open set    such that $  \overline H\subset\Omega.$   Let  $u\in C\left(\overline\Omega\right),$ and let $\omega$ denote  a modulus of continuity of $u$ on $\overline \Omega,$ which is nondecreasing  on $(0,+\infty)$ with $\omega(0+)=0.$  
  For $\ve>0,$ let $u_\ve$ be  the inf-convolution of $u$ with respect to $ \Omega.$  %and let $m:=||u||_{L^{\infty}\left(\overline\Omega\times[T_1,T_2]\right)}.$    
Then, there exists $\ve_0>0$ depending only  on $||u||_{L^{\infty}\left(\overline\Omega\right)} , H,$ and $\Omega,$   such that   if $0<\ve<\ve_0,$  then the following statements  hold: let  $x_0\in \overline H,$  and let $y_0\in\overline\Omega$ satisfy  
$$u_\ve(x_0)=u(y_0)+\frac{1}{2\ve} d^2(y_0,x_0).$$
\begin{enumerate}[(a)]
%\item  $y_0:=\exp_{x_0}(-\ve\zeta)\in \overline \Omega$ and $$u_\ve(x_0,t_0)=u(y_0,s_0)+\frac{1}{2\ve}\left\{d^2(y_0,x_0)+|s_0-t_0|^2\right\}$$for some $s_0\in[t_0-\ve p, T_2]$
\item We have that   $y_0\in\Omega,$ and there is a unique minimizing geodesic joining  $x_0$ to $y_0.$
\item If $(\zeta,A)\in \cJ^{2,-}u_\ve(x_0),$ % for $(x_0,t_0)\in \overline H\times[T_1,T_2]$   and  $(y_0,s_0)\in\overline\Omega\times[T_0,T_2]$ realizes the infimum of the inf- convolution $u_\ve$  at $(x_0,t_0)$, i.e.,
%$$u_\ve(x_0,t_0)=u(y_0,s_0)+\frac{1}{2\ve}\left\{d^2(y_0,x_0)+|s_0-t_0|^2\right\},$$
then  we have 
$$y_0=\exp_{x_0}(-\ve\zeta).$$
\item If $(\zeta,A)\in \cJ^{2,-}u_\ve(x_0),$  then we have %there exists $0<\delta_0<1$ such that
 $$ \left(\, L_{x_0,y_0}\zeta, \,L_{x_0,y_0}\circ A-{\color{black}2\kappa\,\omega\left(2\sqrt{\ve\,||u||_{L^\infty\left(\overline\Omega\right)} }\right)}\,  g_{y_0}\right)\,\in \cJ^{2,-}u(y_0),$$
where $L_{x_0,y_0}$ stands for the parallel transport along the unique minimizing geodesic joining $x_0$ to $y_0=\exp_{x_0}(-\ve\zeta).$ 
 \end{enumerate}
\end{prop}

\begin{lemma}\label{prop-u-u-e-superj}
 Under the same assumption  as Proposition \ref{lem-visc-u-u-e}, we also assume that  $F$ satisfies \eqref{Hypo1}  and \eqref{Hypo2}. 
For ${\color{black}f\in C\left( \Omega \right)},$  let $u\in C\left(\overline\Omega\right)$ be a viscosity supersolution of $$F(D^2u)=f\quad \mbox{in $\Omega.$}$$  
If $0<\ve<\ve_0,$ then  the inf-convolution $u_\ve$ (with respect to $\Omega$) is a viscosity supersolution of 
  $$F(D^2u_\ve) =f_\ve\quad\mbox{on $H$},$$
  where 
  $\ve_0>0$ is the constant as in Proposition \ref{lem-visc-u-u-e},  and 
  $$f_{\ve}(x):= \sup_{\overline B_{2\sqrt{m\ve}}(x)}f+{\color{black}\omega_F}\left(2\sqrt{m\ve}\right)+2n\Lambda {\color{black}\kappa\,\omega}\left(2\sqrt{m \ve}\right);\quad m:=||u||_{L^{\infty}\left(\overline\Omega\right)}.$$  
 Moreover,  we have 
   $$F(D^2u_\ve) \leq f_\ve\quad\mbox{ a.e. in $H.$}$$
\end{lemma}
In particular, Lemma \ref{prop-u-u-e-superj} holds for Pucci's extremal operators according to Remark \ref{rmk-AFS}.

The following proposition is quoted from the proof of  \cite[Proposition 4.1]{WZ},  which is a main ingredient in the proof of a priori Harnack estimate.

\begin{prop}\label{thm-WZ}
Assume that $$\Sec
  \geq -\kappa\quad\mbox{on $M,\quad$ for $ \kappa\geq0,$}$$ and $F$  satisfies \eqref{Hypo1} with $F(0)=0.$    Let $0<R\leq R_0$ and  $u  $ be a smooth function  satisfying  $F(D^2u)\leq f$ in $B_{2R}(x_0)\subset M.$ Then, there exist uniform constants $M_0>0, 0<\mu_0<1,$ and $0<\e_0<1,$ depending only on $n,\lambda,\Lambda$ and $\sqrt{\kappa}R_0,$ such that if
  for any $B_r(z_0)\subset B_{2R}(x_0),$
\begin{equation*}
u\geq0\quad\mbox{in $B_r(z_0),$} \quad
  \inf_{B_{r/2}(z_0)}u\leq 1,
  \end{equation*}
   and  $$\left(\fint_{B_{3r/4}(z_0)}|r^2f^+|^{n\theta}\right)^{\frac{1}{n\theta}} \leq \e_0;\quad f^+:=\max(f,0),$$
then we have
  \begin{equation*}
 \frac{\left|\left\{u\leq M_0\right\}\cap B_{r/8}(z_0)\right|}{|B_{r}(z_0)|}\geq \mu_0,
  \end{equation*}
  %\begin{equation}
  %\sup_{B_R(z_0)}u\leq C\left\{\inf_{B_R(z_0)}u+R^2\left(\fint_{B_2R(z_0)}|f|^{n\theta}\right)^{1/(n\theta)}\right\},
  %\end{equation} 
  where $\theta:=1+\log_2\cosh\left(8\sqrt{\kappa}R_0\right).$% and $C$ is a constant that only depends on $\sqrt{K_0}R,n\lambda,$ and $\Lambda.$
\end{prop}

\begin{lemma}\label{lem-key-EHI}
Assume that $$\Sec
  \geq -\kappa\quad\mbox{on $M,\quad$ for $\kappa\geq0$},$$
  and $F$  satisfies \eqref{Hypo1} with $F(0)=0.$ %and $F$   satisfies \eqref{Hypo1}. 
Let $0<R\leq R_0.$ For  $f\in C\left(B_{2R}(x_0)\right),$    let $u\in C\left(B_{2R}(x_0)\right) $ be  a   viscosity   supersolution of  
  $$F(D^2u)=f\quad\mbox{ in $B_{2R}(x_0).$}$$ Then, there exist uniform constants $M_0>0, 0<\mu_0<1,$ and $0<\e_0<1,$ depending only on $n,\lambda,\Lambda$ and $\sqrt{\kappa}R_0,$ such that if
   for any $B_r(z_0)\subset B_{2R}(x_0),$
\begin{equation*}
u\geq0\quad\mbox{in $B_r(z_0),$}\quad\mbox{and}\quad
  \inf_{B_{r/2}(z_0)}u\leq 1,
  \end{equation*}
then we have
  \begin{equation*}
 \frac{\left|\left\{u\leq M_0\right\}\cap B_{r/8}(z_0)\right|}{|B_{r}(z_0)|}\geq \mu_0,
  \end{equation*}
  %\begin{equation}
  %\sup_{B_R(z_0)}u\leq C\left\{\inf_{B_R(z_0)}u+R^2\left(\fint_{B_2R(z_0)}|f|^{n\theta}\right)^{1/(n\theta)}\right\},
  %\end{equation} 
 provided that $$\left(\fint_{B_{2R}(x_0)}|R^2f^+|^{n\theta}\right)^{\frac{1}{n\theta}} \leq \e_0;\quad f^+:=\max(f,0),$$
 where $\theta:=1+\log_2\cosh\left(8\sqrt{\kappa}R_0\right).$\end{lemma}
\begin{proof}
 It suffices to prove the proposition for $F=\cM^-$ from \eqref{Hypo1'}. % so we assume that $F=\cM^-.$    
Set $$\Omega:= B_{7r/8}(z_0),\quad\mbox{and}\quad H:=B_{3r/4}(z_0).$$ 
 We note that $u$ and $f $ belong to $C\left(\overline\Omega\right),$ and denote by
  $\omega$   the modulus of continuity of $u$ on $\overline\Omega,$ which is nondecreasing with $\omega(0+)=0.$ 
  
  For  small $\ve>0,$ let $u_\ve$ be  the inf-convolution  of $u$ with respect to  $\Omega.$    According to  Lemma \ref{prop-u-u-e-superj},  we   find $\ve_0>0$ such that 
  for $0<\ve<\ve_0,$ $u_\ve$ satisfies 
\begin{equation*}%\label{eq-pucci-u-e}
\cM^-(D^2u_\ve) \leq f_\ve \quad\mbox{ \,a.e. in $\,\,   B_{3r/4}(z_0),$}
\end{equation*}
 where  $f_{\ve}$ is defined  as follows: 
  for $x\in B_{3r/4}(z_0),$
 $$f_{\ve}(x):= \sup_{\overline B_{2\sqrt{m\ve}}(x)  }f+2n\Lambda {\color{black}\kappa\,\omega}\left(2\sqrt{m \ve}\right)%+{\color{blue}\omega_F}\left(2\sqrt{m\ve}\right)
 ;\quad m:=||u||_{L^{\infty}\left(\overline B_{7r/8}(z_0)\right)},$$
 and we recall that  $\cM^-$ is intrinsically uniformly continuous with respect to $x$ with $\omega_{\cM^-}\equiv0.$ 
Using  \eqref{eq-weight-int-ellip},  we have    that
  \begin{equation*}%\label{eq-f-av-vol-e}
  \left(\frac{3r}{4}\right)^2\left(\fint_{B_{3r/4}(z_0)}|f^+|^{n\theta}\right)^{\frac{1}{n\theta}} \leq 2\cdot 4R^2\left(\fint_{B_{2R}(x_0)}|f^+|^{n\theta}\right)^{\frac{1}{n\theta}} \leq 8\e_0;\quad\theta:= 1+\log_2\cosh\left(8\sqrt\kappa R_0\right),
  \end{equation*}
  and hence for small $\ve>0,$
%From \eqref{eq-f-av-vol-e} and  \eqref{eq-pucci-u-e}, we have 
  \begin{equation}\label{eq-f-av-vol-e}
\left(\frac{3r}{4}\right)^2\left(\fint_{B_{3r/4}(z_0)}\left|\left\{\cM^-(D^2u_\ve)\right\}^+\right|^{n\theta}\right)^{\frac{1}{n\theta}} \leq \left(\frac{3r}{4}\right)^2\left(\fint_{B_{3r/4}(z_0)}|f_\ve^+|^{n\theta}\right)^{\frac{1}{n\theta}} \leq 9\e_0 
\end{equation}
since $f_\ve$ converges uniformly to $f$ in     $B_{3r/4}(z_0).$ 
For a fixed $\delta>0,$  we may assume that  for small $\ve>0, $
 $$ u_\ve\geq-\delta  \quad\mbox{in}\quad B_{7r/8}(z_0),\quad\mbox{and}\quad\displaystyle\inf_{B_{r/2}(z_0)}  u_\ve\leq 1+\delta $$
since $u_\ve$ converges uniformly to $u$ in $B_{7r/8}(z_0)$ from  Lemma \ref{lem-visc-u-u-e-0}.

Now, we  fix a small $\ve>0.$
According to Lemma \ref{lem-u-e-alexandrov}, $(c),$ there is a smooth function $\psi$ such that $0\leq \psi\leq1$ on $M,$
$$\psi\equiv 1\quad\mbox{in $\overline B_{3r/4}(z_0),$}\quad\mbox{and}\quad {\supp\psi\subset B_{7r/8}(z_0)},$$ 
and  we approximate $\psi u_\ve$ by smooth functions $w_k$ on $M $ satisfying 
\begin{equation*} 
\left\{
\begin{array}{ll}
 w_k\to  \psi u_\ve\qquad &\mbox{uniformly in $M$  as $k\to+\infty,$}\\ 
 |\D w_k|\leq C \qquad &\mbox{in $M  $, }\\
     D^2 w_k\leq C g\qquad &\mbox{in $M $, }\\
 D^2 w_k \to D^2u_\ve \qquad&\mbox{a.e. in $B_{3r/4}(z_0)$ as $k\to+\infty,$}\\
 \end{array}\right. \qquad \qquad
\end{equation*}
where the constant $C>0$ is independent of $k$. 
For large $k,$ we may assume that
$$w_k\geq-2\delta\,\,\,\mbox{in $B_{r}(z_0) $}\quad\mbox{and}\quad\displaystyle\inf_{B_{r/2}(z_0)}w_k\leq 1+2\delta,$$ 
 and 
$$\left(\frac{3r}{4}\right)^2\left(\fint_{B_{3r/4}(z_0)}\left|\left\{\cM^-(D^2w_k)\right\}^+\right|^{n\theta}\right)^{\frac{1}{n\theta}} \leq   10\e_0,$$
 where we used  the dominated convergence theorem to obtain the last estimate from \eqref{eq-f-av-vol-e}. 
By selecting $\e_0>0$ small enough,  we apply Theorem \ref{thm-WZ} to $\displaystyle \frac{w_k+2\delta}{1+4\delta} $   (for large $k$) to    obtain that 
  \begin{equation*}
\frac{\left|\left\{w_k+2\delta\leq (1+4\delta)M_0\right\}\cap B_{r/8}(z_0)\right|}{|B_{r}(z_0)|}\geq \mu_0.
  \end{equation*} 
%  where we used the Jensen inequality.
By letting $k\to+\infty, $ we have 
  \begin{equation*}
\frac{\left|\left\{u_\ve+\delta\leq (1+4\delta)M_0\right\}\cap B_{r/8}(z_0)\right|}{|B_{r}(z_0)|}\geq \mu_0.
  \end{equation*} 
Since $u_\ve$ converges uniformly to $u$ in $B_{7r/8}(z_0),$   %Using Lemmas \ref{lem-visc-u-u-e-0-p},\ref{lem-u-e-alexandrov-p},    
  we let $ \ve\to0$ and  $\delta\to0$  to conclude that 
  \begin{equation*}
\frac{\left|\left\{u \leq  M_0\right\}\cap B_{r/8}(z_0)\right|}{|B_{r}(z_0)|}\geq \mu_0
  \end{equation*}  
   which finishes the proof.
\end{proof}

 \textbf{Proof of Theorems  \ref{thm-EHI} and \ref{thm-weak-EHI}  }
 Theorems  \ref{thm-EHI} and \ref{thm-weak-EHI}   follow from  Lemma \ref{lem-key-EHI} and 
   a standard covering argument 
employing  Bishop and Gromov's Theorem \ref{lem-bishop}    (see \cite{Ca} for instance).
%\textbf{ Proof of Theorem  \ref{thm-EHI}} % and \ref{thm-weak-EHI}   } 
%Theorem  \ref{thm-EHI} %and \ref{thm-weak-EHI} follows from Proposition \ref{lem-key-EHI}  via a standard covering argument using Lemma \ref{lem-prop-pucci-op}. We refer to \cite{Ca}.
\qed

  {\bf Acknowledgement} 
  Ki-Ahm Lee was supported by Basic Science Research Pro- gram through the National Research Foundation of Korea(NRF) grant funded by the Korea government(MEST)(2010-0001985).
Ki-Ahm Lee also hold a joint appointment with the Research Institute of Mathematics of Seoul National University. 
 %%%%%%%%%%%%%%%%%%%%%%%%%%%%%%%%%%%%%%%%%%%%%%%

%------------------------------------------------------------------------------%

\end{document}